\documentclass[leqno,11pt]{amsart}
\pdfoutput1

\usepackage{amssymb,color,hyperref,mathrsfs,stmaryrd,relsize}
\usepackage{amsmath}

\usepackage[usenames,dvipsnames]{xcolor}

\usepackage[curve,matrix,arrow]{xy}

\numberwithin{equation}{section}
\newtheorem{theorem}{Theorem}[section]
\newtheorem{lemma}[theorem]{Lemma}

\newtheorem{prop}[theorem]{Proposition}
\newtheorem{cor}[theorem]{Corollary}
\newtheorem{Th}{Theorem}

\newtheorem{Def}{Definition}

\newtheorem{property}[theorem]{}
\newtheorem{ex}[theorem]{Example}

\newtheorem{corollary}[theorem]{Corollary}

\theoremstyle{definition}
\newtheorem{example}[theorem]{Example}

\newtheorem{remark}[theorem]{Remark}

\newtheorem{notation}[theorem]{Notation}
\newtheorem{definition}[theorem]{Definition}

\newcommand{\bN}{\mathbb{N}}

\newcommand{\F}{\mathcal{F}}
\newcommand{\E}{\mathcal{E}}
\newcommand{\G}{\mathcal{G}}
\newcommand{\C}{\mathcal{C}}
\renewcommand{\L}{\mathcal{L}}

\newcommand{\tDelta}{\tilde{\Delta}}

\newcommand{\tN}{\widetilde{\mathcal{N}}}
\newcommand{\tE}{\widetilde{\mathcal{E}}}
\newcommand{\tT}{\widetilde{T}}
\newcommand{\tS}{\widetilde{S}}
\newcommand{\tF}{\widetilde{\mathcal{F}}}

\newcommand{\M}{\mathcal{M}}
\newcommand{\N}{\mathcal{N}}
\newcommand{\K}{\mathcal{K}}
\newcommand{\D}{\mathbf{D}}
\newcommand{\X}{\mathcal{X}}
\newcommand{\Y}{\mathcal{Y}}
\renewcommand{\H}{\mathcal{H}}
\newcommand{\wL}{\widetilde{\mathcal{L}}}

\newcommand{\Hom}{\operatorname{Hom}}
\newcommand{\Aut}{\operatorname{Aut}}
\newcommand{\Inn}{\operatorname{Inn}}
\newcommand{\Syl}{\operatorname{Syl}}

\newcommand{\m}{\mathcal}
\newcommand{\ov}{\overline}

\newcommand{\id}{\operatorname{id}}
\newcommand{\Ac}{\operatorname{A}^\circ}
\newcommand{\One}{\operatorname{\mathbf{1}}}
\newcommand{\W}{\mathbf{W}}
\newcommand{\fN}{\mathfrak{N}}
\newcommand{\hyp}{\mathfrak{hyp}}
\newcommand{\fC}{\mathfrak{C}}
\newcommand{\subn}{\unlhd\!\unlhd\;}
\newcommand{\Comp}{\operatorname{Comp}}

\def \<{\langle }
\def \>{\rangle }

\renewcommand{\phi}{\varphi}
\title[Fusion systems and localities -- a dictionary]{Fusion systems and localities -- a dictionary}

\author[A.~Chermak]{Andrew Chermak}
\address{Kansas State University, Manhattan Kansas}
\email{chermak@math.ksu.edu}

\author[E.~Henke]{Ellen Henke}
\address{Institut f{\"u}r Algebra, Fakult{\"a}t Mathematik, Technische Universit{\"a}t Dresden, 01062 Dresden, Germany}
\email{ellen.henke@tu-dresden.de}
\thanks{The second author was partially supported by EPSRC grant no EP/R010048/1. The second author also would like to thank the Isaac Newton Institute for Mathematical Sciences, Cambridge, for support und hospitality during the programme Groups, representations and applications, where some of the work on this paper was undertaken and supported by EPSRC grant no EP/R014604/1.}

\begin{document}

\begin{abstract}
Linking systems were introduced to provide algebraic models for $p$-completed classifying spaces of fusion systems. Every linking system over a saturated fusion system $\F$ corresponds to a group-like structure called a locality. Given such a locality $\L$, we prove that there is a one-to-one correspondence between the partial normal subgroups of $\L$ and the normal subsystems of the fusion system $\F$. This is then used to obtain a kind of dictionary, which makes it possible to translate between various concepts in localities and corresponding concepts in fusion systems. As a byproduct, we obtain new proofs of many known theorems about fusion systems and also some new results. For example, we show in this paper that, in any saturated fusion system, there is a sensible notion of a product of normal subsystems. 
\end{abstract}

\maketitle

%\textbf{Keywords:} Fusion systems; localities; transporter systems.

%\smallskip

%\textbf{Subject classification:} 20N99, 20D20

%\bigskip

\section{Introduction}

This paper is intended to provide a bridge between the two theories referenced in the title, and 
to illustrate the usefulness of that bridge by proving theorems about fusion systems which have 
previously been out of reach. For example, it is known (see \cite{Henke:2015a}) that the product of two 
partial normal subgroups of a (finite) locality $\L$ is itself a partial normal subgroup of $\L$, 
and we will be able to translate that result into a corresponding theorem on fusion systems. 
Thus, if $\E_1$ and $\E_2$ are normal 
subsystems of a saturated fusion system $\F$ over a finite $p$-group, then there is a product system 
$\E_1\E_2$ which is normal in $\F$ - a result which has hitherto been known to hold only in 
some very special cases. We also obtain new proofs of several known results about fusion systems. Further applications of the one-to-one correspondences shown in this paper are given in \cite{Henke:NK,Henke:NEX}.

\smallskip

The fusion systems (and their related structures) to be considered here will always be 
defined with respect to a finite $p$-group for a fixed prime $p$.   
The theory of such fusion systems has 
undergone a rapid development, not least because of a series of papers by Michael Aschbacher which lay 
the foundation for a new treatment of the Classification of the Finite Simple Groups. 
Part of our motivation is to enrich the conceptual basis for Aschbacher's ongoing program, and to 
point the way toward possible simplifications in it. 

\smallskip

Localities were introduced in \cite{Chermak:2013}, for the purpose of proving that to every 
saturated fusion system $\F$ there corresponds a unique ``classifying space" which, when viewed 
combinatorially (in analogy with the simplicial set underlying the classifying space of a finite 
group) is a ``linking system" $\L$ as defined in \cite{BLO2}. The point of view taken in \cite{Chermak:2013} was 
to view $\L$ differently, as a finite ``partial group" having a ``Sylow $p$-subgroup" $S$, and 
where there is a set $\Delta$ of subgroups of $S$ which encodes 
the set $\D$ of $n$-tuples of elements of $\L$ which can be multiplied together under 
the product $\Pi$ associated with $\L$. In particular, for each $P\in\Delta$, 
the normalizer $N_{\L}(P)$ is a subgroup of $\L$. The fusion system $\F$ is recoverable from 
$\L$ as the category $\F_S(\L)$ whose objects are the subgroups of $S$ and whose morphisms are 
compositions of conjugation maps via elements of $\L$. These notions and others that are basic to fusion systems and localities will be reviewed, or 
references given for them, in Sections \ref{S:NormalSubsystems} and \ref{LocalitiesSection}, below.

\begin{Def}
 \begin{itemize}
  \item A finite group $G$ is \emph{of characteristic $p$} if $C_G(O_p(G))\leq O_p(G)$, where $O_p(G)$ denotes the largest normal $p$-subgroup of $G$.
 \item A locality $(\L,\Delta,S)$ over a fusion system $\F$ is called \emph{proper} if $\F$ is saturated, $\F^{cr}\subseteq\Delta$  and $N_\L(P)$ is of characteristic $p$ for every $P\in\Delta$. 
 \end{itemize}
\end{Def}

We say that $(\L,\Delta,S)$ is a locality \emph{over} $\F$ if $\F=\F_S(\L)$. Proper localities are called \emph{linking localities} in \cite{Henke:2016,Henke:2015,Henke:2020,Henke:Regular} to emphasize the connection with their origins in \cite{BLO2}.

\smallskip

Let $\L$ be a proper locality over $\F=\F_S(\L)$ and let $\Delta_0$ be the set of subgroups $P$ 
of $S$ such that $P$ contains a member of $\F^{cr}$. Then $\Delta_0$ defines the structure of a proper 
locality $\L_0$, having the same fusion system $\F$, and $\Delta_0$ is the smallest possible choice for a 
set $\Delta$ of objects which supports such a structure. There are various other choices which it will be useful to single out. In particular, there is a \emph{largest} possible choice 
for $\Delta$, consisting of the set $\F^s$ of subgroups of $S$ which are \emph{$\F$-subcentric} (cf. 
Definition \ref{D:Subcentric}), and there are the intermediate collections $\F^q$ of \emph{$\F$-quasicentric} 
subgroups of $S$ (first introduced in \cite{BCGLO1}), and $\F^c$ of \emph{$\F$-centric} subgroups of $S$. 

\smallskip

One further possible collection $\Delta$ of objects, denoted $\delta(\F)$, will play a major role in 
this paper. In order to describe it, and in order to state some of our main theorems, we need to consider the 
notions of \emph{partial normal subgroup} $\N\unlhd\L$, and of normal subsystem $\E\unlhd\F$. 

\smallskip

The notion of normal subsystem of a saturated fusion system was introduced by Aschbacher \cite{Aschbacher:2008}, and 
it led to a ``local theory" of saturated fusion systems $\F$ which in many respects mirrors the 
$p$-local theory of finite groups. Thus, in Aschbacher's theory (see especially \cite{Aschbacher:2011} or \cite[Chapter~2]{Aschbacher/Kessar/Oliver:2011}) 
one has the notion of the generalized Fitting subsystem $F^*(\F)$, the quasisimple \emph{components} of 
$\F$, and the \emph{layer} $E(\F)$ consisting of the ``product" of the ``components" of $\F$. As will be seen 
in Lemma~\ref{L:deltaF}, one way to define $\delta(\F)$ is as the set of all subgroups $P$ of $S$ such that $P$ 
contains a member of $F^*(\F)^s$. An equivalent definition of $\delta(\F)$ will be given much earlier, 
in Definition~\ref{D:Regular}, after the notion of $F^*(\L)$ has been reviewed.

\smallskip

A partial normal subgroup $\N\unlhd\L$ of the locality $\L$ (or, more generally, of any partial 
group $\L$) is, first of all a partial subgroup $\N$ of $\L$. This means that $\N$ is non-empty, 
that $\N$ 
is closed under the inversion map on $\L$, and that whenever $w$ is an $n$-tuple of elements of $\N$ 
such that $w$ lies in the domain $\D$ for the product $\Pi$ with which $\L$ is 
(partially) equipped, we then have $\Pi(w)\in\N$. In order for the partial subgroup $\N$ of $\L$ 
to be normal in $\L$ we require further that 
\[x^g:= \Pi(g^{-1},x,g)\in\N\] 
for all triples $(g^{-1},x,g)\in\D$ such that $x\in\N$. 

\smallskip

The key result of this paper is that there is a perfect correspondence between the two notions: 
``normal subsystem of a saturated fusion system" and ``partial normal subgroup of a proper locality". 
For $\N\unlhd\L$ and for $T=S\cap\N$, define $\F_T(\N)$ to be the fusion system on 
$T$ which is generated by conjugation maps via elements of $\N$.

\begin{Th}\label{main}
Suppose $(\L,\Delta,S)$ is a proper locality over $\F$, let $\fN(\L)$ be the set of partial normal subgroups of $\L$, and let $\fN(\F)$ be the set of normal subsystems of $\F$. Then $\fN(\F)$ is a poset under the normality relation $\unlhd$. Moreover, there exists a bijection 
\[\Psi=\Psi_\L\colon\fN(\L)\rightarrow\fN(\F)\]
such that the following hold:
\begin{itemize}
\item [(a)] For every $\N\unlhd\L$, $\Psi(\N)$ is a fusion system over $S\cap\N$ and the smallest normal subsystem of $\F$ containing $\F_{S\cap \N}(\N)$.
\item [(b)] If $F^*(\F)^{cr}\subseteq\Delta$ (which is in particular the case if $\delta(\F)\subseteq\Delta$ or $\F^q\subseteq\Delta$), then \[\Psi(\N)=\F_{S\cap\N}(\N).\] 
\item [(c)] The map $\Psi$ is an isomorphism of posets from $(\fN(\L),\subseteq)$ to $(\fN(\F),\unlhd)$. 
\end{itemize}
\end{Th}

In part (c) of the theorem above, it is important that we consider $\fN(\F)$ as a poset under $\unlhd$ rather than under inclusion, since the map $\Psi^{-1}$ is not necessarily inclusion-preserving as we show in Example~\ref{E:1}. By means of Theorem~\ref{main} it is possible to produce a sort of ``bilingual dictionary", translating between the 
local theory of fusion systems and the theory of localities via the mapping $\Psi$. As a consequence, one obtains proofs of theorems on fusion systems. We will describe this now in more detail.

\smallskip

In general, if $\L$ is a locality (not necessarily proper) and $\N_1$ and $\N_2$ are partial normal subgroups, then it is a triviality that $\N_1\cap\N_2\unlhd\L$. One also has 
the non-trivial fact \cite{Henke:2015a} that 
\[\N_1\N_2:=\{\Pi(x,y)\colon x\in\N_1,\;y\in\N_2,\;(x,y)\in\D\}\unlhd\L.\]
These constructions, along with Theorem~\ref{main}, lead to the following two theorems. Part (a) of Theorem~\ref{T:Intersections} was first proved by Aschbacher \cite[Theorem~1]{Aschbacher:2011}.

\begin{Th}\label{T:Intersections}
Let $\F$ be a saturated fusion system over $S$. Suppose that $\E_1$ and $\E_2$ are normal subsystems of $\F$ over $T_1$ and $T_2$ respectively. Then there exists a subsystem $\E_1\wedge\E_2$ such that the following hold:
\begin{itemize}
 \item [(a)] The subsystem $\E_1\wedge\E_2$ is a normal subsystem of $\F$ over $T_1\cap T_2$ contained in $\E_1\cap\E_2$. Moreover, it is the largest normal subsystem of $\F$ which is normal in $\E_1$ and $\E_2$.
 \item [(b)] Every normal subsystem of $\F$ which is normal in $\E_1$ and $\E_2$ is also normal in $\E_1\wedge\E_2$.  
 \item [(c)] Suppose $(\L,\Delta,S)$ is a proper locality over $\F$ and $\Psi=\Psi_\L$ is the map from Theorem~\ref{main}. If $\N_1$ and $\N_2$ are partial normal subgroups of $\L$ with $\E_i=\Psi(\N_i)$ for each $i=1,2$, then 
\[\Psi(\N_1\cap\N_2)=\E_1\wedge\E_2.\]
\end{itemize}
\end{Th}

Since the intersection of sets (and thus of partial normal subgroups) is an associative operation, there is indeed an associative notion of a ``normal intersection'' of an arbitrary number of normal subsystems (cf. Theorem~\ref{T:IntersectionsI}). 

\begin{Th}\label{ProductCorollary}
Let $\F$ be a saturated fusion system over $S$ and let $\m{E}_i$ be a normal subsystem of $\F$ over $T_i$ for $i=1,2$. Set $T:=T_1T_2$. Then there exists a subsystem $\E_1\E_2$ of $\F$ such that the following hold:
\begin{itemize}
\item [(a)] $\E_1\E_2$ is the (unique) smallest normal subsystem of $\F$ over $T$ containing $\E_1$ and $\E_2$ as a normal subsystems.
\item [(b)] If $\tE\subn\F$ is a subnormal subsystem such that $\E_1$ and $\E_2$ are subnormal in $\tE$ then $\E_1$, $\E_2$, and $\E_1\E_2$ are normal subsystems of $\tE$.
\item [(c)] $\E_1\E_2$ is generated by $\E_1T$ and $\E_2T$.
\item [(d)] Suppose $(\L,\Delta,S)$ is a proper locality over $\F$. If $\Psi=\Psi_\L$ is the map from Theorem~\ref{main} and  $\N_i:=\Psi^{-1}(\E_i)$ for $i=1,2$, then  
\[\Psi(\N_1\N_2)=\E_1\E_2.\] 
\end{itemize}
\end{Th}

Since the partial product in a locality $\L$ satisfies a natural associativity condition, it can indeed be shown that there is an associative notion of a product of any number of partial normal subgroups of $\L$, leading to products of any number of normal subsystems of $\F$ (cf. Theorem~\ref{T:ProductCorollaryDetails}). A construction of a normal subsystem $\E_1\E_2$ was given before by Aschbacher \cite[Theorem~3]{Aschbacher:2008} in the special case that $T_1$ and $T_2$ are commuting subgroups.

\smallskip

We used in Theorem~\ref{ProductCorollary}(c) that, if $\F$ is a saturated fusion system over $S$ then (see Definition~\ref{D:Product}) one has the notion of a product subsystem $\E R$ where $\E$ is a normal subsystem of $\F$ and where 
$R$ is a subgroup of $S$. On the other hand, if $\N\unlhd\L$ is a partial normal subgroup of a  
locality $(\L,\Delta,S)$ over $\F$, and $R\leq S$ is a subgroup of $S$, then the product $\N R$ is a partial subgroup of $\L$. As we state in the next theorem the two concepts are closely related. A proper locality $(\L,\Delta,S)$ over $\F$ is called \emph{regular} if $\Delta=\delta(\F)$.

\begin{Th}\label{T:mainProductWithPSubgroup}
 Let $(\L,\Delta,S)$ be a proper locality over $\F$, and let $\N\unlhd\L$ be a partial normal subgroup. Set $T:=S\cap\N$ and $\E=\F_T(\N)$, and let $R\leq S$ be a subgroup of $S$. Then the following hold:
\begin{itemize}
 \item [(a)] If $(\L,\Delta,S)$ is regular, then $(\N R,\delta(\E R),TR)$ is a regular locality over $\E R$.
 \item [(b)] If $\delta(\F)\subseteq\Delta$, then $\F_{T R}(\N R)=\E R$. Moreover, $(\N S,\Delta,S)$ and $(\N S,\Delta\cup\delta(\E S),S)$ are proper localities over $\E S$.  
 \item [(c)] If $\F^q\subseteq\Delta$, then $\F_{T R}(\N R)=\E R$ whenever $C_S(\N)\leq R$. Moreover, $(\N S,\Delta,S)$ is a proper locality over $\E S$.  
\end{itemize}
\end{Th}

In order for the next part of our dictionary to be intelligible and meaningful some preliminary discussion may be helpful. First of all, since a locality $\L$ is something like a finite group, there are definitions of certain 
partial normal subgroups associated with $\L$ that are suggested by the corresponding definitions in 
group theory. For example, if $G$ is a finite group and $p$ is a prime then $O^p(G)$ is the intersection 
of the normal subgroups $N\unlhd G$ such that $G/N$ is a $p$-group - and one finds then that 
$G/O^p(G)$ is a $p$-group. The same definition (and the same conclusion) apply to proper localities and 
partial normal subgroups, yielding the notion of $O^p(\L)$ and similarly $O^{p^\prime}(\L)$ (cf. Subsection~\ref{SS:ResiduesLocalities}). The 
definition of $F^*(\L)$ (reviewed here in Definition~\ref{D:FstarELLocalities}) is based on the fact that if $G$ is a finite group 
then $F^*(G)$ may be characterized as the smallest normal subgroup $N\unlhd G$ such that $N$ contains 
every normal nilpotent subgroup of $G$, and such that $C_G(N)\leq N$. Similar formulations apply to the 
subgroup $E(G)$ and to $E(\L)$. In the case of \emph{regular} localities, it turns out that partial subnormal subgroups form again regular localities, which leads to a natural notion of components of regular localities. Indeed, if $\L$ is regular, then $E(\L)$ can be described as the product of components of $\L$.

\smallskip

In the definition of $F^*(\L)$ it is used that, for every partial normal subgroup $\N$ of $\L$, there is a largest partial normal subgroup $\N^\perp$ of $\L$ which ``commutes'' with $\N$ and thus should be thought of as a replacement for the centralizer of $\N$ (see Definition~\ref{D:Commute}). If $\L$ is regular, then $\N^\perp$ equals the actual centralizer $C_\L(\N)$ which is the set of all $g\in\L$ such that $(g^{-1},x,g)\in\D$ and $\Pi(g^{-1},x,g)=x$ for all $x\in\N$. 

\smallskip

The definitions of $O^p(\F)$, $F^*(\F)$, ad lib., for a saturated fusion system $\F$ are also motivated by their group-theoretic analogs, but are  
(notoriously) much more difficult to formulate (cf. Subsections~\ref{SS:Fusionppowerindex}, \ref{SS:CFE}, \ref{SS:ResiduesFusionLoc} and Definition~\ref{D:ComponentsEFFstarF}). It is because of this that we believe our 
dictionary to be meaningful.

\begin{Th}\label{T:mainDetails}
 Let $(\L,\Delta,S)$ be a proper locality over $\F$ and let $\Psi$ be the map from Theorem~\ref{main}. Then the following hold.
\begin{itemize}
\item [(a)] For every subgroup $R\leq S$, we have $R\unlhd\L$ if and only if $R\unlhd\F$. In particular, $\Psi(R)=R$ and $\Psi(O_p(\L))=O_p(\L)=O_p(\F)$.  
\item [(b)] $\Psi(O^p(\L))=O^p(\F)$ and $\Psi(O^{p^\prime}(\L))=O^{p^\prime}(\F)$.
\item [(c)] If $\N\unlhd\L$ and $\E:=\Psi(\N)$, then $\Psi(\N^\perp)=C_\F(\E)$.  
\item [(d)] $\Psi(F^*(\L))=F^*(\F)$ and $\Psi(E(\L))=E(\F)$. 
\end{itemize}
\end{Th}

\begin{Th}\label{mainSubnormalGeneralizedFitting}
Suppose $(\L,\Delta,S)$ is a regular locality over $\F$. Write $\mathfrak{S}(\L)$ for the set of partial subnormal subgroups of $\L$ and $\mathfrak{S}(\F)$ for the set of subnormal subsystems of $\F$. Then the map 
\[\hat{\Psi}=\hat{\Psi}_\L\colon \mathfrak{S}(\L)\rightarrow\mathfrak{S}(\F),\H\mapsto\F_{S\cap\H}(\H)\]
is well-defined and an isomorphism of posets from $(\mathfrak{S}(\L),\subseteq)$ to $(\mathfrak{S}(\F),\subn)$. Moreover, $\hat{\Psi}$ restricts to the map $\Psi=\Psi_\L$ from Theorem~\ref{main}, and it induces a bijection from the set of components of $\L$ to the set of components of $\F$.
\end{Th}

The results summarized in Theorems~\ref{T:mainDetails} and \ref{mainSubnormalGeneralizedFitting} will be employed in Subsection~\ref{SS:EFNewProofs} to give new proofs of known results concerning components, the layer, and the generalized Fitting subsystem of a saturated fusion system.

\smallskip

The final result to be mentioned in this introduction concerns the ``subnormal intersections" 
developed by Aschbacher \cite[7.2.2]{Aschbacher:2011} for saturated fusion systems, and for which our result 
may be seen to provide a more precise formulation (as well as simplified proofs). It is 
essentially a refinement of Theorem~\ref{T:Intersections} above, except that in the final part we need to restrict attention to regular localities.

\begin{Th}\label{T:IntersectionsSubnormal}
 Let $\F$ be a saturated fusion system over $S$. Suppose that $\E_1$ and $\E_2$ are subnormal subsystems of $\F$ over $T_1$ and $T_2$ respectively. Then there exists a subsystem $\E_1\wedge\E_2$ of $\F$ such that the following hold:
\begin{itemize}
 \item [(a)] $\E_1\wedge\E_2$ is a subnormal subsystem of $\F$ over $T_1\cap T_2$ contained in $\E_1\cap\E_2$. Moreover, $\E_1\wedge\E_2$ is the largest subnormal subsystem of $\F$ which is subnormal in $\E_1$ and $\E_2$. 
 \item [(b)] Every subnormal subsystem of $\F$ which is subnormal in $\E_1$ and $\E_2$ is also subnormal in $\E_1\wedge\E_2$. 
 \item [(c)] If $\E_1\unlhd\F$, then $\E_1\wedge\E_2\unlhd \E_2$. 
 \item [(d)] Suppose $(\L,\Delta,S)$ is a regular locality over $\F$ and $\hat{\Psi}_\L$ is the map from Theorem~\ref{mainSubnormalGeneralizedFitting}. If $\H_i\subn\L$ with $\hat{\Psi}_\L(\H_i)=\E_i$ for $i=1,2$, then $\hat{\Psi}_\L(\H_1\cap\H_2)=\E_1\wedge\E_2$.
\end{itemize}
\end{Th}

Since $\hat{\Psi}_\L$ restricts to $\Psi_\L$, it follows from Theorem~\ref{T:Intersections}(c) and Theorem~\ref{T:IntersectionsSubnormal}(d) (or from the constructions of $\E_1\wedge\E_2$ in the proofs) that, for any two normal subsystems $\E_1$ and $\E_2$ of $\F$, the subsystem $\E_1\wedge\E_2$ from Theorem~\ref{T:Intersections} coincides with the subsystem $\E_1\wedge\E_2$ in Theorem~\ref{T:IntersectionsSubnormal}. 

\smallskip

We wish to add some remarks concerning the demands that we make on the reader concerning background material. For fusion fusion systems, beyond the basics to be found (for example) in \cite{Aschbacher/Kessar/Oliver:2011} (see Remark~\ref{R:FusionSystemsBackground}) we need the following concepts, and results concerning them, for a saturated fusion system $\F$ over $S$. 
\begin{itemize}
 \item the centralizer in $\F$ of a normal subsystem of $\F$, which is again a normal subsystem;  
 \item the product of two normal subsystems $\F_i$ over $S_i$ with $\F_i\subseteq C_\F(S_{3-i})$ for $i=1,2$, which is also a normal subsystem;
 \item the product of a normal subsystem of $\F$ with a subgroup of $S$, which can be shown to be a saturated subsystem of $\F$.
\end{itemize}
All of the above were introduced by Aschbacher in \cite{Aschbacher:2011}. For our proofs we rely on the alternative treatment of these concepts in \cite{Henke:2013} and \cite{Henke:2018}.

\smallskip

Background material on localities is provided in Section~\ref{LocalitiesSection}. Above all, the proof of Theorem~\ref{main} relies 
on the existence and uniqueness of a linking system \cite{Oliver:2013,Glauberman/Lynd} (or, via the Appendix to \cite{Chermak:2013} a 
proper locality) associated with a given saturated fusion system. Our proofs use moreover \cite{Chermak:2015,Henke:2015a,Henke:2016}, parts of \cite{Henke:2015}, and the theory developed in \cite{ChermakII,ChermakIII} in the form stated (and proved) in \cite{Henke:2020,Henke:Regular}.

\section{Definitions and preliminary results on fusion systems}\label{S:NormalSubsystems}

\textbf{Throughout this section let $\F$ be a fusion system over $S$.}

\smallskip

The reader is referred to \cite[Part~I]{Aschbacher/Kessar/Oliver:2011} for an introduction to the theory of fusion systems. We will adapt the notation and terminology from this reference with the exception that we will write group homomorphisms on the right hand side of the argument similarly as in \cite[Part~II]{Aschbacher/Kessar/Oliver:2011}. Furthermore, it will be convenient to write $\F^f$ for the set of fully $\F$-normalized subgroups of $S$. We will refer to the Sylow axiom and the extension axiom for saturated fusion systems as stated in \cite[Proposition~I.2.5]{Aschbacher/Kessar/Oliver:2011} most of the time without further reference.

\smallskip

We will build in particular on the definition of $\F$-invariant, weakly normal and normal subsystem of $\F$ as  introduced in \cite[Definition~I.6.1]{Aschbacher/Kessar/Oliver:2011}. Moreover, we will refer to the \emph{Frattini condition} and to the \emph{extension condition}  as introduced in that definition. 

\smallskip

If $R\leq S$ and $\C$ is a collection of injective group homomorphism between subgroups of $R$, then we write $\<\C\>_R$ for the smallest fusion system over $R$ containing all the elements of $\C$ as morphisms.

\begin{remark}\label{R:FusionSystemsBackground}
Since the main results of this paper are used to reprove many properties of fusion systems, let us summarize precisely which background references on fusion systems we rely on. We need Sections~I.1-I.7, Section~II.5, Theorem~II.7.5 and Theorem~III.5.10 in the book by Aschbacher, Kessar and Oliver \cite{Aschbacher/Kessar/Oliver:2011}. Complete proofs of most of the results we use are given in the book. The only exceptions are Theorems~I.3.10 and I.7.4 in \cite{Aschbacher/Kessar/Oliver:2011}; for the proofs the reader is referred to \cite[Theorem~2.2]{BCGLO1} and \cite[Theorem~4.3]{BCGLO2}. Our proofs rely moreover on \cite{Henke:2013} and \cite{Henke:2018}.
\end{remark}

\subsection{Morphisms of fusion systems}

We continue to assume that $\F$ is a fusion system over $S$. In addition, throughout this subsection, let $\tF$ be a fusion system over $\tS$. 

\begin{definition}
A group homomorphism $\alpha\colon S\longrightarrow \tS$ is said to \textit{induce a morphism} from $\F$ to $\tF$ if, for each $\phi\in\Hom_\F(P,Q)$, there exists $\psi\in\Hom_{\tF}(P\alpha,Q\alpha)$ such that $(\alpha|_P)\psi=\phi(\alpha|_Q)$. 
\end{definition}

If $\alpha$ induces a morphism from $\F$ to $\tF$, then for any $\phi\in\Hom_\F(P,Q)$, the map $\psi\in\Hom_{\tF}(P\alpha,Q\alpha)$ as in the above definition is uniquely determined. So if $\alpha$ induces a morphism from $\F$ to $\tF$, then $\alpha$ induces a map 
\[\alpha_{P,Q}\colon\Hom_\F(P,Q)\longrightarrow \Hom_{\tF}(P\alpha,Q\alpha).\] 
Together with the map $P\mapsto P\alpha$ from the set of objects of $\F$ to the set of objects of $\tF$ this gives a functor $\alpha^*$ from $\F$ to $\tF$. Moreover, $\alpha$ together with the maps $\alpha_{P,Q}$ ($P,Q\leq S$) is a morphism of fusion systems in the sense of \cite[Definition~II.2.2]{Aschbacher/Kessar/Oliver:2011}.

\begin{definition}
Suppose $\alpha\colon S\longrightarrow \tS$ induces a morphism from $\F$ to $\tF$. Then $\alpha$ is said to \emph{induce an epimorphism} from $\F$ to $\tF$ if $\alpha$ is surjective as a map $S\longrightarrow \tS$ and, for all $P,Q\leq S$ with $\ker(\alpha)\leq P\cap Q$, the map $\alpha_{P,Q}$ is surjective. If $\alpha$ is in addition injective, then we say that $\alpha$ \emph{induces an isomorphism} from $\F$ to $\tF$. Write $\Aut(\F)$ for the set of automorphisms of $S$ which induce an isomorphism from $\F$ to $\F$. Accordingly, $\alpha\in\Aut(S)$ is said to induce an automorphism of $\F$ if $\alpha\in\Aut(\F)$. 
\end{definition}

If $\alpha\colon S\longrightarrow \tS$ is an isomorphism of groups, then $\alpha$ induces an isomorphism from $\F$ to $\tF$ if and only if, for all $P,Q\leq S$ and every group homomorphism $\phi\colon P\longrightarrow Q$, the following equivalence holds:
\[\phi\in\Hom_\F(P,Q)\Longleftrightarrow \alpha^{-1}\phi\alpha\in\Hom_{\tF}(P\alpha,Q\alpha).\]
If so, then $\alpha_{P,Q}$ as above is given by $\phi\alpha_{P,Q}=\alpha^{-1}\phi\alpha$ for all $\phi\in\Hom_\F(P,Q)$.

\begin{definition}\label{D:MorphismFusionSystemImage}
Let $\E$ be a subsystem of $\F$ over $T\leq S$ and suppose $\alpha\colon S\longrightarrow \tS$ induces a morphism from $\F$ to $\tF$. Then $\E^\alpha$ denotes the subsystem of $\tF$ over $T\alpha$ with $\Hom_{\E^\alpha}(P\alpha,Q\alpha)=\Hom_\E(P,Q)\alpha_{P,Q}$ for all $P,Q\leq T$ with $\ker(\alpha)\cap T\leq P\cap Q$.
\end{definition}

In the situation above notice that $\alpha|_T$ induces an epimorphism from $\E$ to $\E^\alpha$.

\subsection{Centric, quasicentric and centric-radical subgroups}

\begin{definition}\label{D:Fclosed}
Let $\Delta$ be a set of subgroups of $S$.
\begin{itemize}
 \item The set $\Delta$ is said to be \emph{closed under $\F$-conjugacy} if every $\F$-conjugate of an element of $\Delta$ is an element of $\Delta$. 
 \item $\Delta$ is said to be \emph{$\F$-closed} if $\Delta$ is overgroup-closed in $S$ and closed under $\F$-conjugacy. 
\end{itemize}
\end{definition}

There are several $\F$-closed collections of subgroups of $S$ which play a particularly important role, for example the set $\F^c$ of $\F$-centric subgroups, the set $\F^q$ of $\F$-quasicentric subgroups and, if $\F$ is saturated, the set $\F^s$ of $\F$-subcentric subgroups. We will introduce $\F$-subcentric subgroups in Definition~\ref{D:Subcentric}. A subgroup $P\leq S$ is called \emph{$\F$-quasicentric} if $C_\F(Q)=\F_{C_S(Q)}(C_S(Q))$ for some fully centralized $\F$-conjugate $Q$ of $P$. The reader is referred to  \cite[Definition~I.3.1]{Aschbacher/Kessar/Oliver:2011} for the definition of $\F$-centric and $\F$-centric-radical subgroups of $S$. The set $\F^{cr}$ of $\F$-centric-radical subgroups is closed under $\F$-conjugacy, but not overgroup-closed in $S$. We have the following lemma.

\begin{lemma}\label{L:EcFinvariant}
If $\E$ is an $\F$-invariant subsystem of $\F$, then $\E^c$ and $\E^{cr}$ are closed under $\F$-conjugacy.
\end{lemma}

\begin{proof}
By definition of $\F$-invariant subsystems, every element of $\Aut_\F(T)$ induces an automorphism of $\E$. Therefore,  $\Aut_\F(T)$ acts on $\E^c$ and $\E^{cr}$. Moreover, we have $P^\E\subseteq\E^c$ for all $P\in\E^c$ and similarly $P^\E\subseteq\E^{cr}$ for all $P\in\E^{cr}$. Hence, the assertion follows from the Frattini condition for $\F$-invariant subsystems.
\end{proof}

\begin{lemma}\label{L:QuasicentricApply}
If $\F$ is saturated and $P\in\F^q$, then every saturated subsystem of $C_\F(P)$ is the fusion system of a $p$-group.
\end{lemma}

\begin{proof}
Suppose $\F$ is saturated and $P$ is $\F$-quasicentric. Then we may pick $Q\in P^\F$ fully centralized such that $C_\F(Q)=\F_{C_S(Q)}(C_S(Q))$. By the extension axiom, there exists $\phi\in\Hom_\F(PC_S(P),S)$ such that $P\phi=Q$. Suppose $\G$ is a saturated subsystem of $C_\F(P)$ over $R\leq C_S(P)$. For every $U\leq R$, we have $\Aut_\G(U)\cong \phi^{-1}\Aut_\G(U)\phi\leq \Aut_{C_\F(Q)}(U\phi)$. Hence, $\Aut_\G(U)$ is a $p$-group and, if $U\in\G^f$, then the Sylow axiom implies $\Aut_\G(U)=\Aut_R(U)$ as $\G$ is saturated. It follows now from Alperin's Fusion Theorem \cite[Theorem~I.3.6]{Aschbacher/Kessar/Oliver:2011} that $\G=\F_R(R)$. 
\end{proof}

\subsection{Groups of characteristic $p$ and models}

Recall from the introduction that a group $G$ is \emph{of characteristic $p$} if $C_G(O_p(G))\leq O_p(G)$. A group of characteristic $p$ is called a \emph{model} for $\F$ if $S$ is a Sylow $p$-subgroup of $G$ and $\F=\F_S(G)$. Notice that $\F$ is saturated by \cite[Theorem~2.3]{Aschbacher/Kessar/Oliver:2011} if there is a model for $\F$.

\smallskip

The fusion system $\F$ is called \emph{constrained} if $\F$ is saturated and $C_S(O_p(\F))\leq O_p(\F)$. If $\F$ is constrained, then the existence of a model for $\F$ follows from a Theorem of Broto, Castellana, Grodal, Levi and Oliver; see \cite[Proposition~C]{BCGLO1} or \cite[Theorem~III.5.10]{Aschbacher/Kessar/Oliver:2011}.

\begin{lemma}\label{L:MSCharp}
If $G$ is a group of characteristic $p$, then the following hold:
\begin{itemize}
 \item [(a)] Every subnormal subgroup of $G$ is of characteristic $p$.
 \item [(b)] For every $p$-subgroup $P$ of $G$, the normalizer $N_G(P)$ is of characteristic $p$. 
 \item [(c)] If $S\in\Syl_p(G)$ and $U\in \F_S(G)^c$, then $C_G(U)\leq U$.
\end{itemize}
\end{lemma}

\begin{proof}
For parts (a) and (b) see \cite[Lemma~1.2(a),(c)]{MS:2012b}. For the proof of (c), we may replace $U$ by any $G$-conjugate of $U$ in $S$. So we may assume without loss of generality that
\[N_S(U)\in\Syl_p(N_G(U)).\]
Then $C_S(U)\in Syl_p(C_G(U))$ as $C_G(U)\unlhd N_G(U)$. Since $U\in\F_S(G)^c$, we have $C_S(U)\leq U$. Setting $Q:=O_p(N_G(U))$, it follows  $[Q,C_G(U)]\leq C_Q(U)\leq C_S(U)\leq U$. Hence, $[Q,O^p(C_G(U))]=[Q,O^p(C_G(U)),O^p(C_G(U))]=1$ (cf. \cite[Lemma~A.2]{Aschbacher/Kessar/Oliver:2011}). By (b), $N_G(U)$ is of characteristic $p$. Hence, we have $O^p(C_G(U))\leq C_{N_G(U)}(Q)\leq Q$ and thus $O^p(C_G(U))=1$. This implies $C_G(U)=C_S(U)\leq U$, i.e. (c) holds.  
\end{proof}

\begin{lemma}\label{ModelLemma}
If $G$ is a model for $\F$, then for every subnormal subsystem $\E$ of $\F$, there exists a unique subnormal subgroup $H$ of $G$ with $\F_{S\cap H}(H)=\E$.
\end{lemma}

\begin{proof}
If $\E$ is a normal subsystem of $\F$, then by \cite[Theorem~II.7.5]{Aschbacher/Kessar/Oliver:2011}, there exists a unique normal subgroup $N$ of $G$ with $\E=\F_{S\cap N}(N)$. Moreover, by Lemma~\ref{L:MSCharp}(a), every normal subgroup $N$ of $G$ is of characteristic $p$ and thus a model for $\F_{S\cap N}(N)$. Thus, it follows by induction on the subnormal length that, for every subnormal subsystem $\E$ of $\F$, there exists a subnormal subgroup $H$ of $G$ with $\E=\F_{S\cap H}(H)$. It remains to show that $H$ is unique. 

\smallskip

Let $K$ and $H$ be subnormal subgroups of $G$ with $\E=\F_{S\cap H}(H)=\F_{S\cap K}(K)$. Then in particular, $T:=S\cap H=S\cap K$. By Lemma~\ref{L:MSCharp}(a), $H$ and $K$ are models for $\E$. Hence, by \cite[Theorem~III.5.10]{Aschbacher/Kessar/Oliver:2011}, $Q:=O_p(\E)$ is normal in $H$ and $K$, which implies $Q=O_p(H)=O_p(K)$. So $H$ and $K$ are also subnormal subgroups of $N_G(Q)$. Moreover, $N_G(Q)$ is of characteristic $p$ by Lemma~\ref{L:MSCharp}(b). If $S_0$ is a Sylow $p$-subgroup of $N_G(Q)$ containing $T$, then $T=S_0\cap H=S_0\cap K$ and $\E=\F_T(K)=\F_T(H)$ is subnormal in $\F_{S_0}(N_G(Q))$ by \cite[Proposition~I.6.2]{Aschbacher/Kessar/Oliver:2011}. So replacing $G$ by $N_G(Q)$, $S$ by $S_0$ and $\F$ by $\F_{S_0}(N_G(Q))$, we may and will assume from now on that $Q$ is normal in $G$. Choose a subnormal series
\[H=H_0\unlhd H_1\unlhd\dots\unlhd H_n=G.\]
Notice that, for all $i=1,\dots,n$, we have 
\[[O_p(G)\cap H_i,H]\leq [O_p(G)\cap H_i,H_{i-1}]\leq O_p(G)\cap H_{i-1}.\]
We will use the well-known property of coprime action stated e.g. in \cite[Lemma~A.2]{Aschbacher/Kessar/Oliver:2011}. As $O_p(G)=O_p(G)\cap H_n$ and $O_p(G)\cap H_0=O_p(G)\cap H\leq O_p(H)=Q$, this property yields that every $p^\prime$-element of $H$ acts trivially on $O_p(G)/Q$ and thus $[O_p(G),O^p(H)]\leq Q$. Similarly it follows that $[O_p(G),O^p(K)]\leq Q$. 

\smallskip

If $h\in H$ is a $p^\prime$-element, then $\F_T(H)=\F_T(K)$ implies the existence of $k\in K$ with $c_h|_Q=c_k|_Q$, i.e. with $hk^{-1}\in C_G(Q)$. We may choose $k$ to be a $p^\prime$-element as well. As $[O_p(G),h]\leq [O_p(G),O^p(H)]\leq Q$ and $[O_p(G),k]\leq [O_p(G),O^p(K)]\leq Q$, we have $hk^{-1}\leq C_G(O_p(G)/Q)\cap C_G(Q)$. Using again the property of coprime action stated in \cite[Lemma~A.2]{Aschbacher/Kessar/Oliver:2011} together with the fact that $G$ has characteristic $p$, we can conclude that $C_G(O_p(G)/Q)\cap C_G(Q)$ is a normal $p$-subgroup of $G$, as every $p^\prime$-element of $C_G(O_p(G)/Q)\cap C_G(Q)$ is contained in $C_G(O_p(G))\leq O_p(G)$ and thus trivial. Hence, $hk^{-1}\in C_G(O_p(G)/Q)\cap C_G(Q)\leq O_p(G)$. This shows $O^p(H)\leq O^p(K)O_p(G)$. Using a symmetric argument one concludes that $O^p(H)O_p(G)=O^p(K)O_p(G)$. As \[[O_p(G),O^p(H)]=[O_p(G),O^p(H),O^p(H)]\leq [Q,O^p(H)]\leq O^p(H),\]
it follows that $O^p(H)\unlhd O^p(H)O_p(G)$ and thus $O^p(H)=O^p(O^p(H)O_p(G))$. Similarly, we have $O^p(K)=O^p(O^p(K)O_p(G))$ and thus $O^p(H)=O^p(K)$. This implies $H=O^p(H)T=O^p(K)T=K$ as required.  
\end{proof}

\subsection{Normal subsystems of $p$-local subsystems}\label{SS:Normal}

In this subsection we show how normal subsystems of fusion systems lead to normal subsystems of $p$-local subsystems. Most lemmas we prove here are similar or identical to results by Aschbacher \cite{Aschbacher:2008}. However, since we aim to revisit Aschbacher's theory in later chapters, we have  chosen to give direct proofs.

\smallskip

\textbf{Throughout this subsection $\F$ is assumed to be saturated.}

\smallskip

From the results we prove in this subsection, only Lemma~\ref{LocalNormalSubsystems} is cited directly in the proof of Theorem~\ref{main}. The other results we state are used to prove Lemma~\ref{LocalNormalSubsystems} and also Lemma~\ref{L:EcsubsetFq} in Subsection~\ref{SS:CFE}. Lemma~\ref{L:LocalNormalConstrainedSubsystems}(a) and Lemma~\ref{LocalNormalSubsystems}  are moreover used in our background references \cite{Henke:2013,Henke:2018} (cf. Remark~\ref{R:ERProof} and Remark~\ref{R:CFEBackground}).

\begin{lemma}\label{L:WeaklyNormalLocal}
Let $\E$ be a weakly normal subsystem of $\F$ over $T\leq S$. Given $U\leq T$ with $U\in\F^f$, the following properties hold:
\begin{itemize}
\item [(a)] We have $U\in\E^f$. Moreover, $N_\E(U)$ is a weakly normal subsystem of $N_\F(U)$.
\item [(b)] If $U\in\E^c$, then for every $P$ with $U\leq P\leq N_T(U)$, we have $[P,C_S(U)]\leq U$ and $C_S(U)\leq N_S(P)$.
\item [(c)] If $U\in\E^c$, then $N_\F(UC_S(U))=N_{N_\F(U)}(UC_S(U))$ is constrained. Moreover, $N_\E(U)$ is a weakly normal subsystem of $N_\F(UC_S(U))$.
\end{itemize}
\end{lemma}

\begin{proof}
We will use throughout that, by \cite[Theorem~I.5.5]{Aschbacher/Kessar/Oliver:2011}, normalizers of fully $\F$-normalized subgroups of $S$ are saturated. 

\smallskip

\textbf{(a)} As $U\in\F^f$, it is a particular consequence of the just stated fact that the normalizer $N_\F(U)$ is saturated. By \cite[Lemma~2.3]{Henke:2018}, $U\in\E^f$ and so $N_\E(U)$ is saturated. Using the equivalent characterization of $\F$-invariant subsystems given in \cite[Proposition~I.6.4(d)]{Aschbacher/Kessar/Oliver:2011}, it is straightforward to see that $N_\E(U)$ is $N_\F(U)$-invariant. So (a) holds. 

\smallskip

\textbf{(b)} Suppose now that $U\in \E^c$. If $U\leq P\leq N_T(U)$, then
\[[P,C_S(U)]\leq [N_T(U),C_S(U)]\leq C_T(U)\leq U\leq P\]
and so (b) holds.

\smallskip

\textbf{(c)} As before let $U\in\E^c$. Observe first that $UC_S(U)\cap T=UC_T(U)=U$ as $U\in\E^c$. So the fact that $T$ is strongly closed implies that \[N_\F(UC_S(U))=N_{N_\F(U)}(UC_S(U)).\]
As $UC_S(U)$ is self-centralizing in $S$ and weakly closed in $N_\F(U)$, it follows that $N_{N_\F(U)}(UC_S(U))$ is saturated and constrained. If $N_\E(U)\subseteq N_\F(UC_S(U))$, then using that $N_\E(U)$ is $N_\F(U)$-invariant and appealing again to \cite[Proposition~I.6.4]{Aschbacher/Kessar/Oliver:2011}, one checks easily that $N_\E(U)$ is $N_\F(UC_S(U))$-invariant and thus weakly normal in $N_\F(UC_S(U))$. So it is sufficient to argue that $N_\E(U)\subseteq N_\F(UC_S(U))$.

\smallskip

Let $P\in N_\E(U)^{cr}$ and $\phi\in\Aut_{N_\E(U)}(P)$. As $N_\E(U)$ is saturated, by Alperin's fusion theorem \cite[Theorem~I.3.6]{Aschbacher/Kessar/Oliver:2011}, we only need to show that $\phi$ is a morphism in $N_\F(UC_S(U))$. Using \cite[Proposition~I.4.5]{Aschbacher/Kessar/Oliver:2011}, one observes that $U\leq P$. Clearly $P\leq N_T(U)$. So (b) gives $C_S(U)\leq N_S(P)$ and $[P,C_S(U)]\leq U$. Hence, $\Aut_{C_S(U)}(P)$ lies in the centralizer in $\Aut_{N_\F(U)}(P)$ of $P/U$ and $U$. This centralizer is a normal $p$-subgroup of $\Aut_{N_\F(U)}(P)$ (cf. e.g. \cite[Lemma~A.2]{Aschbacher/Kessar/Oliver:2011}). Hence it follows 
\[\Aut_{C_S(U)}(P)\leq O_p(\Aut_{N_\F(U)}(P)).\]
As $N_\F(U)$ is saturated, by \cite[Lemma~II.3.1]{Aschbacher/Kessar/Oliver:2011}, there exists $\alpha\in\Hom_{N_\F(U)}(N_{N_S(U)}(P),N_S(U))$ such that $P\alpha$ is fully $N_\F(U)$-normalized. Fixing such $\alpha$, it follows from the Sylow axiom that $\Aut_{N_S(U)}(P\alpha)$ is a Sylow $p$-subgroup of $\Aut_{N_\F(U)}(P\alpha)$. So $\chi:=\phi\alpha\in\Hom_{N_\F(U)}(P,P\alpha)$ and 
\[\chi^{-1}\Aut_{C_S(U)}(P)\chi\leq \chi^{-1}O_p(\Aut_{N_\F(U)}(P))\chi=O_p(\Aut_{N_\F(U)}(P\alpha))\leq \Aut_{N_S(U)}(P\alpha).\]
Since the extension axiom (\cite[Proposition~I.2.5]{Aschbacher/Kessar/Oliver:2011}) holds in $N_\F(U)$, it follows that $\chi$ extends to 
\[\hat{\chi}\in \Hom_{N_\F(U)}(PC_S(U),N_S(U)).\]
Notice that $C_S(U)\hat{\chi}=C_S(U)=C_S(U)\alpha$ and $P\hat{\chi}=P\chi=P\alpha$. Hence, $\hat{\chi}\alpha^{-1}\in\Aut_{N_\F(U)}(PC_S(U))$ is well-defined, extends $\phi=\chi\alpha^{-1}$ and acts on $U$ and $C_S(U)$. So $\phi$ is a morphism in $N_\F(UC_S(U))$ as required.    
\end{proof}

\begin{lemma}\label{L:LocalNormalTechnicalHelp}
Let $\E$ be a normal subsystem of $\F$ over $T$, let $U\in\E^c$ and $U\leq P\leq N_T(U)$. Suppose furthermore that there exists $Q\in P^\F$ fully $\F$-normalized such that 
\[N_\E(Q)\unlhd N_\F(QC_S(Q)).\]
Then every $\phi\in\Aut_{N_\E(U)}(P)$ extends to $\hat{\phi}\in\Aut_\F(PC_S(U))$ with $[C_S(U),\hat{\phi}]\leq Z(U)$.   
\end{lemma}

\begin{proof}
Let $\phi\in\Aut_{N_\E(U)}(P)$ and $\alpha\in\Hom_\F(N_S(P),S)$ with $P\alpha=Q$. Notice that $\alpha$ exists by \cite[Lemma~II.3.1]{Aschbacher/Kessar/Oliver:2011}. Moreover, Lemma~\ref{L:EcFinvariant} implies $U\alpha\in\E^c$. 

\smallskip

Assume first that $\psi:=\alpha^{-1}\phi\alpha\in\Aut_{N_\E(U\alpha)}(Q)$ extends to $\hat{\psi}\in\Aut_\F(QC_S(U\alpha))$ with 
\[[C_S(U\alpha),\hat{\psi}]\leq Z(U\alpha).\]
By Lemma~\ref{L:WeaklyNormalLocal}(b), we have $C_S(U)\leq N_S(P)$ and so $C_S(U)\alpha\leq C_S(U\alpha)$. Observe also that $U\alpha\leq Q$. In particular, $\hat{\psi}$ normalizes $Q(C_S(U)\alpha)=(PC_S(U))\alpha$. Now
\[\hat{\phi}:=\alpha(\hat{\psi}|_{(PC_S(U))\alpha})\alpha^{-1}\in\Aut_\F(PC_S(U))\]
extends $\phi=\alpha\psi\alpha^{-1}$ and $[C_S(U),\hat{\phi}]\leq Z(U)$. So the assertion holds in this case. Thus, replacing $(P,U)$ by $(Q,U\alpha)$, we can and will assume from now on that
\[P\in\F^f\mbox{ and }N_\E(P)\unlhd N_\F(PC_S(P)).\]
As $U\in\E^c$, we have $P\in\E^c$. So by Lemma~\ref{L:WeaklyNormalLocal}(c), $N_\F(PC_S(P))=N_{N_\F(P)}(PC_S(P))$ is a constrained fusion system. Hence, by Theorem~III.5.10 and Theorem~II.7.5 in \cite{Aschbacher/Kessar/Oliver:2011}, there exists a model $G$ of $N_\F(PC_S(P))$ and a normal subgroup $N$ of $G$ such that $N_T(P)\in\Syl_p(N)$ and $\F_{N_T(P)}(N)=N_\E(P)$. As $U\in\E^c$, we have in particular $U\in N_\E(P)^c$. By Lemma~\ref{L:MSCharp}(a), the group $N$ is of characteristic $p$. So by part (c) of the same lemma, we have $C_N(U)=Z(U)$. Notice that $\phi$ is a morphism in $N_\E(P)$ normalizing $U$. Hence, $\phi=c_n|_P$ for some $n\in N_N(U)$. As already noted above, we have $C_S(U)\leq N_S(P)$ and thus $C_S(U)\leq N_S(PC_S(P))\leq G$. Hence
\[[C_S(U),n]\leq [C_S(U),N_N(U)]\leq C_N(U)=Z(U)\leq P.\]
By \cite[Theorem~2.1(b)]{Henke:2015}, we have $P\unlhd G$ as $P\unlhd N_\F(PC_S(P))$. So $n$ normalizes $PC_S(U)$ and thus $\hat{\phi}:=c_n|_{PC_S(U)}\in\Aut_\F(PC_S(U))$ is well-defined. Observe that $\hat{\phi}|_P=c_n|_P=\phi$ and $[C_S(U),\hat{\phi}]\leq Z(U)$. This proves the assertion.
\end{proof}

Part (a) of the next lemma could also be obtained as a consequence of \cite[Theorem~2]{Aschbacher:2008}.

\begin{lemma}\label{L:LocalNormalConstrainedSubsystems}
Suppose $\E$ is a normal subsystem of $\F$ over $T\leq S$ and $U\in\E^c$. Then the following hold:
\begin{itemize}
 \item [(a)] If $U\in\F^f$, then the subsystem $N_\F(UC_S(U))$ is constrained and $N_\E(U)$ is a normal subsystem of $N_\F(UC_S(U))$.
 \item [(b)] Every $\phi\in\Aut_\E(U)$ extends to $\hat{\phi}\in\Aut_\F(UC_S(U))$ with $[C_S(U),\hat{\phi}]\leq Z(U)$.
\end{itemize}
\end{lemma}

\begin{proof}
For every fully $\F$-normalized $\F$-conjugate $Q$ of $U$, we we have $Q\in\E^c$ by Lemma~\ref{L:EcFinvariant}. Hence, if (a) holds, then for any such $Q$, we know that $N_\E(Q)$ is normal in $N_\F(QC_S(Q))$ and (b) follows from Lemma~\ref{L:LocalNormalTechnicalHelp} applied with $P=U$. Therefore, it is sufficient to prove (a).  

\smallskip

By Lemma~\ref{L:WeaklyNormalLocal}(c), $N_\F(UC_S(U))$ is a constrained fusion system over $N_S(U)$ and $N_\E(U)$ is a weakly normal subsystem of $N_\F(UC_S(U))$ over $N_T(U)$. Hence, we only need to show that the extension condition for normal subsystems holds. Assume that $U$ is a counterexample to (a) such that $|U|$ is maximal. 

\smallskip

Since $\E$ is normal in $\F$, every $\phi\in\Aut_\E(T)=\Aut_{N_\E(T)}(T)$ extends to $\hat{\phi}\in\Aut_\F(TC_S(T))=\Aut_{N_\F(TC_S(T))}(TC_S(T))$ with $[C_S(T),\hat{\phi}]\leq Z(T)$. Hence, $N_\E(T)\unlhd N_\F(TC_S(T))$. As $U$ is a counterexample to (a), we have 
\[U<T\mbox{ and }U<P:=N_T(U).\]
Let $Q$ be a fully $\F$-normalized $\F$-conjugate of $P$. Notice that $P\in\E^c$ and so, by Lemma~\ref{L:EcFinvariant}, $Q\in\E^c$. By the maximality of $|U|$, $Q$ is not a counterexample, i.e. $N_\E(Q)\unlhd N_\F(QC_S(Q))$. So by Lemma~\ref{L:LocalNormalTechnicalHelp}, an automorphism $\phi\in\Aut_{N_\E(U)}(P)$ extends to $\hat{\phi}\in\Aut_\F(PC_S(U))$ with $[C_S(U),\hat{\phi}]\leq Z(U)$. Now $\hat{\phi}$ is a morphism in $N_\F(UC_S(U))$. As $C_S(P)\leq C_S(U)$, the restriction $\ov{\phi}:=\hat{\phi}|_{PC_S(P)}$ is well-defined and an element of $\Aut_{N_\F(UC_S(U))}(PC_S(P))$ with 
\[[C_S(P),\ov{\phi}]\leq U\cap C_S(P)\leq Z(P).\]
So the extension property for normal subsystems holds for the pair $(N_\E(U),N_\F(UC_S(U)))$. This contradicts the assumption that $U$ is a counterexample. 
\end{proof}

The following lemma is essential in the proof of Theorem~\ref{main}.

\begin{lemma}\label{LocalNormalSubsystems}
Suppose $\m{E}$ is a normal subsystem of $\F$ over $T$. Let $P\in\F^f$ such that $P\leq T$. Then $P\in \E^f$ and the subsystems $N_\F(P)$ and $N_\E(P)$ are saturated. Moreover, $N_\E(P)$ is a normal subsystem of $N_\F(P)$.
\end{lemma}

\begin{proof}
By Lemma~\ref{L:WeaklyNormalLocal}(a), $P\in\E^f$ and $N_\E(P)$ is a weakly normal subsystem of $N_\F(P)$. In particular, the subsystems $N_\E(P)$ and $N_\F(P)$ are saturated. Notice that $U:=N_T(P)\in\E^c$. Hence by Lemma~\ref{L:LocalNormalConstrainedSubsystems}(b), every $\phi\in\Aut_\E(U)$ extends to $\hat{\phi}\in\Aut_\F(UC_S(U))$ with $[C_S(U),\hat{\phi}]\leq Z(U)$. This shows that the extension property for normal subsystems holds for the pair $(N_\E(P),N_\F(P))$. Thus the assertion follows.
\end{proof}

We take the opportunity to state a further technical lemma which will be used in the proof of Lemma~\ref{L:EcsubsetFq}, which in turn is needed in the proof of Theorem~\ref{main}.

\begin{lemma}\label{L:LocalNormalConstrainedSubsystemsV}
Suppose $\E$ is a normal subsystem of $\F$. Let $U\in\E^c\cap\F^f$ and $V\in C_\F(U)^c$ such $UV\in N_\F(U)^f$. Then $N_\F(UV)=N_{N_\F(U)}(UV)$ is constrained and $N_\E(U)$ is a normal subsystem of $N_\F(UV)$.  
\end{lemma}

\begin{proof}
Let $T\leq S$ such that $\E$ is a fusion system over $T$. As $U$ is fully $\F$-normalized and $UV$ is fully $N_\F(U)$-normalized, $N_\F(U)$ and $N_{N_\F(U)}(UV)$ are saturated. Notice that $V\cap T\leq C_T(U)\leq U$ and so $UV\cap T=U(V\cap T)=U$. Since $T$ is strongly closed, it follows that $N_\F(UV)=N_{N_\F(U)}(UV)$. Since $V\in C_\F(U)^c$, we have $C_S(UV)=C_{C_S(U)}(V)\leq V\leq UV$ and $N_\F(UV)$ is constrained. 

\smallskip

By Lemma~\ref{L:LocalNormalConstrainedSubsystems}(a), $N_\F(UC_S(U))$ is a constrained fusion system which contains $N_\E(U)$ as a normal subsystem. Hence, by \cite[Theorem~III.5.10]{Aschbacher/Kessar/Oliver:2011}, there exists a model $G$ for $N_\F(UC_S(U))$ such that $UC_S(U)\unlhd G$; moreover, by \cite[Theorem~II.7.5]{Aschbacher/Kessar/Oliver:2011}, there is a normal subgroup $N$ of $G$ such that $N_T(U)\in\Syl_p(N)$ and $N_\E(U)=\F_{N_T(U)}(N)$. The fact that $N_T(U)$ is Sylow in $N$ implies in particular that $N_S(UC_S(U))\cap N=N_T(U)$. Using $UC_S(U)\unlhd G$ and $C_T(U)\leq U$, we conclude 
\[[UV,N]\leq [UC_S(U),N]\leq UC_S(U)\cap N=U(C_S(U)\cap N)\leq UC_T(U)=U.\]
So $N$ normalizes $UV$ and $N_\E(U)=\F_{N_T(U)}(N)\subseteq N_\F(UV)\subseteq N_\F(U)$. As $N_\E(U)$ is $N_\F(U)$-invariant by Lemma~\ref{L:WeaklyNormalLocal}(a), it is also $N_\F(UV)$-invariant by \cite[Proposition~I.6.4]{Aschbacher/Kessar/Oliver:2011}. Notice that $P:=N_T(U)\in \E^c$. By Lemma~\ref{L:LocalNormalConstrainedSubsystems}(a), for a fully normalized $\F$-conjugate $Q$ of $P$, we have $N_\E(Q)\unlhd N_\F(QC_S(Q))$. Hence, by Lemma~\ref{L:LocalNormalTechnicalHelp}, an automorphism $\phi\in\Aut_{N_\E(U)}(P)$ extends to $\hat{\phi}\in\Aut_\F(PC_S(U))$ with $[C_S(U),\hat{\phi}]\leq Z(U)$. As $VC_S(P)\leq C_S(U)$, $\hat{\phi}$ normalizes $UV$ and $\ov{\phi}=\hat{\phi}|_{PC_S(P)}\in\Aut_{N_\F(UV)}(PC_S(P))$ with $[C_S(P),\hat{\phi}]\leq Z(U)\cap C_S(P)\leq Z(P)$. Hence, the extension condition for normal subsystems holds and the assertion follows.
\end{proof}

\subsection{Central products}\label{SS:CentralProdFusion}

Aschbacher \cite[pp.13-14]{Aschbacher:2011} introduced central products of fusion systems as certain quotients of direct products of fusion systems. Central products in this definition can be seen as external central products. When stating properties of components in \cite[Chapter~9]{Aschbacher:2011}, Aschbacher uses implicitly a notion of internal central products. In this subsection we give a precise definition of internal central products of fusion systems. Moreover, we study properties of such internal central products and show their close relationship to external central products. The results will be used e.g. in Subsection~\ref{SS:CentralizingCommuting}, Lemma~\ref{GeneralizedFitting} and Subsection~\ref{SS:EFNewProofs}.

\smallskip

We continue to assume that $S$ is a $p$-group and $\F$ is a fusion system over $S$. In this subsection, $\F$ is not assumed to be saturated. Indeed, for part of our definitions and results the fusion system $\F$ will not play any role. Throughout we pick $k\in\mathbb{N}$ with $k\geq 1$.

\begin{notation}\label{N:CentralProduct} 
Suppose that $\F_i$ is a fusion system over $S_i\leq S$ for $i=1,2,\dots,k$ such that $[S_i,S_j]=1$ for all $i\neq j$. Assume furthermore that $S_i\cap \prod_{j\neq i}S_j\leq Z(\F_i)$ for all $i=1,2,\dots,k$. Then we use the following notation.
\begin{itemize}
\item Given $P_i,Q_i\leq S_i$ and $\phi_i\in\Hom_{\F_i}(P_i,Q_i)$ for $i=1,2,\dots,k$, write $\phi_1*\phi_2*\cdots *\phi_k$ for the map $P_1P_2\cdots P_k\rightarrow Q_1Q_2\cdots Q_k$ sending $x_1x_2\cdots x_k$ to $(x_1\phi_1)(x_2\phi_2)\cdots (x_k\phi_k)$ whenever $x_i\in P_i$ for $i=1,\dots,k$.
\item Write $\F_1*\F_2*\cdots *\F_k$ for the fusion system over $S_1S_2\cdots S_k$ which is generated by the maps $\phi_1*\phi_2*\cdots *\phi_k$ with $P_i\leq S_i$ and $\phi_i\in\Hom_{\F_i}(P_i,S_i)$ for $i=1,2,\dots,k$.
\item For $A\leq S_1S_2\cdots S_k$ set
\[A_i:=\{x_i\in S_i\colon \exists x\in\prod_{j\neq i}S_j\mbox{ such that }x_ix\in A\}.\]
\end{itemize}   
\end{notation}

\begin{lemma}\label{L:StarProdBasic}
Let $\F_i$ be a fusion system over $S_i\leq S$ for $i=1,2,\dots,k$ such that $[S_i,S_j]=1$ for $i\neq j$. Assume furthermore that $S_i\cap \prod_{j\neq i}S_j\leq Z(\F_i)$ for $i=1,2,\dots,k$. Then the following hold:
\begin{itemize}
 \item [(a)] Let $P_i\leq S_i$ and $\phi_i\in\Hom_{\F_i}(P_i,S_i)$ for $i=1,\dots,k$. Then $\phi_1*\phi_2*\cdots *\phi_k$ is well-defined and an injective group homomorphism. In particular $\F_1*\F_2*\cdots *\F_k$ is well-defined.
 \item [(b)] $Z(\F_1)Z(\F_2)\cdots Z(\F_k)=Z(\F_1*\F_2*\cdots *\F_k)$. Moreover, if $s_i\in S_i$ for $i=1,\dots,k$ such that $s_1s_2\cdots s_k\in Z(\F_1*\F_2*\cdots *\F_k)$, then $s_i\in Z(\F_i)$ for all $i=1,2,\dots,k$. 
 \item [(c)] If $1\leq l<k$, then $(\prod_{i=1}^l S_i)\cap (\prod_{i=l+1}^kS_i)\leq Z(\F_1*\cdots *\F_l)\cap Z(\F_{l+1}*\cdots *\F_k)$ and
\[ \F_1*\F_2*\cdots *\F_k=(\F_1*\cdots \F_l)*(\F_{l+1}*\cdots *\F_k). \]
\item [(d)] Let $A\leq \prod_{i=1}^kS_i$ and let $\phi\colon A\rightarrow \prod_{i=1}^k S_i$ be a morphism in $\F_1*\F_2*\cdots *\F_k$. Then there exist $\phi_i\in \Hom_\F(A_i,S_i)$ for $i=1,\dots,k$ (with $A_i$ as defined in Notation~\ref{N:CentralProduct}) such that $\phi=(\phi_1*\phi_2*\cdots *\phi_k)|_A$.
\item [(e)] $\F_i$ is $\F_1*\F_2*\cdots *\F_k$-invariant for $i=1,\dots,k$.
\end{itemize}
\end{lemma}

\begin{proof}
\textbf{(a)} Let $P_i$ and $\phi_i$ be as in (a) and $x_i,y_i\in P_i$ for $i=1,\dots,k$. Suppose first $x_1x_2\cdots x_k=y_1y_2\cdots y_k$. Using $[S_i,S_j]=1$ for all $i\neq j$, this equality can be reformulated to $\prod_{i=1}^k x_iy_i^{-1}=1$ and to $x_iy_i^{-1}=\prod_{j\neq i}y_jx_j^{-1}$ for all $i=1,\dots,k$. Since $S_i\cap \prod_{j\neq i}S_j\leq Z(\F_i)$ for all $i=1,\dots,k$, the latter equation implies $x_iy_i^{-1}\in Z(\F_i)$ and hence $(x_i\phi_i)(y_i\phi_i)^{-1}=(x_iy_i^{-1})\phi_i=x_iy_i^{-1}$ for all $i=1,\dots,k$. Hence, we obtain $\prod_{i=1}^k(x_i\phi_i)(y_i\phi_i)^{-1}=1$ and thus $(x_1\phi_1)(x_2\phi_2)\cdots (x_k\phi_k)=(y_1\phi_1)(y_2\phi_2)\cdots (y_k\phi_k)$. So we have shown the implication
\begin{equation}\label{E:Implication}
 x_1x_2\cdots x_k=y_1y_2\cdots y_k\Longrightarrow (x_1\phi_1)(x_2\phi_2)\cdots (x_k\phi_k)=(y_1\phi_1)(y_2\phi_2)\cdots (y_k\phi_k),
\end{equation}
i.e. $\phi_1*\phi_2*\cdots *\phi_k$ is well-defined. Using \eqref{E:Implication} now with $P_i\phi_i$ and $\phi_i^{-1}$ in place of $P_i$ and $\phi_i$ we also obtain the implication
\[(x_1\phi_1)(x_2\phi_2)\cdots (x_k\phi_k)=(y_1\phi_1)(y_2\phi_2)\cdots (y_k\phi_k)\Longrightarrow x_1x_2\cdots x_k=y_1y_2\cdots y_k,\]
which says that $\phi_1*\phi_2*\cdots *\phi_k$ is injective. Hence (a) holds.

\smallskip

\textbf{(b)} Set $Z:=Z(\F_1)Z(\F_2)\cdots Z(\F_k)$. As before let $P_i\leq S_i$ and $\phi_i\in\Hom_{\F_i}(P_i,S_i)$ for $i=1,\dots,k$. To prove $Z\leq Z(\F_1*\F_2*\cdots *\F_k)$, it is enough to show that $\phi_1*\phi_2*\cdots *\phi_k$ extends in $\F_1*\F_2*\cdots *\F_k$ to a morphism $(P_1P_2\cdots P_k)Z\rightarrow S_1S_2\cdots S_k$ which fixes every element of $Z$. Indeed, each $\phi_i$ extends to a morphism $\hat{\phi}_i\in\Hom_{\F_i}(P_iZ(\F_i),S_i)$ such that $\hat{\phi}_i$ fixes every element of $Z(\F_i)$. Then $\hat{\phi}_1*\hat{\phi}_2 *\cdots *\hat{\phi}_k$ is an extension of $\phi_1*\phi_2 *\cdots *\phi_k$ with the desired properties.

\smallskip

Let now $s_i\in S_i$ for $i=1,\dots,k$ such that $z:=s_1s_2\cdots s_k\in Z(\F_1*\F_2*\cdots *\F_k)$. Fix $i\in\{1,2,\dots,k\}$, $Q_i\leq S_i$ and $\psi_i\in\Hom_{\F_i}(Q_i,S_i)$. Set $\psi:=\id_{S_1}*\cdots *\id_{S_{i-1}}*\psi_i*\id_{S_{i+1}}*\cdots *\id_{S_k}$ and $Q:=S_1\cdots S_{i-1}Q_iS_{i+1}\cdots S_k$. Then $\psi\colon Q\rightarrow S_1S_2\cdots S_k$ is a morphism in $\F_1*\F_2*\cdots *\F_k$ and extends thus to $\hat{\psi}\in\Hom_{\F_1*\cdots *\F_k}(\<Q,z\>,S_1S_2\cdots S_k)$ with $\hat{\psi}(z)=z$. Notice that $\<Q,z\>=\<Q,s_i\>$ and $s_j\in \<Q,z\>$ for $j=1,\dots,k$. It follows from the definition of $\F_1*\F_2*\cdots *\F_k$ that $\hat{\psi}_i:=\hat{\psi}|_{\<Q_i,s_i\>}$ is a morphism in $\F_i$. Since $\hat{\psi}|_{S_j}=\psi|_{S_j}=\id_{S_j}$ for $j\neq i$, we have  $s_1s_2\cdots s_k=z=\hat{\psi}(z)=s_1\cdots s_{i-1}\hat{\psi}_i(s_i)s_{i+1}\cdots s_k$ and thus $\hat{\psi}_i(s_i)=s_i$. Moreover, $\hat{\psi}_i$ extends $\psi_i=\psi|_{Q_i}=\hat{\psi}|_{Q_i}$. This proves $s_i\in Z(\F_i)$ and completes thus the proof of (b). 

\smallskip

\textbf{(c,d)} Properties (c) and (d) are trivially true for $k=1$, so we may assume $k\geq 2$ and fix $1\leq l<k$. Let first $s_i\in S_i$ for $i=1,\dots,k$ such that $s_1\cdots s_l=s_{l+1}\cdots s_k$. By rearranging this equality, one sees  that $s_i\in S_i\cap\prod_{j\neq i}S_j\leq Z(\F_i)$ for $i=1,\dots,k$. Using (b) we can thus conclude that
\[(\prod_{i=1}^l S_i)\cap (\prod_{i=l+1}^kS_i)\leq (Z(\F_1)\cdots Z(\F_l))\cap (Z(\F_{l+1})\cdots Z(\F_k))=  Z(\F_1*\cdots *\F_l)\cap Z(\F_{l+1}*\cdots *\F_k).\]
In particular, $(\F_1*\cdots *\F_l)*(\F_{l+1}*\cdots *\F_k)$ is well-defined. If $P_i\leq S_i$ and $\phi_i\in\Hom_{\F_i}(P_i,S_i)$ for $i=1,\dots,k$, then one observes easily that 
\[\phi_1*\phi_2*\cdots *\phi_k=(\phi_1*\cdots *\phi_l)*(\phi_{l+1}*\cdots *\phi_k).\] This implies $\F_1*\F_2*\cdots *\F_k\subseteq (\F_1*\cdots *\F_l)*(\F_{l+1}*\cdots *\F_k)$. So it remains to prove that $(\F_1*\cdots *\F_l)*(\F_{l+1}*\cdots *\F_k)\subseteq \F_1*\F_2*\cdots *\F_k$ and that (d) holds. Fix $A\leq \prod_{i=1}^k S_i$ and a morphism $\phi\colon A\rightarrow \prod_{i=1}^k S_i$ in $(\F_1*\cdots *\F_l)*(\F_{l+1}*\cdots *\F_k)$. It is sufficient to argue that $\phi$ is of the form $\phi=(\phi_1*\phi_2*\cdots *\phi_k)|_A$ for some $\phi_i\in\Hom_{\F_i}(A_i,S_i)$. We prove this by induction on $k$. For $k=2$, part (d) and thus our claim is true by \cite[Lemma~2.11(f)]{Henke:Regular}. Set $T_1:=\prod_{i=1}^lS_i$, $T_2:=\prod_{i=l+1}^kS_i$ and
\[B_j:=\{x\in T_j\colon \exists t\in T_{3-j}\mbox{ such that }xt\in A\}\mbox{ for }j=1,2.\]
Since (d) is true for $k=2$, there are $\psi_1\in\Hom_{\F_1*\cdots *\F_l}(B_1,T_1)$ and $\psi_2\in \Hom_{\F_{l+1}*\cdots *\F_k}(B_2,T_2)$ such that $\phi=(\psi_1*\psi_2)|_A$. Now notice that 
\[A_i=\{x_i\in S_i\colon \exists y\in\prod_{j\neq i,\;1\leq j\leq l}S_j\mbox{ such that }x_iy\in B_1\} \mbox{ for }i=1,\dots,l\]
and
\[A_i=\{x_i\in S_i\colon \exists y\in\prod_{j\neq i,\;l+1\leq j\leq k}S_j\mbox{ such that }x_iy\in B_2\}\mbox{ for }i=l+1,\dots,k. \]
Therefore, induction on $k$ gives the existence of $\phi_i\in\Hom_{\F_i}(A_i,S_i)$ for $i=1,\dots,k$ such that $\psi_1=(\phi_1*\cdots *\phi_l)|_{Q_1}$ and $\psi_2=(\phi_{l+1}*\cdots *\phi_k)|_{Q_2}$. This yields
$\phi=(\psi_1*\psi_2)|_A=(\phi_1*\phi_2*\cdots *\phi_k)|_A$. Thus (c) and (d) hold.

\smallskip

\textbf{(e)} Observe that the restriction of a morphism in $\F_1*\F_2*\cdots *\F_k$ to a subgroup of $S_i$ is always a morphism $\F_i$ for $i=1,\dots,k$. It follows from this property that (e) is true and the proof is complete.  
\end{proof}

The following lemma complements the statement in Lemma~\ref{L:StarProdBasic}(c). 

\begin{lemma}\label{L:StarAssociativeHelp}
Let $\F_i$ be a fusion system over $S_i\leq S$ for $i=1,2,\dots,k$ such that $[S_i,S_j]=1$ for $i\neq j$. Fix $l\in\mathbb{N}$ with $1\leq l<k$. Then $S_i\cap (\prod_{j\neq i,\;1\leq j\leq k}S_j)\leq Z(\F_i)$ for all $i=1,2,\dots,k$ if and only if the following properties hold:
\begin{itemize}
 \item [(i)] $S_i\cap (\prod_{j\neq i,\;1\leq j\leq l}S_j)\leq Z(\F_i)$ for all $i=1,2,\dots,l$;
 \item [(ii)] $S_i\cap (\prod_{j\neq i,\;l+1\leq j\leq k}S_j)\leq Z(\F_i)$ for all $i=l+1,\dots,k$;
 \item [(iii)] $(\prod_{i=1}^l S_i)\cap (\prod_{i=l+1}^kS_i)\leq Z(\F_1*\cdots *\F_l)\cap Z(\F_{l+1}*\cdots *\F_k)$.
\end{itemize}
\end{lemma}

\begin{proof}
If $S_i\cap \prod_{j\neq i}S_j\leq Z(\F_i)$ for all $i=1,2,\dots,k$, then clearly (i) and (ii) hold, and (iii) follows from Lemma~\ref{L:StarProdBasic}(c).

\smallskip

Suppose now that properties (i)-(iii) hold. Let $i\in\{1,2,\dots,k\}$ and $s_i\in S_i\cap (\prod_{j\neq i,\;1\leq j\leq k}S_j)$. Then for $j\in\{1,2,\dots,k\}$ with $j\neq i$ there exist $s_j\in S_j$ such that $s_i=\prod_{j\neq i,\;1\leq j\leq k}s_j^{-1}$. It follows $s_1s_2\cdots s_k=1$ and thus $\prod_{j=1}^ls_j=\prod_{j=l+1}^ks_j^{-1}\in Z(\F_1*\cdots *\F_l)\cap Z(\F_{l+1}*\cdots *\F_k)$ by (iii). Therefore, (i),(ii) and Lemma~\ref{L:StarProdBasic}(b) yield $s_j\in Z(\F_j)$ for $j=1,2,\dots,k$. In particular, $s_i\in Z(\F_i)$, which proves $S_i\cap (\prod_{j\neq i,\;1\leq j\leq k}S_j)\leq Z(\F_i)$.
\end{proof}

\begin{definition}
Suppose that $\F_i$ is a fusion system over a $p$-group $S_i$ for $i=1,\dots,k$. 
\begin{itemize}
\item If $P_i\leq S_i$ and $\phi_i\in\Hom_{\F_i}(P_i,S_i)$ for $i=1,\dots,k$ let 
\[\phi_1\times\phi_2\times\cdots\times \phi_k\colon P_1\times P_2\times\cdots\times P_k\rightarrow S_1\times S_2\times \cdots \times S_k\] be the map with $(x_1,x_2,\cdots,x_k)\mapsto (x_1\phi_1,x_2\phi_2,\cdots,x_k\phi_k)$. It is easy to check that $\phi_1\times\phi_2\cdots\times\phi_k$ is an injective group homomorphism. Write $\F_1\times\F_2\times\cdots\times \F_k$ for the fusion system over $S_1\times S_2\times \cdots \times S_k$ which is generated by all the maps $\phi_1\times\phi_2\times\cdots\times\phi_k$ with $\phi_i\in\Hom_{\F_i}(P_i,S_i)$ for $i=1,\dots,k$. 
\item For $i=1,\dots,k$ let $\iota_i\colon S_i\rightarrow S_1\times S_2\times \cdots\times S_k$ be the natural inclusion map sending $s_i\in S_i$ to the tuple with $s_i$ in the $i$th entry and all other entries equal to one. It is easy to check that $\iota_i$ induces a morphism from $\F_i$ to $\F_1\times\F_2\times\cdots\times \F_k$. Set $\hat{S}_i:=S_i\iota_i$ and $\hat{\F}_i=\F_i^{\iota_i}$ (cf. Definition~\ref{D:MorphismFusionSystemImage}). 
\end{itemize}
\end{definition}

Observe that $\F_i\cong \hat{\F}_i$.

\begin{lemma}\label{L:InternalExternalCentralProduct}
For $i=1,\dots,k$ let $\F_i$ be a fusion system over $S_i\leq S$. Suppose $[S_i,S_j]=1$ for $i\neq j$. Consider the map
\[\alpha \colon S_1\times S_2\times\cdots\times S_k\rightarrow \prod_{i=1}^kS_i,(s_1,s_2,\cdots,s_k)\mapsto s_1s_2\cdots s_k.\]
Then the following hold:
\begin{itemize}
\item [(a)] $\alpha$ is a surjective group homomorphism with $\ker(\alpha)\cap \hat{S}_i=1$ for $i=1,2,\dots,k$. Moreover, $S_i\cap \prod_{j\neq i}S_j\leq Z(\F_i)$ for all $i=1,\dots,k$ if and only if 
\[\ker(\alpha)\leq Z(\F_1)\times Z(\F_2)\times\cdots\times Z(\F_k).\]
\item [(b)] If $S_i\cap \prod_{j\neq i}S_j\leq Z(\F_i)$ for all $i=1,\dots,k$, then $\alpha$ induces an epimorphism of fusion systems from $\F_1\times \F_2\times\cdots\times\F_k$ to $\F_1*\F_2*\cdots *\F_k$ with $\hat{\F}_i^\alpha=\F_i$. In particular, 
\[\F_1*\F_2*\cdots *\F_k\cong (\F_1\times\F_2\times\cdots \times\F_k)/\ker(\alpha)\]
is isomorphic to an (external) central product of $\F_1,\dots,\F_k$ in the sense of \cite[p.14]{Aschbacher:2011}.
\item [(c)] If $\F_1,\F_2,\dots,\F_k$ are saturated and $S_i\cap \prod_{j\neq i}S_j\leq Z(\F_i)$ for all $i=1,\dots,k$, then $\F_1*\F_2*\cdots *\F_k$ is saturated and $\F_i\unlhd\F_1*\F_2*\cdots *\F_k$ for $i=1,\dots,k$. 
\end{itemize}
\end{lemma}

\begin{proof}
Write $\F:=\F_1\times\F_2\times\cdots\times\F_k$ and $\tF:=\F_1*\F_2*\cdots *\F_k$ for short. Set $Z_i:=S_i\cap \prod_{j\neq i}S_j$ for $i=1,\dots,k$. For every subgroup $P$ of $S_1\times S_2\times\cdots\times S_k$ and $i=1,\dots,k$ write $P_i$ for the image of $P$ under the natural projection map $S_1\times S_2\times\cdots \times S_k\rightarrow S_i$. 

\smallskip

It follows from $[S_i,S_j]=1$ for $i\neq j$ that $\alpha$ is a group homomorphism. Clearly $\alpha$ is surjective and $\ker(\alpha)\cap\hat{S}_i=1$. Notice that $\ker(\alpha)\leq Z_1\times Z_2\times\cdots\times Z_k$. Moreover, for $z_i\in Z_i$, we have $z_i\in S_i$ and there exist $z_j\in S_j$ for $j\neq i$ such that $z_i=\prod_{j\neq i}z_j^{-1}$. Then $(z_1,z_2,\dots,z_k)\alpha=z_1z_2\cdots z_k=1$, thus $(z_1,z_2,\dots,z_k)\in\ker(\alpha)$ and $z_i\in\ker(\alpha)_i$. Moreover, if $Q\leq S_1S_2\cdots S_k$, then $z_1z_2\cdots z_k=1\in Q$ and hence $z_i\in Q_i$, where $Q_i$ is defined according to Notation~\ref{N:CentralProduct}. This proves 
\begin{equation}\label{E:Zi}
 Z_i=\ker(\alpha)_i\mbox{ and }Z_i\leq Q_i\mbox{ for all }Q\leq S_1S_2\cdots S_k\mbox{ and all }i=1,\dots,k.
\end{equation} 
\indent\textbf{(a)} Part (a) follows directly from $Z_i=\ker(\alpha)_i$ for all $i=1,\dots,k$, which is true by \eqref{E:Zi}. 

\smallskip

\textbf{(b)} Assume now $Z_i\leq Z(\F_i)$ for all $i=1,\dots,k$. Let $P\leq S_1\times S_2\times\cdots\times S_k$.  
 As an intermediate step, we argue next that
\begin{equation}\label{E:PiZi}
 P_iZ_i=(P\alpha)_i
\mbox{ for }i=1,\dots,k.
\end{equation}
Clearly $P_i\leq (P\alpha)_i$. It follows from \eqref{E:Zi} applied with $P\alpha$ in place of $Q$ that $Z_i\leq (P\alpha)_i$. Hence $P_iZ_i\leq (P\alpha)_i$. Given $x_i\in (P\alpha)_i$, there exist $x_j\in S_j$ for all $j\neq i$ such that $x_1x_2\cdots x_k\in P\alpha$. Further, there is then $(y_1,\dots,y_k)\in P$ with $y_1y_2\cdots y_k=(y_1,\dots,y_k)\alpha=x_1x_2\cdots x_k$. Notice now that $y_i\in P_i$ and $x_iy_i^{-1}=\prod_{j\neq i}y_jx_j^{-1}\in Z_i$. Hence $x_i\in Z_iP_i$ and \eqref{E:PiZi} follows.

\smallskip

We are now in a position to show that $\alpha$ induces an epimorphism from $\F$ to $\tF$. Let $P,Q\leq S_1\times S_2\times\cdots\times S_k$ and $\phi\in\Hom_\F(P,Q)$. It is easy to check (and proved for $k=2$ in \cite[Theorem~I.6.6]{Aschbacher/Kessar/Oliver:2011}) that $\phi$ is of the form $\phi=(\phi_1\times\phi_2\times\cdots\times\phi_k)|_P$ for some $\phi_i\in\Hom_{\F_i}(P_i,Q_i)$ ($i=1,\dots,k$). As $Z_i\leq Z(\F_i)$ for $i=1,\dots,k$, each $\phi_i$ extends to a unique $\hat{\phi}_i\in\Hom_{\F_i}(P_iZ_i,Q_iZ_i)$ with $\hat{\phi}_i|_{Z_i}=\id_{Z_i}$. Given such $\phi_i$ and $\hat{\phi}_i$, we have 
\[((\phi_1\times \phi_2\times\cdots\times\phi_k)|_P)\alpha=\alpha((\hat{\phi}_1*\hat{\phi}_2*\cdots *\hat{\phi}_k)|_{P\alpha});\]
this uses \eqref{E:PiZi} to see that $(\hat{\phi}_1*\hat{\phi}_2*\cdots *\hat{\phi}_k)|_{P\alpha}$ is well-defined. Thus, $\alpha$ induces a morphism from $\F$ to $\hat{\F}$ where the corresponding map $\alpha_{P,Q}\colon \Hom_\F(P,Q)\rightarrow \Hom_{\tF}(P\alpha,Q\alpha)$ is given by $(\phi_1\times\phi_2\times\cdots \times\phi_k)|_P\mapsto (\hat{\phi}_1*\hat{\phi}_2*\cdots *\hat{\phi}_k)|_{P\alpha}$. Suppose now  $\ker(\alpha)\leq P$. Then it follows from \eqref{E:Zi} and \eqref{E:PiZi} that $Z_i=\ker(\alpha)_i\leq P_i$ and $P_i=(P\alpha)_i$ for $i=1,\dots,k$. Hence, the map $\alpha_{P,Q}$ maps $(\phi_1\times \phi_2\times \cdots \phi_k)|_P$ simply to $(\phi_1*\phi_2*\cdots *\phi_k)|_{P\alpha}$ whenever $\phi_i\in \Hom_{\F_i}(P_i,Q_i)$ for $i=1,\dots,k$. Hence, it is a consequence of Lemma~\ref{L:StarProdBasic}(d) that $\alpha_{P,Q}$ is surjective. This proves that $\alpha$ induces an epimorphism from $\F$ to $\tF$. In particular, $\hat{\F}_i^\alpha$ is defined. Notice that $\alpha|_{\hat{S}_i}=\iota_i^{-1}$ and so $\hat{\F}_i^\alpha=(\F_i^{\iota_i})^\alpha=\F_i$. This shows (b).

\smallskip

\textbf{(c)} We continue to assume that $Z_i\leq Z(\F_i)$ and we assume now in addition that $\F_1,\dots,\F_k$ are saturated. We argue first that $\F_1*\F_2*\cdots \F_k$ is saturated. Induction on $k$ allows us to assume that $\F_1*\F_2*\cdots *\F_{k-1}$ is  saturated. By Lemma~\ref{L:StarProdBasic}(c), $(\prod_{j=1}^{k-1}S_j)\cap S_k\leq Z(\F_1*\F_2*\cdots *\F_{k-1})\cap Z(\F_k)$ and $\F_1*\F_2*\cdots *\F_k=(\F_1*\F_2*\cdots *\F_{k-1})*\F_k$. Hence, part (b) gives that $\F_1*\F_2*\cdots *\F_k$ is an epimorphic image of $(\F_1*\F_2*\cdots *\F_{k-1})\times \F_k$, which is a saturated fusion system by \cite[Theorem~I.6.6]{Aschbacher/Kessar/Oliver:2011}. Hence, $\F_1*\F_2*\cdots *\F_k$ is saturated by \cite[Lemma~II.5.4]{Aschbacher/Kessar/Oliver:2011}. 

\smallskip

Fix now $i\in\{1,2,\dots,k\}$. As $\F_i$ and $\F_1*\F_2*\cdots *\F_k$ are saturated, it follows from Lemma~\ref{L:StarProdBasic}(e) that $\F_i$ is weakly normal in $\F_1*\F_2*\cdots *\F_k$. Moreover, the extension condition for normal subsystems (cf. \cite[Definition~I.6.1]{Aschbacher/Kessar/Oliver:2011}) holds because, for $\phi_i\in\Aut_{\F_i}(S_i)$, we obtain a suitable extension of $\phi_i$ by setting $\phi_j:=\id_{S_j}$ for $j\neq i$ and considering $\phi_1*\phi_2*\cdots *\phi_k$.
\end{proof}

\begin{definition}\label{D:InternalCentralProduct}
For each $i=1,\dots,k$ let $\F_i$ be a subsystem of $\F$ over $S_i\leq S$. Suppose $[S_i,S_j]=1$ for $i\neq j$.
\begin{itemize}
\item The subsystems $\F_1,\dots,\F_k$ are said to \emph{centralize each other in $\F$} if $\F_i\subseteq C_\F(\prod_{j\neq i}S_j)$ and $S_i\cap \prod_{j\neq i}S_j\leq Z(\F_i)$ for all $i=1,\dots,k$. 
\item The fusion system $\F$ is an \emph{internal central product} of $\F_1,\dots,\F_k$ if $S_i\cap \prod_{j\neq i}S_j\leq Z(\F_i)$ for all $i=1,\dots,k$ and $\F=\F_1*\F_2*\cdots*\F_k$.
\end{itemize}
\end{definition}

\begin{lemma}\label{L:CentralizeEachOtherSaturated}
For $i=1,\dots,k$ let $\F_i$ be a saturated subsystem of $\F$ over $S_i$ such that $[S_i,S_j]=1$ for $i\neq j$. Then $\F_1,\dots,\F_k$ centralize each other in $\F$ if and only if $\F_i\subseteq C_\F(\prod_{j\neq i}S_j)$ for $i=1,\dots,k$.
\end{lemma}

\begin{proof}
Fix $i\in\{1,\dots,k\}$ and suppose $\F_i\subseteq C_\F(\prod_{j\neq i}S_j)$. It is sufficient to show that $Z:=S_i\cap \prod_{j\neq i}S_j\leq Z(\F_i)$. Our assumption yields $Z\leq Z(S_i)$. Pick $R\in\F_i^{cr}$. Then $Z\leq Z(S_i)\leq C_{S_i}(R)\leq R$. Since $\F_i\subseteq C_\F(\prod_{j\neq i}S_j)$, it follows moreover that every element of $\Aut_{\F_i}(R)$ acts trivially on $Z$. As $\F_i$ is saturated and $R\in\F_i^{cr}$ was arbitrary, Alperin's Fusion Theorem \cite[Theorem~I.3.6]{Aschbacher/Kessar/Oliver:2011} implies the assertion.
\end{proof}

\begin{lemma}\label{L:CentralizeinFAssocDescribe}
For $i=1,\dots,k$ let $\F_i$ be a subsystem of $\F$ over $S_i$. Set $T:=\prod_{i=1}^kS_i$ and assume that $\F_1,\dots,\F_k$ centralize each other in $\F$. Then 
\[\F_1*\F_2*\cdots *\F_k=\<\phi\in\Hom_\F(P_1P_2\cdots P_k,T)\colon P_i\leq S_i,\;\phi|_{P_i}\in\Hom_{\F_i}(P_i,S_i)\>_T\subseteq\F.\]
\end{lemma}

\begin{proof}
If $\phi\in\Hom_\F(P_1P_2\cdots P_k,T)$ with $\phi_i:=\phi|_{P_i}\in\Hom_{\F_i}(P_i,S_i)$, then one observes easily that $\phi=\phi_1*\phi_2*\cdots *\phi_k$ is a morphism in $\F_1*\F_2*\cdots *\F_k$. This proves one inclusion and it remains to show the converse one.

\smallskip

Let $\phi:=\phi_1*\phi_2*\cdots *\phi_k$ where $P_i\leq S_i$ and $\phi_i\in\Hom_{\F_i}(P_i,S_i)$ for $i=1,\dots,k$. Then $\phi|_{P_i}=\phi_i$. Thus, we only need to argue that $\phi$ is a morphism in $\F$. Set $T_i:=\prod_{j\neq i}S_j$ and  $Q_i:=(\prod_{j=1}^{i-1}P_j\phi_j)(\prod_{j=i}^kP_j)$ for $i=1,\dots,k$. Since $\F_1,\F_2,\dots,\F_k$ centralize each other in $\F$, every $\phi_i$ extends to $\hat{\phi}_i\in\Hom_\F(P_iT_i,T)$ with $\hat{\phi}_i|_{T_i}=\id_{T_i}$. Then 
\[\hat{\phi}_1|_{Q_1}\hat{\phi}_2|_{Q_2}\cdots \hat{\phi}_k|_{Q_k}\in\Hom_\F(P_1P_2\cdots P_k,T)\]
is a morphism in $\F$, which is equal to $\phi$, as its restriction to $P_i$ equals $\phi_i$ for $i=1,\dots,k$. This proves the assertion.
\end{proof}

If $\F_1$ and $\F_2$ are subsystems of $\F$ which centralize each other in $\F$, then Lemma~\ref{L:CentralizeinFAssocDescribe} says that the fusion system $\F_1*\F_2$ as defined above coincides with the equally denoted subsystem introduced in \cite[Definition~3.1]{Henke:2018}. This allows us to apply results from that paper below.

\begin{lemma}\label{L:CharacterizeF1starF2}
Let $\F_1,\F_2,\dots,\F_k$ be normal subsystems of $\F$ 
\begin{itemize}
 \item [(a)] The following conditions are equivalent:
\begin{itemize}
\item[(i)] $\F_1,\F_2,\dots,\F_k$ centralize each other;
\item[(ii)] $\F_i\subseteq C_\F(\prod_{j\neq i}S_j)$ for each $i=1,2,\dots,k$;
\item[(iii)] $S_i\cap \prod_{j\neq i}S_j\leq Z(\F_i)$ for each $i=1,2,\dots,k$.
\end{itemize}
\item [(b)] If $\F_1,\dots,\F_k$ centralize each other, then $\F_1*\F_2*\cdots *\F_k\unlhd \F$.
\item [(c)] Let $\E$ be a saturated subsystem of $\F$ with $\F_i\unlhd \E$ for $i=1,2,\dots,k$. Suppose $\F_1,\dots,\F_k$ centralize each other in $\F$. Then $\F_1,\dots,\F_k$ centralize each other in $\E$ and  
\[\F_1*\F_2*\cdots *\F_k\unlhd \E.\]
\end{itemize}
\end{lemma}

\begin{proof}
Part (c) follows from (a) and (b) applied with $\E$ in place of $\F$. Thus, it remains to prove (a) and (b). Lemma~\ref{L:CentralizeEachOtherSaturated} gives that (i) and (ii) are equivalent and that (ii) implies (iii). Suppose now (iii) holds. Using induction on $k$ we will show that (ii) holds and that $\F_1*\F_2*\cdots *\F_k\unlhd\F$. This will complete the proof.

\smallskip

Clearly the claim is true for $k=1$. Suppose now that $k>1$ and that the assertion is true for $k-1$. Then in particular, $\F_1*\cdots *\F_{i-1}*\F_{i+1}*\cdots \F_k\unlhd\F$ for $i=1,\dots,k$. 
Hence, by \cite[Lemma~7.2(b)]{Henke:2018}, it follows from (iii) that $\F_i\subseteq C_\F(\prod_{j\neq i}S_j)$ for $i=1,\dots,k$. It is a special case of Lemma~\ref{L:StarProdBasic}(c) that $(\prod_{j=1}^{k-1}S_j)\cap S_k\leq Z(\F_1*\F_2*\cdots *\F_{k-1})\cap Z(\F_k)$ and $(\F_1*\F_2*\cdots *\F_{k-1})*\F_k=\F_1*\F_2*\cdots *\F_k$. As $\F_1*\F_2*\cdots *\F_{k-1}\unlhd\F$, \cite[Theorem~3]{Henke:2018} yields $\F_1*\F_2*\cdots *\F_k=(\F_1*\F_2*\cdots *\F_{k-1})*\F_k\unlhd\F$. This proves the assertion.
\end{proof}

\begin{lemma}\label{L:CharacterizeF1starF2Real}
Let $\F_1,\F_2,\dots,\F_k$ be normal subsystems of $\F$ which centralize each other. Then $\F_1*\F_2*\cdots *\F_k$ is the smallest normal subsystem of $\F$ in which $\F_1,\F_2,\dots,\F_k$ are normal.
\end{lemma}

\begin{proof}
By Lemma~\ref{L:CharacterizeF1starF2}(b),(c), $\F_1*\F_2*\cdots *\F_k$ is a normal subsystem of $\F$ that is contained in every normal subsystem in which $\F_1,\dots,\F_k$ are normal. By Lemma~\ref{L:InternalExternalCentralProduct}(c),  $\F_i\unlhd\F_1*\F_2*\cdots *\F_k$ for $i=1,2,\dots,k$. 
\end{proof}

\subsection{Normal subsystems of $p$-power index and products of normal subsystems with $p$-subgroups}\label{SS:Fusionppowerindex} For the convenience of the reader we will summarize some background definitions and results in this subsection. We also take the opportunity to prove some more specialized lemmas that are needed later on. 

\smallskip

\textbf{For the remainder of this section $\F$ is assumed to be saturated.}

\begin{definition}[{\cite[Definitions~I.7.1 and I.7.3]{Aschbacher/Kessar/Oliver:2011}}]
The \emph{hyperfocal subgroup} of $\F$ is the subgroup
\[\hyp(\F)=\<[P,O^p(\Aut_\F(P))]\colon P\leq S\>.\]
A subsystem $\m{E}$ of $\F$ over $T\leq S$ has $p$-power index if $\hyp(\F)\leq T$ and $O^p(\Aut_\F(P))\leq \Aut_{\m{E}}(P)$ for every subgroup $P\leq T$. 
\end{definition}

It turns out that there is a unique smallest saturated subsystem of $\F$ of $p$-power index. It is explicitly given by
\[O^p(\F)=\<O^p(\Aut_\F(P))\colon P\leq \hyp(\F)\>_{\hyp(\F)}.\]
Moreover, $O^p(\F)$ is a normal subsystem of $\F$. More generally, for every $T\leq S$ with $\hyp(\F)\leq T$, the subsystem
\[\F_T:=\<O^p(\Aut_\F(P))\colon P\leq T\>_T\]
is the unique saturated subsystem of $\F$ over $T$ of $p$-power index. Furthermore, $\F_T$ is normal in $\F$ if and only if $T\unlhd S$. These properties are stated in \cite[Theorem~I.7.4]{Aschbacher/Kessar/Oliver:2011}; for the proof see also \cite[Theorem~4.3]{BCGLO2}. 

\smallskip

If $\E$ is a normal subsystem of $\F$ and $R$ is a subgroup of $S$, then a product subsystem $\E R$ is defined. Such a product was first introduced by Aschbacher \cite[Chapter~8]{Aschbacher:2011}, but we will use the construction given in \cite{Henke:2013}. For the convenience of the reader we summarize this construction in the definition below.

\begin{definition}\label{D:Product}
Let $\E$ be a normal subsystem of $\F$ over $T\leq S$. For every $P\leq S$, set
\begin{eqnarray*}
\Ac(P)&:=&\Ac_{\F,\E}(P)\\
&:=&\<\phi\in\Aut_\F(P)\colon \phi\mbox{ $p^\prime$-element},\;[P,\phi]\leq P\cap T\mbox{ and }\phi|_{P\cap T}\in\Aut_\E(P\cap T)\>.
\end{eqnarray*}
For every subgroup $R$ of $S$, define 
\begin{eqnarray*}
\E R:=(\E R)_\F:=\<\Ac(P)\colon P\leq TR\mbox{ and }P\cap T\in\E^c\>_{TR}
\end{eqnarray*} 
and call $\E R=(\E R)_\F$ the \emph{product} of $\E$ with $R$ formed inside of $\F$.
\end{definition}

By \cite[Theorem~1]{Henke:2013},  $\E R$ is saturated; moreover, $\E R$ is the unique saturated subsystem $\m{D}$ of $\F$ over $TR$ with $O^p(\m{D})=O^p(\E)$. An important consequence of the latter fact is stated in the following remark.

\begin{remark}\label{R:ERsubG}
Suppose $\E$ is a normal subsystem both of $\F$ and of a saturated subsystem $\G$ of $\F$ over $Q$. If $R\leq Q$, then  $(\E R)_\G$ is a saturated subsystem of $\F$ over $TR$ with $O^p((\E R)_\G)=O^p(\E)$. Thus $(\E R)_\G$ is equal to $(\E R)_\F$. In particular, $\E R=(\E R)_\F$ is contained in $\G$.
\end{remark}

\begin{remark}\label{R:ERProof}
The proof of \cite[Theorem~1]{Henke:2013} refers to \cite[Theorem~2]{Aschbacher:2008} at one point (namely in \cite[Notation~5.2]{Henke:2013}). The property that is needed is however exactly the one stated in Lemma~\ref{L:LocalNormalConstrainedSubsystems}(a). All other background results used in \cite{Henke:2013} are stated in \cite[Sections~I.1-I.7]{Aschbacher/Kessar/Oliver:2011}. 
\end{remark}

\begin{lemma}\label{L:pPowerIndexProducts}
Let $\E$ be a saturated subsystem of $\F$. Then the following conditions are equivalent:
\begin{itemize}
 \item[(i)] $\E$ has $p$-power index;
 \item[(ii)] $O^p(\F)\subseteq \E$;
 \item[(iii)] $O^p(\F)=O^p(\E)$.
\end{itemize}
In particular $O^p(\F)=O^p(O^p(\F))$. If $\E\unlhd\F$, then conditions (i)-(iii) are also equivalent to 
\begin{itemize}
 \item[(iv)] $\F=\E S$.
\end{itemize}
\end{lemma}

\begin{proof}
It follows from the definition of $p$-power index and from the concrete description of $O^p(\F)$ given above that (i) implies (ii). From the description of $O^p(\F)$ and $O^p(\E)$ one sees also that (ii) implies (iii); to show that (ii) implies $\hyp(\F)\subseteq\hyp(\E)$ one uses \[[P,O^p(\Aut_\F(P))]=[P,O^p(\Aut_\F(P)),O^p(\Aut_\F(P))]\mbox{ for every }P\leq S,\] which is true e.g. by \cite[Theorem~A2]{Aschbacher/Kessar/Oliver:2011}. Notice that the implication ``(ii)$\Longrightarrow$(iii)'' gives in particular that $O^p(\F)=O^p(O^p(\F))$.

\smallskip

Assume now that (iii) holds, i.e. $O^p(\F)=O^p(\E)$. Let $T\leq S$ such that $\E$ is a subsystem over $T$. Then $\F_T$ as defined above is a subsystem of $\F$ over $T$ with $O^p(\F)\subseteq \F_T$. Using the property that (ii) implies (iii) with $\F_T$ in place of $\E$, we conclude $O^p(\F_T)=O^p(\F)=O^p(\E)$. As $O^p(\F) T$ is by \cite[Theorem~1]{Henke:2013} the unique saturated subsystem $\m{D}$ of $\F$ over $T$ with $O^p(\m{D})=O^p(O^p(\F))$ and since $O^p(O^p(\F))=O^p(\F)$, it follows $\E=O^p(\F) T=\F_T$. In particular, (i) holds. This shows that properties (i)-(iii) are equivalent.

\smallskip

Suppose now $\E\unlhd\F$. Using again \cite[Theorem~1]{Henke:2013} we argue now that (iii) and (iv) are equivalent. If (iii) holds, then $\E S$ is the unique saturated subsystem $\m{D}$ of $\F$ such that $O^p(\m{D})$ equals $O^p(\E)=O^p(\F)$ and this implies $\E S=\F$. So (iii) implies (iv). If (iv) holds, then $O^p(\F)=O^p(\E S)=O^p(\E)$, i.e. (iv) implies (iii). This completes the proof. 
\end{proof}

\begin{corollary}\label{C:FT}
Let $\hyp(\F)\leq T\leq S$. Then $\F_T=O^p(\F)T$.
\end{corollary}

\begin{proof}
Recall that $\F_T$ is a saturated subsystem of $\F$ over $T$ of $p$-power index.  Lemma~\ref{L:pPowerIndexProducts} applied with $\F_T$ in place of $\E$ yields thus $O^p(\F_T)=O^p(\F)$. By the same lemma, $O^p(O^p(\F))=O^p(\F)$. As $O^p(\F)T$ is by \cite[Theorem~1]{Henke:2013} the unique saturated fusion system $\m{D}$ of $\F$ over $T$ with $O^p(\m{D})=O^p(O^p(\F))=O^p(\F)$, it follows $\F_T=O^p(\F)T$.
\end{proof}

\begin{lemma}\label{L:EsubseteqCERX}
Let $\E$ be a normal subsystem of $\F$ over $T$, let $R\leq S$ and $X\leq TR$ with $\E\subseteq C_\F(X)$. Then $\E\subseteq C_{\E R}(X)$. 
\end{lemma}

\begin{proof}
Let $Q\in\E^{cr}\cap \E^f$. As $Q\in\E^f$, $\Aut_T(Q)$ is a Sylow $p$-subgroup of $\Aut_\E(Q)$ and thus $\Aut_\E(Q)=\Aut_T(Q)O^p(\Aut_\E(Q))$. Notice that our assumption implies in particular that $[X,T]=1$. Hence, the elements of $\Aut_T(Q)$ are morphisms in $C_{\E R}(X)$. So by Alperin's Fusion Theorem \cite[Theorem~I.3.6]{Aschbacher/Kessar/Oliver:2011}, it is sufficient to show that the elements of $O^p(\Aut_\E(Q))$ are morphisms in $C_{\E R}(X)$. Set $P:=QX$ and fix a $p^\prime$-element $\alpha\in\Aut_\E(Q)$. As $\E\subseteq C_\F(X)$, $\alpha$ extends to $\hat{\alpha}\in\Aut_\F(P)$ with $\hat{\alpha}|_X=\id_X$. Such $\hat{\alpha}$ has the same order as $\alpha$ and is thus a $p^\prime$-element. Note also that $P\cap T=Q(X\cap T)=Q\in\E^c$ as $(X\cap T)\leq Z(T)\leq C_T(Q)\leq Q$. Moreover, as $\hat{\alpha}|_X=\id_X$, we have $[P,\hat{\alpha}]\leq Q=P\cap T$. Observe also that $\hat{\alpha}|_Q=\alpha\in\Aut_\E(Q)$. So $\hat{\alpha}\in\Ac(P)$ is a morphism in $\E R$ fixing every element of $X$. This shows that $\alpha$ is a morphism in $C_{\E R}(X)$. Since $\alpha$ was arbitrary, this proves the assertion.
\end{proof}

\begin{lemma}\label{L:ERequalsEstarR}
 Let $\E$ be a normal subsystem of $\F$ and let $R\leq S$ such that $\E\subseteq C_\F(R)$. Then $\E$ and $\F_R(R)$ centralize each other in $\E R$ and thus in $\F$. Moreover, $\E R=\E *\F_R(R)$.
\end{lemma}

\begin{proof}
By Lemma~\ref{L:EsubseteqCERX}, $\E \subseteq C_{\E R}(R)$. Using Lemma~\ref{L:CentralizeEachOtherSaturated}, one can thus conclude that $\E$ and $\F_R(R)$ centralize each other in $\E R$. Therefore, $\G:=\E *\F_R(R)\subseteq \E R$ by Lemma~\ref{L:CentralizeinFAssocDescribe}. By Lemma~\ref{L:InternalExternalCentralProduct}(c), $\G$ is saturated and $\E$ is normal in $\G$. Hence, it follows from Remark~\ref{R:ERsubG} that $\E R=(\E R)_\G\subseteq\G$.   
\end{proof}

If $\E$ is a normal subsystem of $\F$ over $T$, then it is shown in \cite[Chapter~6]{Aschbacher:2011} and in \cite[Theorem~1]{Henke:2018} that there is a subgroup $C_S(\E)$ of $S$ which is maximal with respect to inclusion among all subgroups of $S$ containing $\E$ in its centralizer. The next lemma will be needed in the proof of Proposition~\ref{P:ERNR0}, which is used in turn to show Theorem~\ref{T:mainProductWithPSubgroup}. 

\begin{lemma}\label{CSENormalER}
Let $\E$ be a normal subsystem of $\F$ over $T$, and let $R$ be a subgroup of $S$. Then $\E\unlhd\E R$. If $C_S(\E)\leq TR$, then $C_S(\E)$ is a normal subgroup of $\E R$.
\end{lemma}

\begin{proof}
Replacing $R$ by $TR$, we may assume without loss of generality that $T\leq R$. We will show first the following property.
\begin{equation}\label{E:ExtensionProp}
\mbox{Let $\alpha\in\Aut_\E(T)$. Then $\alpha$ extends to $\hat{\alpha}\in\Aut_{\E R}(TC_R(T))$ with $[TC_R(T),\hat{\alpha}]\leq T$.}
\end{equation}
Note that $\Inn(T)\in\Syl_p(\Aut_\E(T))$ by the saturation axioms. Therefore, it is indeed enough to show \eqref{E:ExtensionProp} in the case that $\alpha$ is either an element of $\Inn(T)$ or a $p^\prime$-element. If $\alpha=c_t|_T\in\Inn(T)$ with $t\in T$, then $\hat{\alpha}$ can be chosen to be $c_t|_{TC_R(T)}$.  So assume that $\alpha$ is a $p^\prime$-automorphism. As $\E$ is normal in $\F$, $\alpha$ extends to $\ov{\alpha}\in\Aut_\F(TC_S(T))$ with $[TC_S(T),\ov{\alpha}]\leq T$. We can always replace $\ov{\alpha}$ by a suitable power of $\ov{\alpha}$ to assume that $\ov{\alpha}$ is a $p^\prime$-automorphism as well. Now $\hat{\alpha}:=\ov{\alpha}|_{TC_R(T)}\in\Ac(TC_R(T))$ is a morphism in $\E R$. Moreover, we have $[TC_R(T),\hat{\alpha}]\leq [TC_S(T),\ov{\alpha}]\leq T$. This shows \eqref{E:ExtensionProp}.

\smallskip

If $\E$ is contained in a saturated subsystem $\m{D}$ of $\F$, then it follows from the equivalent definition of $\F$-invariant subsystems given in \cite[Proposition~I.6.4(d)]{Aschbacher/Kessar/Oliver:2011} that $\E$ is $\m{D}$-invariant. Thus, as $\E$ is saturated, it follows 
\begin{equation}\label{E:weaklyNormal}
\mbox{$\E$ is weakly normal in every saturated subsystem of $\F$ containing $\E$.} 
\end{equation}
Note that $\E\subseteq \E R$, so in particular $\E$ is weakly normal in $\E R$. It follows therefore from \eqref{E:ExtensionProp} that $\E\unlhd\E R$, which proves the first part of the assertion. 

\smallskip

To prove the second part of the assertion assume now that $C_S(\E)\leq R=TR$. As $C_S(\E)$ is by \cite[Theorem~1(a)]{Henke:2018} strongly closed in $\F$, it is then also strongly closed in $\E R=(\E R)_\F$. In particular, $C_S(\E)$ is fully $\E R$-normalized and so $\G:=N_{\E R}(C_S(\E))$ is a saturated subsystem of $\F$ over $R$. By Lemma~\ref{L:EsubseteqCERX} applied with $C_S(\E)$ in place of $X$, we have $\E\subseteq C_{\E R}(C_S(\E))\subseteq\G$. So $\E$ is by \eqref{E:weaklyNormal} weakly normal in $\G$. Every $\alpha\in\Aut_\E(T)$ extends by \eqref{E:ExtensionProp} to $\hat{\alpha}\in\Aut_{\E R}(TC_R(T))$ with $[TC_R(T),\hat{\alpha}]\leq T$. Our assumption yields $C_S(\E)\leq C_R(T)$. Thus, as $C_S(\E)$ is strongly closed in $\E R$, it follows $\hat{\alpha}\in\Aut_\G(TC_R(T))$. This shows $\E\unlhd\G$. 
It follows therefore from Remark~\ref{R:ERsubG} that $\E R=(\E R)_{\G}\subseteq \G:=N_{\E R}(C_S(\E))$. This means that $C_S(\E)$ is normal in $\E R$.
\end{proof}

\subsection{Centralizers of normal subsystems}\label{SS:CFE}

We continue to assume in this subsection that $\F$ is a saturated fusion system over $S$. Let $\m{E}$ be a normal subsystem of $\F$ over $T\leq S$. Recall that $C_S(\m{E})$ is a subgroup of $S$ which is maximal with respect to inclusion among all subgroups of $S$ containing $\m{E}$ in its centralizer. Thus, $C_S(\E)$ plays the role of a centralizer of $\E$ in $S$. It turns out that there is also a normal subsystem $C_\F(\E)$ of $\F$ over $C_S(\E)$ which functions as a centralizer of $\E$ in $\F$. This was first shown by Aschbacher \cite[Chapter~6]{Aschbacher:2011}. A new proof and a new construction of $C_\F(\E)$ is given in \cite{Henke:2018}. Namely, it is shown in \cite[Proposition~1]{Henke:2018} that $\hyp(C_\F(T))\leq C_S(\E)$. Using the notation from Subsection~\ref{SS:Fusionppowerindex}, the centralizer of $\E$ is then defined as
\[C_\F(\E):=C_\F(T)_{C_S(\E)},\]
i.e. as the unique saturated subsystem of $C_\F(T)$ over $C_S(\E)$ of $p$-power index. We will build on this construction of $C_\F(\E)$ and on the results from \cite{Henke:2018} in this text.

\begin{remark}\label{R:CFEBackground}
The proofs of Lemma~2.5 and Lemma~2.6 in \cite{Henke:2018} cite results from \cite{Aschbacher:2008}. However, \cite[Lemma~2.5]{Henke:2018} is the same as Lemma~\ref{LocalNormalSubsystems}. The reference to \cite[6.10.2]{Aschbacher:2008} in the proof of \cite[Lemma~2.6]{Henke:2018} could be replaced by a reference to Lemma~\ref{L:LocalNormalConstrainedSubsystems}(a). All other proofs in \cite{Henke:2018} rely solely on background results which are anyway used in the present paper.  
\end{remark}

The next proposition captures some of the main properties of $C_\F(\E)$.

\begin{prop}\label{P:DEcentralize}
 Let $\m{D}$ and $\E$ be normal subsystems of $\F$. Then the following are equivalent:
\begin{itemize}
\item [(i)] $\m{D}$ and $\E$ centralize each other;
\item [(ii)] $\m{D}\subseteq C_\F(\E)$;
\item [(iii)] $\m{D}\unlhd C_\F(\E)$.
\end{itemize}
\end{prop}

\begin{proof}
By Lemma~\ref{L:CentralizeEachOtherSaturated} and \cite[Theorem~2]{Henke:2018}, (i) and (ii) are equivalent. Clearly (iii) implies (ii). Thus, it remains only to prove that (i) and (ii) imply (iii). So assume that (i) and (ii) hold. Let $\m{D}$ be a subsystem over $R\leq S$ and let $\E$ be a subsystem over $T\leq S$. As $\m{D}$ is $\F$-invariant, it follows from the equivalent characterization of $\F$-invariant subsystems given in \cite[Proposition~I.6.4]{Aschbacher/Kessar/Oliver:2011} that $\m{D}$ is $C_\F(\E)$-invariant. 

\smallskip

As $\m{D}$ is saturated, it is thus sufficient to prove the extension condition for normal subsystems. More precisely, we need to show that every $\alpha\in\Aut_{\m{D}}(R)$ extends to $\hat{\alpha}\in\Aut_{C_\F(\E)}(RC_{C_S(\E)}(R))$ with $[RC_{C_S(\E)}(R),\hat{\alpha}]\leq R$. If $\alpha=c_r|_R\in\Inn(R)$ with $r\in R$, then $\hat{\alpha}=c_r|_{RC_{C_S(\E)}(R)}$ is such an extension of $\alpha$. As $\Inn(R)$ is a normal Sylow $p$-subgroup of $\Aut_{\m{D}}(R)$ by the saturation axioms, it is thus sufficient to prove the extension property for every  $p^\prime$-automorphism $\alpha\in\Aut_{\m{D}}(R)$. 

\smallskip

Fix a $p^\prime$-automorphism $\alpha\in\Aut_{\m{D}}(R)$. The assumption that $\m{D}$ is a normal subsystem of $\F$ implies the existence of an automorphism $\tilde{\alpha}\in\Aut_\F(RC_S(R))$ with $\tilde{\alpha}|_R=\alpha$ and $[RC_S(R),\tilde{\alpha}]\leq R$. Replacing $\tilde{\alpha}$ by a suitable power of $\tilde{\alpha}$ if necessary, we may assume that $\tilde{\alpha}$ is also a $p^\prime$-automorphism. As $T$ is strongly closed, we have $[T,\tilde{\alpha}]\leq R\cap T\leq Z(\m{D})$, where the last inclusion uses (i). Using the  property of coprime action stated e.g. in \cite[Lemma~A.2]{Aschbacher/Kessar/Oliver:2011}, we can conclude that $[T,\tilde{\alpha}]=[T,\tilde{\alpha},\tilde{\alpha}]\leq [Z(\m{D}),\tilde{\alpha}]=[Z(\m{D}),\alpha]=1$. Hence, $\tilde{\alpha}|_T=\id_T$, so in particular, $\hat{\alpha}=\tilde{\alpha}|_{RC_{C_S(\E)}(R)}$ is an element of $C_\F(T)$. As $\hat{\alpha}$ is a $p^\prime$-automorphism, it follows then from the definition of $C_\F(\E)$ above that $\hat{\alpha}$ is an element of $C_\F(\E)$. Clearly, $[RC_{C_S(\E)}(R),\hat{\alpha}]\leq [RC_S(R),\tilde{\alpha}]\leq R$, so the assertion follows. 
\end{proof}

The following lemma will be important in the proof of Theorem~\ref{main}. It seems interesting to remark that a somewhat similar result is shown in \cite[Lemma~8.2]{Henke:Regular} for partial normal subgroups of proper localities. It would imply Lemma~\ref{L:EcsubsetFq} if Theorem~\ref{main} was given. 

\begin{lemma}\label{L:EcsubsetFq}
Let $\E$ be a normal subsystem of $\F$ over $T\leq S$ such that $\hyp(C_\F(T))\leq T$ or $C_S(\E)\leq T$. Then $\E^c\subseteq\F^q$.
\end{lemma}

\begin{proof}
By \cite[Proposition~1]{Henke:2018}, $\hyp(C_\F(T))\leq C_S(\E)$. Hence, our assumption yields $\hyp(C_\F(T))\leq Z(T)$. So for $P\leq C_S(T)$ and a $p^\prime$-element $\phi\in \Aut_{C_\F(T)}(P)$, we have $[P,\phi]\leq Z(T)$ and thus $[P,\phi]=[P,\phi,\phi]=1$ by \cite[Lemma~A.2]{Aschbacher/Kessar/Oliver:2011}. Recall that $T$ is strongly closed and so, in particular, $T$ is fully centralized and $C_\F(T)$ is saturated. Thus, Alperin's Fusion Theorem \cite[Theorem~I.3.6]{Aschbacher/Kessar/Oliver:2011} yields that $C_\F(T)$ is the fusion system of a $p$-group and $T\in\F^q$.

\smallskip

Assume now that there exists $U\in \E^c$ with $U\not\in\F^q$. Choose such $U$ of maximal order. By Lemma~\ref{L:EcFinvariant}, $\E^c$ is  closed under $\F$-conjugacy. The set $\F^q$ is by definition closed under $\F$-conjugacy. So we may assume $U\in\F^f$. As $T\in\F^q$, we have $U<T$ and thus $U<P:=N_T(U)$. Observe that $P\in\E^c$. It follows therefore from the maximality of $|U|$ that $P\in\F^q$.  

\smallskip

As $U\in\F^f$ and $U\not\in\F^q$, the centralizer $C_\F(U)$ is not the fusion system of a $p$-group. Hence, by Alperin's fusion theorem, there exists $V\in C_\F(U)^c$ such that $\Aut_{C_\F(U)}(V)$ is not a $p$-group. One observes easily that $C_\F(U)$ is weakly normal in $N_\F(U)$. Thus, using  Lemma~\ref{L:EcFinvariant}, we see that we can replace $V$ by any $N_\F(U)$-conjugate of $V$. Therefore, we may assume that $UV\in N_\F(U)^f$. Then by Lemma~\ref{L:LocalNormalConstrainedSubsystemsV}, $N_\F(UV)=N_{N_\F(U)}(UV)$ is a constrained fusion system and $N_\E(U)$ is normal in $N_\F(UV)$. Hence, using Theorem~III.5.10 and Theorem~II.7.5 in \cite{Aschbacher/Kessar/Oliver:2011} in combination with Lemma~\ref{L:MSCharp}(a), we see that there exists a model $G$ of $N_\F(UV)$ and a normal subgroup $N$ of $G$ which is a model for $N_\E(U)$. As $U$ is a normal centric subgroup of $N_\E(U)$, using \cite[Theorem~III.5.10]{Aschbacher/Kessar/Oliver:2011} again, we see that $U\unlhd N$ and $C_N(U)\leq U$. Since $P=N_T(U)\leq N$, it follows that 
\[[P,C_G(U)]\leq [N,C_G(U)]\leq C_N(U)=Z(U)\leq P.\]
Hence, $C_G(U)$ normalizes $P$ and 
\[[P,O^p(C_G(U))]=[P,O^p(C_G(U)),O^p(C_G(U))]\leq [Z(U),C_G(U)]=1.\]
This shows $O^p(C_G(U))\leq C_G(P)$. 

\smallskip

Notice that $P$ is normal in the Sylow $p$-subgroup $S_0:=N_S(UV)=N_{N_S(U)}(UV)$ of $G$. Hence, $C_{S_0}(P)$ is a Sylow $p$-subgroup of $C_G(P)$ and $\F_{C_{S_0}(P)}(C_G(P))$ is a saturated subsystem of $C_\F(P)$. As $P\in\F^q$, it follows from  Lemma~\ref{L:QuasicentricApply} that $\F_{C_{S_0}(P)}(C_G(P))=\F_{C_{S_0}(P)}(C_{S_0}(P))$. Thus, a theorem of Frobenius \cite[Theorem~1.4]{Linckelmann:2007a} yields that $C_G(P)=C_{S_0}(P)O_{p^\prime}(C_G(P))$. On the other hand, as $G$ is of characteristic $p$, Lemma~\ref{L:MSCharp}(b) gives that $C_G(P)$ is of characteristic $p$, which implies $O_{p^\prime}(C_G(P))=1$. Hence, $C_G(P)=C_{S_0}(P)$ is a $p$-group. As $O^p(C_G(U))\leq C_G(P)$, it follows that $C_G(U)$ is a $p$-group. However, every element of $\Aut_{C_\F(U)}(V)$ is a morphism in $N_\F(UV)$ centralizing $U$ and thus realized by conjugation with an element of $C_G(U)$. This contradicts the assumption that $\Aut_{C_\F(U)}(V)$ is not a $p$-group.  
\end{proof}

In the proof of Theorem~\ref{main}, we will use Lemma~\ref{L:EcsubsetFq} in the form of the following lemma.

\begin{lemma}\label{L:tE}
Let $\E$ be a normal subsystem of $\F$ over $T\leq S$. Then $\E$ and $C_\F(\E)$ centralize each other in $\F$. Moreover, setting $\tE:=\E *C_\F(\E)$, the following hold:
\begin{itemize}
 \item [(a)] $\tE$ is a normal subsystem of $\F$. 
 \item [(b)] $\tE^{cr}=\{P_1P_2\colon P_1\in\E^{cr},\;P_2\in C_\F(\E)^{cr}\}$.
 \item [(c)] $C_S(\tE)\subseteq C_S(\E)$ and $\tE^{cr}\subseteq\tE^c\subseteq\F^q$.
\end{itemize}
\end{lemma}

\begin{proof}
\textbf{(a,b)} It follows from Proposition~\ref{P:DEcentralize} or from a direct argument using Lemma~\ref{L:CentralizeEachOtherSaturated} that $\E$ and $C_\F(\E)$ centralize each other in $\F$. In particular, (a) holds by Lemma~\ref{L:CharacterizeF1starF2}. It follows now from Lemma~\ref{L:InternalExternalCentralProduct} that $\tE$ is an internal central product of $\E$ and $C_\F(\E)$ not only in our definition above, but also in the sense of \cite[Definition~2.9]{Henke:2016}. Hence, (b) follows from \cite[Lemma~2.10(a)]{Henke:2016}. 

\smallskip

\textbf{(c)} Note that $\E\subseteq \tE\subseteq C_\F(C_S(\tE))$. Thus, by the characterization of $C_S(\E)$ as the largest subgroup of $S$  containing $\E$ in its centralizer (cf. \cite[Theorem~1]{Henke:2018}), we have $C_S(\tE)\subseteq C_S(\E)\subseteq TC_S(\E)$. As $\tE$ is a subsystem over $TC_S(\E)$, Lemma~\ref{L:EcsubsetFq} gives therefore that $\tE^{cr}\subseteq\tE^c\subseteq\F^q$. 
\end{proof}

\section{Partial groups and localities}\label{LocalitiesSection}

The purpose of this section is to recall some basic definitions and background results on partial groups and localities, and to prove some more specialized results needed in the proofs of our main theorems.

\subsection{Partial groups} 

For any set $X$ write $\W(X)$ for the free monoid on $X$. Thus, an element 
of $\W(X)$ is a finite sequence of (or \emph{word in}) the elements of $X$, and the multiplication in 
$\W(X)$ consists of concatenation of words. The concatenation of words $v_1,\dots,v_k\in\W(X)$ is denoted by $v_1\circ v_2\circ\dots\circ v_k$. The \emph{length} of the word 
$(x_1,\cdots,x_n)$ is $n$. The \emph{empty word} is the word $\emptyset$ of length 0. We make no 
distinction between $X$ and the set of words of length $1$. 

\begin{definition}
Let $\L$ be a non-empty set, and let $\D$ be a subset of $\W(\L)$ such that: 
\begin{itemize}
\item [(1)] $\L\subseteq\D$ (i.e. $\D$ contains all words of length 1), and 
\[u\circ v\in\bold D\implies u,v\in\bold D. \]
\end{itemize}
Notice that since $\L$ is non-empty, (1) implies that also the empty word is in $\D$. A mapping $\Pi\colon\D\to\L$ is a \emph{(partial) product} if:  
\begin{itemize}
 \item [(2)] $\Pi$ restricts to the identity map on $\L$, and 
 \item [(3)] $u\circ v\circ w\in\D\Longrightarrow u\circ\Pi(v)\circ w\in\D$, and 
$\Pi(u\circ v\circ w)=\Pi(u\circ\Pi(v)\circ w)$. 
\end{itemize}
An \emph{inversion} on $\L$ consists of an involutory bijection $x\mapsto x^{-1}$ on $\L$, 
together with the mapping $w\mapsto w^{-1}$ on $\W(\L)$ given by 
\[(x_1,\cdots,x_n)\mapsto(x_n^{-1},\cdots x_1^{-1}).\]
We say that $\L$, with the product $\Pi\colon\D\rightarrow\L$ and inversion $(-)^{-1}$, is a 
\emph{partial group} if: 
\begin{itemize}
\item [(4)] $w\in\bold D\Longrightarrow w^{-1}\circ w\in\bold D$ and $\Pi(w^{-1}\circ w)=\One$, 
\end{itemize}
where $\One$ denotes the image of the empty word under $\Pi$. Notice that (1) and (4) yield 
$w^{-1}\in\D$ if $w\in\D$. As $(w^{-1})^{-1}=w$, condition (4) is symmetric. 
\end{definition}

\textbf{For the remainder of this section let $\L$ be a partial group with product $\Pi\colon \D\rightarrow \L$ defined on the domain $\D\subseteq\W(\L)$.} 

\smallskip

As above we will always write $\One$ for the image of the empty word under $\Pi$. Moreover, given a word $v=(f_1,\dots,f_n)\in\D$, we write sometimes $f_1f_2\dots f_n$ for the product $\Pi(v)$.

\smallskip

A \emph{partial subgroup} of $\L$ is a subset $\H$ of $\L$ such that $f^{-1}\in\H$ for all $f\in\H$ and $\Pi(w)\in\H$ for all $w\in\W(\H)\cap\D$. Note that, for any partial subgroup $\H$ of $\L$, we have $\emptyset\in\W(\H)\cap\D$ and thus $\One=\Pi(\emptyset)\in\H$. It is easy to see that a partial subgroup of $\L$ is always a partial group itself whose product is the restriction of the product $\Pi$ to $\W(\H)\cap\D$. Observe furthermore that $\L$ forms a group in the usual sense if $\W(\L)=\D$; see \cite[Lemma~1.3]{Chermak:2015}. So it makes sense to call a partial subgroup $\H$ of $\L$ a \emph{subgroup of $\L$} if $\W(\H)\subseteq\D$. In particular, we can talk about \emph{$p$-subgroups of $\L$} meaning subgroups of $\L$ whose order is a power of $p$.

\begin{notation}
For any $g\in\L$, the set of elements $x\in\L$ with $(g^{-1},x,g)\in\D$ is denoted by $\D(g)$. Thus, $\D(g)$ is the set of elements $x\in\L$ for which the conjugate $x^g:=\Pi(g^{-1},x,g)$ is defined. Given $g\in\L$ and $X\subseteq \D(g)$, we set $X^g:=\{x^g\colon x\in X\}$.  
\end{notation}

\begin{lemma}\cite[Lemma~1.6(c)]{Chermak:2015}\label{L:cgInverse}
For every $g\in\L$, the map $c_g\colon \D(g)\rightarrow \D(g^{-1}),\;x\mapsto x^g$ is well-defined and a bijection with inverse map $c_{g^{-1}}$.
\end{lemma}

\begin{definition}
Define the \emph{normalizer} and the \emph{centralizer} of a subset $X$ of $\L$ by 
\[N_\L(X):=\{g\in\L\colon X\subseteq\D(g)\mbox{ and }X^g=X\}\]
and
\[C_\L(X):=\{g\in\L\colon X\subseteq\D(g)\mbox{ and }x^g=x\mbox{ for all }x\in X\}.\]
If $\H\subseteq\L$ and $X\subseteq\L$, then set $N_\H(X)=N_\L(X)\cap\H$ and $C_\H(X)=C_\L(X)\cap\H$. Set moreover
\[Z(\L):=C_\L(\L).\]
\end{definition}

\begin{definition}
\begin{itemize}
\item A partial subgroup $\N$ of $\L$ is called a \emph{partial normal subgroup} of $\L$ if, for all $g\in\L$ and all $n\in\N\cap\D(g)$, we have $n^g\in\N$. We write $\N\unlhd\L$ to indicate that $\N$ is a partial normal subgroup of $\L$. 
\item A partial subgroup $\H$ of $\L$ is called a \emph{partial subnormal subgroup} if there exists a sequence  $\H=\H_0\unlhd\H_1\unlhd\dots\unlhd\H_n=\L$, which is then called a \emph{subnormal series} of $\H$ in $\L$. Write $\H\subn\L$ to indicate that $\H$ is subnormal in $\L$. 
\end{itemize}
\end{definition}

If $\L$ is a group (i.e. if $\W(\L)=\D$), then notice that every partial subgroup of $\L$ is a subgroup, every partial normal subgroup of $\L$ is a normal subgroup, and every partial subnormal subgroup of $\L$ is a subnormal subgroup in the usual sense.

\begin{lemma}\label{L:SubnormalBasic}
\begin{itemize}
 \item [(a)] If $\K\subn\H$ and $\H\subn\L$, then $\K\subn\L$. 
 \item [(b)] If $\H$ and $\K$ are partial subnormal subgroups of $\L$, then $\H\cap\K$ is a partial subnormal subgroup of $\L$. Hence, the intersection of a finite number of partial subnormal subgroups is a partial subnormal subgroup.
 \item [(c)] If $\H$ is a partial subgroup of $\L$ and $\K\subn\L$, then $\K\cap \H\subn\H$. In particular, if in addition $\K\subseteq\H$, then $\K\subn\H$.  
\end{itemize}
\end{lemma}

\begin{proof}
Clearly (a) holds. Let $\H$ be a partial subgroup of $\L$ and $\K\subn\L$. Then there exists a subnormal series $\K=\K_0\unlhd\K_1\unlhd\dots\unlhd\K_n=\L$. Observe that \[\K\cap\H=\K_0\cap\H\unlhd\K_1\cap\H\unlhd\dots\unlhd\K_n\cap\H=\L\cap\H=\H\]
is a subnormal series of $\K\cap\H$ in $\H$. So  $\K\cap\H\subn\H$ and (c) holds. If $\H$ is also subnormal in $\L$, then it follows thus from (a) that $\K\cap\H\subn\L$ and (b) holds.
\end{proof}

\subsection{Localities}

Localities are partial groups with some additional structure. To give the precise definition we will use the following notation. 

\begin{notation}
Let $\Delta$ be a set of subgroups of $\L$. Write $\D_\Delta$ for the set of words $(f_1,\dots,f_n)\in\W(\L)$ for which there exist $P_0,\dots,P_n\in\Delta$ with 
\begin{itemize}
\item [(*)] $P_{i-1}\subseteq \D(f_i)$ and $P_{i-1}^{f_i}=P_i$ for all $i=1,\dots,n$.
\end{itemize}
If $v=(f_1,\dots,f_n)\in\W(\L)$, then we say that $v\in\D_\Delta$ via $P_0,\dots,P_n$ (or $v\in\D_\Delta$ via $P_0$), if $P_0,\dots,P_n\in\Delta$ and (*) holds.
\end{notation}

The following definition of a locality was given in \cite[Definition~5.1]{Henke:2015}. As argued in the same paper in Remark~5.2, it is equivalent to the definition of a locality given in \cite[Definition~2.9]{Chermak:2013} and \cite[Definition~2.7]{Chermak:2015}.

\begin{definition}\label{LocalityDefinition}
The triple $(\L,\Delta,S)$ is called a \emph{locality} if the partial group $\L$ is finite as a set, $S$ is a $p$-subgroup of $\L$ and $\Delta$ is a non-empty set of subgroups of $S$ such that the following conditions hold:
\begin{itemize}
\item[(L1)] $S$ is maximal with respect to inclusion among the $p$-subgroups of $\L$.
\item[(L2)] $\D=\D_\Delta$.
\item[(L3)] The set $\Delta$ is closed under passing to $\L$-conjugates and overgroups in $S$.
\end{itemize}
If $(\L,\Delta,S)$ is a locality and $v=(f_1,\dots,f_n)\in\W(\L)$, then we say that $v\in\D$ via $P_0,\dots,P_n$ (or $v\in\D$ via $P_0$), if $v\in\D_\Delta$ via $P_0,\dots,P_n$.
\end{definition}

The condition that $\Delta$ is \emph{closed under passing to $\L$-conjugates in $S$} means here more precisely that, given an element $P\in\Delta$, we have $P^g\in\Delta$ for every $g\in\L$ with $P\subseteq \D(g)$ and $P^g\subseteq S$. As we will see next, many natural examples of localities can be constructed from groups.

\begin{example}\label{LGammaM}
 Let $M$ be a finite group and $S\in\Syl_p(M)$. Let $\Gamma$ be a non-empty collection of subgroups of $S$ such that $\Gamma$ is $\F_S(M)$-closed in the sense of Definition~\ref{D:Fclosed}. Set 
\[\L_\Gamma(M):=\{g\in G\colon S\cap S^g\in \Gamma\}=\{g\in G\colon \mbox{There exists } P\in\Gamma\mbox{ with }P^g\leq S\}.\]
Let $\D$ be the set of tuples $(g_1,\dots,g_n)\in\W(M)$ such that there exist $P_0,P_1,\dots,P_n\in\Gamma$ with $P_{i-1}^{g_i}=P_i$ for all $i=1,\dots,n$. Then $\L_\Gamma(M)$ forms a partial group whose product is the restriction of the multivariable product on $M$ to $\D$, and whose inversion map is the restriction of the inversion map on the group $M$ to $\L_\Gamma(M)$. Moreover, $(\L_\Gamma(M),\Gamma,S)$ forms a locality. 
\end{example}

\begin{proof}
 See \cite[Example/Lemma~2.10]{Chermak:2013}.
\end{proof}

\begin{notation}
Let $(\L,\Delta,S)$ be a locality. For any $g\in\L$ set 
\[S_g:=\{s\in\D(g)\colon s^g\in S\}.\]
More generally, if $w=(g_1,\dots,g_n)\in\W(\L)$, write $S_w$ for the set of all $s\in S$ such that there exists a word $(s_0,s_1,\dots,s_n)\in\W(S)$ with $s_0=s$, $s_{i-1}\in\D(g_i)$ and $s_{i-1}^{g_i}=s_i$ for all $i=1,\dots,n$. 
\end{notation}

We will now give a summary of some basic properties of localities. 

\begin{lemma}[Important properties of localities]\label{LocalitiesProp}
Suppose $(\L,\Delta,S)$ is a locality. The following hold:
\begin{itemize}
\item [(a)] If $P\in\Delta$, then $N_\L(P)$ is a subgroup of $\L$. Moreover, $N_\N(P)$ is a normal subgroup of $N_\L(P)$ for any partial normal subgroup $\N$ of $\L$.   
\item [(b)] If $P\in\Delta$ and $g\in\L$ with $P\subseteq S_g$, then $Q:=P^g\in\Delta$. Moreover, $N_\L(P)\subseteq \D(g)$ and 
\[c_g\colon N_\L(P)\rightarrow N_\L(Q),\;x\mapsto x^g\]
is an isomorphism of groups. For any partial normal subgroup $\N$, the map $c_g$ induces an isomorphism of groups from $N_\N(P)$ to $N_\N(Q)$. For $x\in N_\L(P)$, 
\[(c_g|_P)^{-1}(c_x|_P)(c_g|_P)=c_{x^g}|_P.\]
In particular, setting $\Aut_\N(X):=\{c_n|_X\colon n\in N_\N(X)\}$ for every $X\in\Delta$,   
\[(c_g|_P)^{-1}\Aut_\N(P)(c_g|_P)=\Aut_\N(Q).\] 
\item [(c)] Let $w=(g_1,\dots,g_n)\in\D$ via $(X_0,\dots,X_n)$. Then 
$$c_{g_1}\circ \dots \circ c_{g_n}=c_{\Pi(w)}$$
 is a group isomorphism $N_\L(X_0)\rightarrow N_\L(X_n)$.
\item [(d)] For every $g\in\L$, $S_g\in\Delta$. In particular, $S_g$ is a subgroup of $S$. Moreover,  $S_g^g=S_{g^{-1}}$.
\item [(e)] Let $w\in\W(\L)$. Then $S_w$ is a subgroup of $S$. Moreover, $S_w\in\Delta$ if and only if $w\in\D$. If so then $S_w\leq S_{\Pi(w)}$.
\end{itemize}
\end{lemma}

\begin{proof}
\textbf{(a,c)} The first part of (a) and property (c) correspond to the statements (a) and (c) in \cite[Lemma~2.3]{Chermak:2015}. It is easy to see that the intersection of a partial normal subgroup with a subgroup $\H$ is always a normal subgroup of $\H$, so the second part of (a) holds as well.

\smallskip

\textbf{(b)} By \cite[Lemma~2.3(b)]{Chermak:2015}, $N_\L(P)\subseteq \D(g)$ and 
$$c_g\colon N_\L(P)\rightarrow N_\L(Q)$$
is an isomorphism of groups provided $Q:=P^g\in\Delta$. With the definition of a locality given above, it is however immediate from (L3) that $Q\in\Delta$, so the first part of (b) holds. Let now $\N$ be a partial normal subgroup of $\L$. Then $N_\N(P)^g\subseteq N_\L(Q)\cap \N=N_\N(Q)$. By (d), $Q\leq S_{g^{-1}}$ and $Q^{g^{-1}}=P$, so a symmetric argument gives $N_\N(Q)^{g^{-1}}\subseteq N_\N(P)$. Using Lemma~\ref{L:cgInverse}, conjugating on both sides by $g$ gives $N_\N(Q)\subseteq N_\N(P)^g$. Hence, $N_\N(Q)=N_\N(P)^g$ as required. For every $x\in N_\L(P)$, we have $(g^{-1},x,g)\in\D$ via $Q,P,P,Q$. Therefore, it follows from (c) and Lemma~\ref{L:cgInverse} that $(c_g|_P)^{-1}(c_x|_P)(c_g|_P)=(c_{g^{-1}}|_Q)(c_x|_P)(c_g|_P)=c_{x^g}|_Q$. In particular, since we have seen that the elements of $N_\N(Q)$ are precisely the elements of the form $x^g$ with $x\in N_\N(P)$, this shows $(c_g|_P)^{-1}\Aut_\N(P)(c_g|_P)=\Aut_\N(Q)$. Hence, (b) holds.

\smallskip

\textbf{(d)} This is true by \cite[Proposition~2.5(a),(b)]{Chermak:2015}.

\smallskip

\textbf{(e)} If $w\in\W(\L)$, then it is shown in \cite[Corollary~2.6]{Chermak:2015} that $S_w\leq S$ and $S_w\in\Delta$ if and only if $w\in\D$. It follows from (c) that $S_w\subseteq S_{\Pi(w)}$ if $w\in\D$. Hence (e) holds.
\end{proof}

We will use from now on without further reference that, for every locality $(\L,\Delta,S)$ and every $g\in\L$, the subset $S_g$ is an element of $\Delta$ and thus a subgroup of $S$. Moreover, $c_g\colon S_g\rightarrow S,x\mapsto x^g$ is an injective group homomorphism. We will also use that, for every $P\in\Delta$, the normalizer $N_\L(P)$ is a subgroup of $\L$, and that $P^g\in\Delta$ for every $g\in\L$ with $P\subseteq S_g$.

\begin{definition}\label{D:RestrictionLocality}
Let $(\L^+,\Delta^+,S)$ be a locality with a partial product $\Pi^+\colon\D^+\longrightarrow\L^+$. Suppose that  $\emptyset\neq \Delta\subseteq \Delta^+$ such that $\Delta$ is closed with respect to taking $\L^+$-conjugates and overgroups in $S$. Set 
\[\L^+|_{\Delta}:=\{f\in\L^+\colon S_f\in\Delta\}.\]
Note that $\D:=\D_{\Delta}\subseteq\D^+\cap\W(\L^+|_{\Delta})$ and, by Lemma~\ref{LocalitiesProp}(c), $\Pi^+(w)\in\L|_{\Delta}$ for all $w\in\D$. We call $\L:=\L^+|_{\Delta}$ together with the partial product $\Pi^+|_{\D}\colon\D\longrightarrow \L$ and the restriction of the inversion map on $\L^+$ to $\L$ the \emph{restriction} of $\L^+$ to $\Delta$.
\end{definition}

With the hypothesis as in the preceding definition, it turns out that the restriction of $\L^+$ to $\Delta$ forms a partial group. Indeed, the triple $(\L^+|_\Delta,\Delta,S)$ is a locality (cf. \cite[Lemma~2.21(a)]{Chermak:2013} and \cite[Lemma~2.23(a),(c)]{Henke:2020}). The reader should note that, in the definition of the restriction, $S_f$ and $\D_\Delta$ are a priori formed inside of $\L^+$. However, as argued in \cite[Lemma~2.23(b)]{Henke:2020}, it does not matter whether one forms $S_f$ and $\D_\Delta$ inside of $\L^+$ or inside of the partial group $\L^+|_\Delta$.

\smallskip

We repeat the following definition from \cite[Lemma~2.13]{Chermak:2015}.

\begin{definition}
If $(\L,\Delta,S)$ is a locality, then set
\[O_p(\L):=\bigcap\{S_w\colon w\in\W(\L)\}\] 
\end{definition}

\begin{lemma}\label{L:OpL}
If $(\L,\Delta,S)$ is a locality, then $O_p(\L)$ is the unique largest subgroup of $S$ which is a partial normal subgroup of $\L$. Moreover, a subgroup $P\leq S$ is a partial normal subgroup of $\L$ if and only if  $N_\L(P)=\L$.
\end{lemma}

\begin{proof}
This is proved in \cite[Lemma~3.13]{Henke:Regular} based on \cite[Lemma~2.13]{Chermak:2015}.
\end{proof}

\subsection{Fusion systems of localities}

Given a locality $(\L,\Delta,S)$, a fusion system $\F_S(\L)$ is naturally defined. In fact, we have the following more general definition. 

\begin{definition}\label{D:F(H)}
Let $(\L,\Delta,S)$ be a locality and $\H$ a partial subgroup of $\L$. By $\F_{S\cap \H}(\H)$ we denote the fusion system over $S\cap\H$ generated by all the maps of the form \[c_h\colon S_h\cap\H\rightarrow S\cap\H,x\mapsto x^h\] with $h\in\H$.  In particular, $\F_S(\L)$ is the fusion system over $S$ generated by the maps $c_g\colon S_g\rightarrow S$ with $g\in\L$. If $\F=\F_S(\L)$, then $(\L,\Delta,S)$ is said to be a \emph{locality over $\F$}.
\end{definition}

Note that we use implicitly in the above definition that, by Lemma~\ref{LocalitiesProp}(b),(d), for any $g\in\L$, the subset $S_g$ is a subgroup in $\Delta$ and $c_g$ induces an injective group homomorphism from $S_g$ to $S$. In particular, if $\H$ is a partial subgroup, then for any $h\in\H$, $S_h\cap\H$ is a subgroup of $S$ and $c_h$ induces an injective group homomorphism from $S_h\cap \H$ to $S\cap \H$.

\begin{definition}\label{D:Fnatural}
Let $\F$ be a fusion system over $S$ and let $\Delta$ be a set of subgroups of $S$.
\begin{itemize}  
\item Let $\F|_\Delta$ denote the subsystem of $\F$ generated by all the morphism sets $\Hom_\F(P,Q)$ with  $P,Q\in\Delta$. 
\item A locality $(\L,\Delta,S)$ is said to be \emph{$\F$-natural} if $\Delta$ is $\F$-closed (in the sense of Definition~\ref{D:Fclosed}) and $\F_S(\L)=\F|_{\Delta}$. 
 \end{itemize}
\end{definition}

If $(\L,\Delta,S)$ is a locality, then notice that $\Delta$ is $\F_S(\L)$-closed by axiom (L3) in Definition~\ref{LocalityDefinition}. Moreover, as $S_g\in\Delta$ for all $g\in\L$ by Lemma~\ref{LocalitiesProp}(d), we have $\F_S(\L)|_\Delta=\F_S(\L)$. Hence $(\L,\Delta,S)$ is $\F_S(\L)$-natural. We have the following lemma.

\begin{lemma}\label{L:LocalitiesPropg}
Suppose $\F$ is a fusion system. If $(\L,\Delta,S)$ is a locality over $\F$ or, more generally, an $\F$-natural locality, then the following hold:
\begin{itemize}
\item [(a)] If $\phi\in\Hom_\F(P,R)$ for some object $P\in\Delta$ and some subgroup $R\leq S$, then $R\in\Delta$ and there exists $g\in\L$ with $\phi=c_g|_P$. In particular $\Delta$ is $\F$-closed.
\item [(b)] If an object $Q\in\Delta$ is fully $\F$-normalized, then $N_S(Q)\in\Syl_p(N_\L(Q))$. 
\item [(c)] $N_\F(P)=\F_{N_S(P)}(N_\L(P))$ for every $P\in\Delta$. 
\end{itemize}
\end{lemma}

\begin{proof}
Since $\F|_\Delta=\F_S(\L)$, property (a) follows from Lemma~\ref{LocalitiesProp}(c) and the fact that $\Delta$ is closed under passing to $\L$-conjugates and overgroups in $S$. Property (b) is proved in \cite[Proposition~2.18(c)]{Chermak:2013}. Using (a) and the fact that $\Delta$ is overgroup-closed in $S$, one sees easily that (c) holds.  
\end{proof}

\subsection{Products in localities}

Given subsets $\M$ and $\N$ of a partial group $\L$ with product $\Pi\colon\D\rightarrow\L$, we set
\[\M\N:=\{\Pi(m,n)\colon m\in\M,\;n\in\N,\;(m,n)\in\D\}.\]
More generally, if $\N_1,\dots,\N_k\subseteq\L$, set
\[\N_1\N_2\cdots\N_k:=\{\Pi(n_1,\dots,n_k)\colon n_i\in\N_i\mbox{ for }i=1,\dots,k\mbox{ and }(n_1,\dots,n_k)\in\D\}.\]

\begin{remark}\label{R:MandNinMN}
 If $\M$ and $\N$ are partial groups (or more generally, if $\One\in\M\cap\N$), then $\M$ and $\N$ are both contained in $\M\N$. To see this one uses that $\Pi(f,\One)=f=\Pi(\One,f)$ for alle $f\in\L$ by \cite[Lemma~1.4(c)]{Chermak:2015}.
\end{remark}

The most important properties of products of partial normal subgroups are summarized in the following theorem.

\begin{theorem}\label{T:ProductsPartialNormal}
Let $(\L,\Delta,S)$ be a locality and let $\N_1,\dots,\N_k$ be partial normal subgroups of $\L$. Then the following hold:
\begin{itemize}
 \item [(a)] $\N_1\N_2\cdots\N_k$ is a partial normal subgroup of $\L$ with 
\[(\N_1\N_2\cdots\N_k)\cap S=(\N_1\cap S)(\N_2\cap S)\cdots(\N_k\cap S).\]
 \item [(b)] $\N_1\N_2\cdots \N_k=(\N_1\cdots \N_l)(\N_{l+1}\cdots\N_k)$ for every $1\leq l<k$.
 \item [(c)] $\N_1\N_2\cdots \N_k=\N_{1\sigma}\N_{2\sigma}\cdots\N_{k\sigma}$ for every permutation $\sigma\in S_k$.
 \item [(d)] For every $g\in \N_1\dots\N_k$ there exists $(n_1,\dots,n_k)\in\D$ with $n_i\in\N_i$ for every $i=1,\dots,k$, $g=\Pi(n_1\dots n_k)$ and $S_g=S_{(n_1,\dots,n_k)}$.  
\end{itemize}
\end{theorem}

\begin{proof}
 By \cite[Theorem~2]{Henke:2015a}, $\N_1\N_2\cdots\N_k\unlhd\L$ and parts (b),(c),(d) hold. If $\M,\N\unlhd\L$, then \cite[Theorem~1]{Henke:2015a} says $(\M\N)\cap S=(\M\cap S)(\N\cap S)$. So it follows from (b) and induction on $k$ that the second statement in (a) holds. 
\end{proof}

\begin{corollary}\label{C:ProductRestrictions}
 Let $(\L^+,\Delta^+,S)$ and $(\L,\Delta,S)$ be localities such that $\L=\L^+|_\Delta$. Let $\N_i^+\unlhd\L^+$ and $\N_i:=\N_i^+\cap\L$ for $i=1,\dots,k$. Then
\[(\N_1^+\N_2^+\cdots\N_k^+)\cap \L=\N_1\N_2\cdots\N_k.\]
where $\N_1^+\N_2^+\cdots\N_k^+$ denotes the product of $\N_1^+,\N_2^+,\dots,\N_k^+$ in $\L^+$.
\end{corollary}

\begin{proof}
Let $\Pi^+\colon\D^+\rightarrow\L^+$ and $\Pi\colon\D\rightarrow \L$ denote the partial products on $\L^+$ and $\L$ respectively. Observe that $\N_1\N_2\cdots\N_k\subseteq (\N_1^+\N_2^+\cdots\N_k^+)\cap \L$, as $\N_i\subseteq\N_i^+$ for $i=1,2,\dots,k$ and $\Pi=\Pi^+|_\D$. Let now $f\in (\N_1^+\N_2^+\cdots\N_k^+)\cap \L$. Then by Theorem~\ref{T:ProductsPartialNormal}(d), there exists $u=(n_1,\dots,n_k)\in (\N_1^+\times\N_2^+\times\cdots\times\N_k^+)\cap\D^+$ with $f=\Pi^+(u)$ and $S_f=S_u$. Since $f\in\L$, it follows $S_u=S_f\in\Delta$ and thus $u\in\D$. In particular, $u\in\W(\L)$ and thus $n_i\in\N_i^+\cap\L=\N_i$ for $i=1,2,\dots,n$. As $\Pi^+|_\D=\Pi$, it follows $f=\Pi(u)\in\N_1\N_2\cdots\N_k$.
\end{proof}

We will now look at products of partial normal subgroups with subgroups of $S$. The following general lemma will be useful.

\begin{lemma}[{\cite[Lemma~2.8]{Henke:2020}}]\label{L:NLSbiset}
If $r\in N_\L(S)$ and $f\in\L$, then $(r,f)$, $(f,r)$ and $(r^{-1},f,r)$ are words in $\D$. Moreover,
\[S_{(f,r)}=S_{fr}=S_f,\;S_{(r,f)}=S_{rf}=S_f^{r^{-1}}\mbox{ and }S_{f^r}=S_f^r.\]
\end{lemma}

Notice that the lemma above implies $R\N=\{\Pi(r,n)\colon r\in R,\;n\in\N\}$ and $\N R=\{\Pi(n,r)\colon n\in\N,\;r\in R\}$ for $\N\subseteq\L$ and $R\subseteq N_\L(S)$.

\begin{lemma}\label{ProductSubnormal}
Let $\N\unlhd\L$ and $R\leq S$. Then $\N R=R\N$ is a partial subgroup with $(\N R)\cap S=(\N\cap S)R$. Moreover, if $R_0\unlhd R\leq S$, then $\N R_0\unlhd \N R$. 
\end{lemma}

\begin{proof}
By \cite[Lemma~3.15]{Henke:Regular} $\N R=R\N$ is a partial subgroup of $\L$ and by the Dedekind lemma \cite[Lemma~1.10]{Chermak:2015}, we have $(\N R)\cap S=(\N\cap S)R$. To prove the second part of the claim, let $R_0\unlhd R\leq S$, $n\in \N$, $r\in R$ and $x\in \N R_0$ such that $u=((nr)^{-1},x,(nr))\in\D$. By Lemma~\ref{L:NLSbiset}, $S_{nr}=S_{(n,r)}$ and $S_{(nr)^{-1}}=S_{r^{-1}n^{-1}}=S_{(r^{-1},n^{-1})}$. Hence, $v=(r^{-1},n^{-1},x,n,r)\in\D$ via $S_u$ and $x^{nr}=\Pi(u)=\Pi(v)=\Pi(r^{-1},x^n,r)$. As $\N R_0$ is a partial subgroup, we have $x^n\in \N R_0$, i.e. there exist $m\in\N$ and $r_0\in R_0$ such that $x^n=mr_0$. Notice that $w=(r^{-1},m,r,r^{-1},r_0,r)\in\D$ via $S_m^r$. Hence, $x^{nr}=\Pi(r^{-1},m,r_0,r)=\Pi(w)=\Pi(m^r,r_0^r)\in\N R_0$ since $\N\unlhd \L$ and $R_0\unlhd R$. 
\end{proof}

\subsection{Homomorphisms of partial groups}\label{SS:Homomorphisms}
After introducing some basic definitions, we will outline in this subsection how partial normal subgroups correspond to kernels of homomorphisms of partial groups. This connection will play a crucial role in the proof of our main theorem. 

\begin{definition}\label{D:IsoLoc}
Let $\L$ and $\L'$ be partial groups with products $\Pi\colon\D\rightarrow \L$ and $\Pi'\colon \D'\rightarrow \L'$ respectively. Let $\alpha\colon \L\rightarrow\L'$ be a map.
\begin{itemize}
\item Write $\alpha^*$ for the map 
\[\alpha^*\colon \W(\L)\rightarrow \W(\L'), (f_1,\dots,f_n)\mapsto(f_1\alpha,\dots,f_n\alpha)\]
induced by $\alpha$. 
\item The map $\alpha$ is called a \emph{homomorphism of partial groups} if $\D\alpha^*\subseteq\D'$ and $\Pi(w)\alpha=\Pi'(w\alpha^*)$. 
\item If $\alpha$ is a homomorphism of partial groups such that $\D\alpha^*=\D'$, then $\alpha$ is called a \emph{projection of partial groups}. If $\alpha$ is in addition bijective, then $\alpha$ is called an \emph{isomorphism of partial groups}.
\item An isomorphism of partial groups from $\L$ to itself is called an \emph{automorphism} of $\L$. Write $\Aut(\L)$ for the set of automorphisms of $\L$. If $S\subseteq\L$ write $\Aut(\L,S)$ for the set of all $\alpha\in\Aut(\L)$ with $S\alpha=S$.
\item Suppose $\L$ and $\L'$ contain both the same subgroup $S$ and, for some set $\Delta$, the triples $(\L,\Delta,S)$ and $(\L',\Delta,S)$ are localities. Then a \emph{rigid isomorphism} from $(\L,\Delta,S)$ to $(\L',\Delta,S)$ is an isomorphism $\alpha\colon \L\rightarrow\L'$ of partial groups such that $\alpha$ restricts to the identity on $S$.
\item If $\alpha\colon\L\rightarrow\L'$ is a homomorphism of partial groups and $\One'$ denotes the identity in $\L'$, then
\[\ker(\alpha):=\{f\in\L\colon f\alpha=\One'\}\]
is the kernel of $\alpha$.  
\end{itemize}
\end{definition}

If $\L$ and $\L'$ are partial groups and $\alpha\colon\L\rightarrow\L'$ is a homomorphism of partial groups, then by \cite[Lemma~1.14]{Chermak:2015}, $\ker(\alpha)$ is a partial normal subgroup of $\L$. The other way around, if $(\L,\Delta,S)$ is a locality and $\N\unlhd\L$, then one can construct another partial group $\L/\N$ and a homomorphism $\L\rightarrow\L/\N$ with kernel $\N$. To make this more precise, call a subset of $\L$ of the form \[\N f:=\{nf\colon n\in\N,\;(n,f)\in\D\}\] with $f\in\L$ a \emph{coset} of $\N$ in $\L$. Writing $\L/\N$ for the set of maximal (with respect to inclusion) cosets of $\L$, by \cite[Proposition~3.14(d)]{Chermak:2015}, $\L/\N$ forms a partition of $\L$. Moreover, if
\[\alpha\colon \L\rightarrow \L/\N\]
is the map sending every element $f\in\L$ to the unique maximal coset of $\L$ containing $f$, then by \cite[Lemma~3.16]{Chermak:2015}, there is a unique partial group structure on $\L/\N$ such that the map $\alpha$ becomes a projection of partial groups. The map $\alpha$ is called the \emph{natural projection}. As $\N=\N\One$ is itself a maximal coset, we have $\ker(\alpha)=\N$.

\smallskip

Setting $\ov{\L}:=\L/\N$, we will use a ``bar notation'' similarly as for groups, i.e. for any element or subset $X$ of $\L$, the image of $X$ in $\L$ under the map $\alpha$ as above is denoted by $\ov{X}$. By \cite[Corollary~4.5]{Chermak:2015}, $(\ov{\L},\ov{\Delta},\ov{S})$ is a locality, where $\ov{\Delta}:=\{\ov{P}\colon P\in\Delta\}$.

\subsection{Proper localities and localities of objective characteristic $p$}\label{SS:Proper}

\begin{definition}
 \begin{itemize}
  \item A locality $(\L,\Delta,S)$ is said to be \emph{of objective characteristic $p$}, if $N_\L(P)$ is of characteristic $p$ for every $P\in\Delta$.
 \item A locality $(\L,\Delta,S)$ is called a \emph{proper locality} if it is of objective characteristic $p$, the fusion system $\F_S(\L)$ is saturated and $\F_S(\L)^{cr}\subseteq\Delta\subseteq\F_S(\L)^s$.
 \end{itemize}
\end{definition}

As detailed in \cite[Theorem~3.26]{Henke:Regular}, in the definition of a proper locality the assumption that $\F_S(\L)$ is saturated is unnecessary if one uses a different definition of centric radical subgroups. However, this subtlety does not play a role for the arguments in this paper, so we define proper localities as above and stick to the usual definition of centric radical subgroups as for example given in \cite[Definition~3.1]{Aschbacher/Kessar/Oliver:2011}.

\smallskip

If $\F$ is a saturated fusion system over $S$ and $\Delta$ is a set of subgroups of $S$, then under certain assumptions on $\Delta$ there exists a proper locality $(\L,\Delta,S)$ over $\F$. To make this more precise, we need the following definition. 

\begin{definition}\label{D:Subcentric}
Let $\F$ be a fusion system over $S$. 
\begin{itemize}
\item A subgroup $P$ of $S$ is called \emph{subcentric} (or more precisely \emph{$\F$-subcentric}), if $O_p(N_\F(Q))$ is $\F$-centric for every fully $\F$-normalized $\F$-conjugate $Q$ of $P$. The set of $\F$-subcentric subgroups of $S$ is denoted by $\F^s$.
\item A \emph{subcentric locality} over $\F$ is a proper locality $(\L,\Delta,S)$ over $\F$ with $\Delta=\F^s$. 
\end{itemize}
\end{definition}

As remarked before, whenever $(\L,\Delta,S)$ is a locality over a fusion system $\F$, the set $\Delta$ is $\F$-closed. The existence of a proper locality $(\L,\Delta,S)$ over $\F$ implies in addition that $\F^{cr}\subseteq\Delta\subseteq \F^s$, where the latter inclusion was shown in \cite[Lemma~6.1]{Henke:2015}. The following theorem states basically that the converse of this statement holds for a saturated fusion system $\F$. 

\begin{theorem}\label{T:ProperLocalityExistence}
Let $\F$ be a saturated fusion system over $S$. If $\Delta$ is an $\F$-closed collection of subgroups of $S$ with $\F^{cr}\subseteq\Delta\subseteq \F^s$, then there exists a proper locality $(\L,\Delta,S)$ over $\F$ which is unique up to a rigid isomorphism. The set $\F^s$ is $\F$-closed, so in particular there exists a subcentric locality over $\F$ which is unique up to a rigid isomorphism.
\end{theorem}

\begin{proof}
For $\Delta=\F^c$ this was shown in \cite{Chermak:2013}; a proof which does not rely on the classification of finite simple groups can be given through \cite{Oliver:2013} and \cite{Glauberman/Lynd}. Using that the theorem is true for $\Delta=\F^c$, a proof of the general statement is given in \cite[Theorem~A]{Henke:2015}.  
\end{proof}

\begin{definition}
We call a proper locality $(\L,\Delta,S)$ over a fusion system $\F$ a \emph{subcentric locality} if $\Delta=\F^s$.
\end{definition}

\subsection{Varying the object sets of proper localities}

The proper localities over a fixed fusion system are all closely connected as the following theorem shows. If $\L$ is a partial group we write $\fN(\L)$ for the set of partial normal subgroups of $\L$.

\begin{theorem}\label{T:VaryObjects}
Let $(\L,\Delta,S)$ be a proper locality over a fusion system $\F$ and let $\Delta^+$ be an $\F$-closed collection of subgroups of $S$ such that $\Delta\subseteq \Delta^+\subseteq \F^s$.
\begin{itemize}
\item [(a)] There exists a proper locality $(\L^+,\Delta^+,S)$ over $\F$ such that $\L=\L^+|_\Delta$. Moreover, $(\L^+,\Delta^+,S)$ is unique up to an isomorphism which restricts to the identity on $\L$; that is, if $(\wL^+,\Delta^+,S)$ is another linking locality over $\F$ with $\wL^+|_{\Delta}=\L$, then there exists an isomorphism of partial groups $\beta\colon\L^+\rightarrow\wL^+$ which is the identity on $\L$.
\item [(b)] If $(\L^+,\Delta^+,S)$ is a linking locality over $\F$ with $\L^+|_\Delta=\L$, then the map 
\[\Phi_{\L^+,\L}\colon \fN(\L^+)\rightarrow \fN(\L),\;\N^+\mapsto \N^+\cap\L
\]
 is well-defined and an inclusion-preserving bijection such that $\Phi_{\L^+,\L}^{-1}$ is also inclusion-preserving. 
\item [(c)] If $\N^+\unlhd\L^+$ and $\N:=\L\cap\N^+\unlhd\L$ such that $\F_{S\cap\N}(\N)$ is $\F$-invariant, then $\F_{S\cap\N^+}(\N^+)=\F_{S\cap\N}(\N)$. 
\item [(d)] The set $\F^s$ of subcentric subgroups of $S$ is $\F$-closed and contains $\Delta$. In particular, there exists a subcentric locality $(\L^s,\F^s,S)$ over $\F$ with $\L^s|_\Delta=\L$, and such $\L^s$ is unique up to an isomorphism which restricts to the identity on $\L$. 
\end{itemize}
\end{theorem}

\begin{proof}
Part (a) follows from \cite[Theorem~7.2(a),(b)]{Henke:2015} or \cite[Theorem~A1]{ChermakII}; part (d) follows then from Proposition~3.3 and Lemma~6.1 in \cite{Henke:2015}. Part (b) was first proved in \cite[Theorem~A2]{ChermakII}; proofs of (b) and (c) can be found in  \cite[Theorem~C(a),(b)]{Henke:2020}. 
\end{proof}

As the next lemma shows, there is some flexibility in choosing  object sets of proper localities even if the underlying partial group remains fixed.

\begin{lemma}\label{L:VaryObjects}
Let $(\L,\Delta,S)$ be a locality over $\F$ and let $X\leq S$ with $\L=N_\L(X)$. Set
\[\Delta_0:=\{P\in\Delta\colon X\leq P\}\mbox{ and }\tDelta:=\{P\leq S\colon PX\in\Delta\}.\]
Then $\Delta_0$ and $\tDelta$ are $\F$-closed. Moreover, for every $\F$-closed collection $\Delta'$ with $\Delta_0\subseteq\Delta'\subseteq \tDelta$ the following properties hold:
\begin{itemize}
 \item [(a)] $(\L,\Delta',S)$ is a locality over $\F$.
 \item [(b)] $(\L,\Delta,S)$ is of objective characteristic $p$ if and only if $(\L,\Delta',S)$ is of objective characteristic $p$. 
 \item [(c)] $(\L,\Delta,S)$ is proper if and only if $(\L,\Delta',S)$ is proper.
\end{itemize}
\end{lemma}

\begin{proof}
As $\Delta$ is closed under passing to $\L$-conjugates in $S$ and since $\L=N_\L(X)$, one sees that the sets $\Delta_0$ and $\tDelta$ are closed under passing to $\L$-conjugates in $S$ and thus also closed under $\F$-conjugacy. Moreover, the fact that $\Delta$ is overgroup-closed in $S$ implies that $\Delta_0$ and $\tDelta$ are overgroup-closed in $S$ and $\Delta\subseteq\tDelta$. So $\Delta_0$ and $\tDelta$ are $\F$-closed with $\Delta_0\subseteq \Delta\subseteq\tDelta$. 

\smallskip

\textbf{(a)} As $\Delta_0\subseteq \Delta\subseteq\tDelta$, we have  
\[\D_{\Delta_0}\subseteq\D_\Delta=\D\subseteq\D_{\tDelta}.\]
If $w=(f_1,\dots,f_k)\in\D_{\tDelta}$ via $P_0,P_1,\dots,P_k\in\tDelta$, then $w\in\D_{\Delta_0}$ via $P_0X,P_1X,\dots,P_kX$. So $\D_{\tDelta}\subseteq\D_{\Delta_0}$ and thus $\D_{\Delta_0}=\D=\D_{\tDelta}$. By assumption, $\Delta_0\subseteq\Delta'\subseteq\tDelta$ and thus $\D=\D_{\Delta_0}\subseteq\D_{\Delta'}\subseteq\D_{\tDelta}=\D$, i.e. $\D_{\Delta'}=\D$. So (L2) holds for $(\L,\Delta',S)$ in place of $(\L,\Delta,S)$. As $\Delta'$ is $\F$-closed and $\F=\F_S(\L)$, it follows that $\Delta'$ is also closed under passing to $\L$-conjugates and overgroups in $S$. Since $(\L,\Delta,S)$ is a locality, $S$ is a maximal $p$-subgroup of $\L$. Thus, $(\L,\Delta',S)$ is a locality. Notice also that the definition of $\F_S(\L)$ does not depend on $\Delta$. Hence $(\L,\Delta',S)$ is a locality over $\F$ and (a) holds.

\smallskip

\textbf{(b)} As $\Delta_0\subseteq\Delta\subseteq\tDelta$ and $\Delta_0\subseteq \Delta'\subseteq\tDelta$, we have
\begin{eqnarray*}
(\L,\tDelta,S)\mbox{ is of objective characteristic $p$}&\Longrightarrow&(\L,\Delta,S)\mbox{ is of objective characteristic $p$}\\
&\Longrightarrow&(\L,\Delta_0,S)\mbox{ is of objective characteristic $p$}
\end{eqnarray*}
and
\begin{eqnarray*}
(\L,\tDelta,S)\mbox{ is of objective characteristic $p$}&\Longrightarrow&(\L,\Delta',S)\mbox{ is of objective characteristic $p$}\\
&\Longrightarrow&(\L,\Delta_0,S)\mbox{ is of objective characteristic $p$}
\end{eqnarray*}
Suppose now that $(\L,\Delta_0,S)$ is of objective characteristic $p$. Then for every $P\in\tDelta$, the normalizer $N_\L(PX)$ is a group of characteristic $p$ as $PX\in\Delta_0$. Since $\L=N_\L(X)$, it follows from Lemma~\ref{L:MSCharp}(b) that $N_\L(P)=N_{N_\L(PX)}(P)$ is of characteristic $p$. This shows that $(\L,\tDelta,S)$ is of objective characteristic $p$. Hence all the implications above are equivalences and thus (b) holds. 

\smallskip

\textbf{(c)} Assume that $\F=\F_S(\L)$ is saturated and observe that $X\unlhd\F_S(\L)$. Hence, \cite[Proposition~I.4.5]{Aschbacher/Kessar/Oliver:2011} gives $X\leq P$ for all $P\in\F_S(\L)^{cr}$. So if $\F_S(\L)^{cr}\subseteq\Delta'\subseteq\tDelta$, then $\F_S(\L)^{cr}\subseteq \Delta_0\subseteq\Delta$. Similarly, if $\F_S(\L)^{cr}\subseteq\Delta$, then also $\F_S(\L)^{cr}\subseteq\Delta_0\subseteq\Delta'$. This shows that $\F^{cr}\subseteq \Delta$ if and only if $\F^{cr}\subseteq\Delta'$. Hence (c) follows from (b).
\end{proof}

Using Lemma~\ref{L:VaryObjects}, we are now able to prove the following lemma.

\begin{lemma}\label{L:NFTSubcentricLocality}
 Let $(\L,\Delta,S)$ be a subcentric locality over $\F$ and $T\leq S$ strongly $\F$-closed. Then $(N_\L(T),N_\F(T)^s,S)$ is a subcentric locality over $N_\F(T)$ and $C_\L(T)\unlhd N_\L(T)$ with $C_\F(T)=\F_{C_S(T)}(C_\L(T))$.  
\end{lemma}

\begin{proof}
 By \cite[Lemma~3.35(b)]{Henke:Regular}, $(N_\L(T),\Delta,S)$ is a linking locality over $N_\F(T)$, $C_\L(T)\unlhd N_\L(T)$ and $C_\F(T)=\F_{C_S(T)}(C_\L(T))$. As $\Delta=\F^s$, it follows moreover from \cite[Lemma~3.39]{Henke:Regular} that 
\[ N_\F(T)^s=\{P\leq S\colon PT\in\F^s\}=\{P\leq S\colon PT\in\Delta\}.\]
It follows thus from Lemma~\ref{L:VaryObjects} applied with $N_\L(T)$ and $T$ in the roles of $\L$ and $X$ that $(N_\L(T),N_\F(T)^s,S)$ is a proper locality and thus a subcentric locality over $N_\F(T)$.
\end{proof}

\subsection{Index prime to $p$, $p$-power index, simple and quasisimple localities} \label{SS:ResiduesLocalities}

\begin{definition}\label{D:OpLOpprimeL}
Let $(\L,\Delta,S)$ be a locality and $\N$ a partial normal subgroup of $\L$. Set $T:=S\cap \N$.
\begin{itemize}
 \item Set
\[\mathbb{K}_\N:=\{\K\unlhd\L\colon \K T=\N\}\mbox{ and }\mathbb{K}^\prime_\N:=\{\K\unlhd\L\colon T\subseteq \K\subseteq\N\}\]
and define
\[O_\L^p(\N):=\bigcap \mathbb{K}_\N\mbox{ and }O_\L^{p^\prime}(\N):=\bigcap \mathbb{K}^\prime_\N.\]
\item Write $O^p(\L)$ for $O^p_\L(\L)$ and $O^{p^\prime}(\L)$ for $O^{p^\prime}_\L(\L)$.
\item We say that $\K\unlhd \L$ has $p$-power index in $\N$ if $\K\in\mathbb{K}_\N$. Similarly, we say that $\K$ has index prime to $p$ in $\N$ if $\K\in\mathbb{K}^\prime_\N$.
\end{itemize}
\end{definition}

It is shown in \cite[Proposition~7.2]{ChermakII} or \cite[Lemma~7.2]{Henke:Regular} that $O^p_\L(\N)$ has $p$-power index in $\N$. Clearly, $O^{p^\prime}_\L(\N)$ has index prime to $p$ in $\N$.

\smallskip

Using \cite[Lemma~3.5]{Henke:Regular} one sees that
\[Z(\L)=\{f\in\L\colon f\in\D(g),\;f^g=f\mbox{ for all }g\in\L\}\]
is a partial normal subgroup of $\L$. So the following definition makes sense.

\begin{definition}\label{D:SimpleQuasisimple}
\begin{itemize}
\item A partial group $\L$ is called \emph{simple} if there are exactly two partial normal subgroups of $\L$ (namely $\{\One\}$ and $\L$).  
\item If $(\L,\Delta,S)$ is a proper locality, then $\L$ is called \emph{quasisimple} if $\L=O^p(\L)$ and $\L/Z(\L)$ is simple. 
\end{itemize}
\end{definition}

Notice that $\L\neq\{\One\}$ for every simple partial group $\L$. 

\begin{lemma}\label{L:QuasisimpleOpprime}
Let $(\L,\Delta,S)$ be a proper locality which is  quasisimple. Then $S$ is non-abelian, every partial subnormal subgroup of $\L$ is either equal to $\L$ or contained in $Z(\L)$, and $\L=O^{p^\prime}(\L)$.
\end{lemma}

\begin{proof}
By \cite[Lemma~7.10]{Henke:Regular}, $S$ is non-abelian and every partial subnormal subgroup of $\L$ is either equal to $\L$ or contained in $Z(\L)$. As $S\subseteq O^{p^\prime}(\L)\unlhd\L$, this implies $O^{p^\prime}(\L)=\L$. 
\end{proof}

\subsection{Commuting partial normal subgroups and $F^*(\L)$}\label{SS:Nperp}

\begin{definition}\label{D:Commute}
If $\L$ is a partial group with product $\Pi\colon\D\rightarrow \L$ and $\X,\Y\subseteq\L$, then we say that \emph{$\X$ commutes with $\Y$} if for all $x\in\X$ and $y\in\Y$ the following implication holds:
\[(x,y)\in\D\Longrightarrow (y,x)\in\D\mbox{ and }\Pi(x,y)=\Pi(y,x).\]
\end{definition}

If $(\L,\Delta,S)$ is a locality and $\N\unlhd\L$, then by \cite[Corollary~5.11]{Henke:Regular}, there is a (with respect to inclusion) largest partial normal subgroup of $\L$ which commutes with $\N$. We denote it by $\N^\perp$. If $(\L,\Delta,S)$ is a proper locality, then an a priori different, but equivalent definition was given in \cite[Definition~5.5]{ChermakIII}. The fact that both definitions are equivalent is shown in \cite[Corollary~9.9]{Henke:Regular}.

\begin{definition}\label{D:FstarELLocalities}
 Let $(\L,\Delta,S)$ be a proper locality and $\N\unlhd\L$. 
\begin{itemize}
\item The intersection of all partial normal subgroups $\N$ of $\L$ with $\N^\perp\subseteq\N$ and $O_p(\L)\subseteq\N$ is denoted by $F^*(\L)$ and called the \emph{generalized Fitting subgroup} of $\L$.  
\item Set $E(\L):=O^p_\L(F^*(\L))$ and call $E(\L)$ the \emph{layer} of $\L$.
\end{itemize}
\end{definition}

It is shown in \cite[Theorem~2]{Henke:Regular} that $F^*(\L)$ itself has the property that $F^*(\L)^\perp\subseteq F^*(\L)$ and $O_p(\L)\subseteq F^*(\L)$; see also \cite[Lemma~6.7, Corollary~6.9]{ChermakIII}. If $(\L,\Delta,S)$ is a subcentric locality or a regular locality as defined in the next section and $\N\unlhd\L$, then $\N^\perp\subseteq\N$ if and only if $C_S(\N)\subseteq \N$ (cf. \cite[Lemma~9.21, Corollary~10.10]{Henke:Regular}).

\subsection{Regular localities} \label{SS:Regular}

Let $\F$ be a saturated fusion system over $S$ and let $(\L^s,\F^s,S)$ be a subcentric locality over $\F$. As stated in Theorem~\ref{T:ProperLocalityExistence}, such a subcentric locality over $\F$ exists always and is unique up to a rigid isomorphism (with rigid isomorphisms introduced in Definition~\ref{D:IsoLoc}). This implies that the set $\delta(\F)$ defined below depends only on $\F$ and not on $\L^s$ (cf. \cite[Lemma~10.2]{Henke:Regular}).\newline

\begin{definition}[{cf. \cite[p.37]{ChermakIII}, \cite[Definition~10.1]{Henke:Regular}}]~\label{D:Regular} 
\begin{itemize}
 \item Let $\delta(\F)$ be the set of all subgroups $P\leq S$ such that $P\cap F^*(\L^s)\in\F^s$.  
 \item A \emph{regular locality} over $\F$ is a proper locality over $\F$ whose object set is the set $\delta(\F)$.
\end{itemize}
\end{definition}

\begin{lemma}[{cf. \cite[Lemma~6.10]{ChermakIII}, \cite[Lemma~10.4]{Henke:Regular}}]\label{L:Regular10.4}
The set $\delta(\F)$ is $\F$-closed and $\F^{cr}\subseteq\delta(\F)\subseteq\F^s$. In particular (by  Theorem~\ref{T:ProperLocalityExistence}) there exists a regular locality over $\F$, which is unique up to a rigid isomorphism.
\end{lemma}

It is one of the main advantages of regular localities that partial normal subgroups have very nice properties. We summarize this in the next theorem. Part (b) will be crucial in the proof of Theorem~\ref{main}

\begin{theorem}\label{T:mainRegularPartialNormal}
Let $(\L,\Delta,S)$ be a regular locality over $\F$, let $\N\unlhd\L$, $T:=\N\cap S$ and $\E:=\F_T(\N)$.
\begin{itemize}
\item [(a)] The subsystem $\E$ is saturated, $(\N,\delta(\E),T)$ is a regular locality over $\E$ and
\[\delta(\E)=\{P\leq T\colon PC_S(\N)\in\Delta\}.\]
\item [(b)] $C_\L(\N)=\N^\perp\unlhd\L$.
\item [(c)] We have $TC_S(T)\in\Delta$ and $\E\unlhd \F$.
\end{itemize}
\end{theorem}

\begin{proof}
Parts (a) and (b) are shown in \cite[Theorem~C]{ChermakIII} and in \cite[Theorem~3]{Henke:Regular}. In particular, $\m{E}$ is saturated, and $TC_S(\N)\in\Delta$ since $T\in\delta(\E)$. This implies $TC_S(T)\in\Delta$ as $\Delta$ is overgroup-closed in $S$. By the saturation axioms, every $\F$-automorphism of $T$ extends to an $\F$-automorphism of $TC_S(T)$, which is then by Lemma~\ref{L:LocalitiesPropg}(a) of the form $c_f|_{TC_S(T)}$ with $f\in N_\L(TC_S(T))\subseteq N_\L(T)$. This implies
\[\Aut_\F(T)=\{c_f|_T\colon f\in N_\L(T)\}.\]
By \cite[Theorem~10.16(f)]{Henke:Regular}, $c_f|_T$ induces an automorphism of $\E$ for every $f\in N_\L(T)$. Hence, every element of $\Aut_\F(T)$ induces an automorphism of $\E$. The assertion follows now from \cite[Lemma~3.28]{Henke:Regular}.
\end{proof}

By induction on the length of a subnormal series one obtains the following corollary to Theorem~\ref{T:mainRegularPartialNormal}.

\begin{corollary}[{cf. \cite[Corollary~7.10]{ChermakIII}, \cite[Corollary~10.19]{Henke:Regular}}]\label{C:RegularSubnormal}
Let $(\L,\Delta,S)$ be a regular locality over $\F$, $\H\subn\L$, $T:=\H\cap S$ and $\E:=\F_T(\H)$. Then $\E$ is  saturated and subnormal in $\F$, and  $(\H,\delta(\E),T)$ is a regular locality over $\E$. Moreover, $PC_S(\H)\in\Delta$ for every $P\in\delta(\E)$.
\end{corollary}

Since every partial subnormal subgroup can be regarded as a regular locality, the following definition makes sense.

\begin{definition}
A \emph{component} of $\L$ is a partial subnormal subgroup $\K$ of $\L$ such that $\K$ is quasisimple. We write $\Comp(\L)$ for the set of components of $\L$.
\end{definition}

As shown in \cite[Theorem~4, Theorem~11.18]{Henke:Regular}, there is a theory of components of regular localities akin to the theory of components of finite groups. For example, if $(\L,\Delta,S)$ is a regular locality and $\K_1,\dots,\K_r\in\Comp(\L)$ are pairwise distinct components of $\L$, then  the product $\K_1\K_2\cdots\K_r$ is a partial normal subgroup of $F^*(\L)$ which is an internal central product of $\K_1,\dots,\K_r$ in the sense of \cite[Definition~4.1]{Henke:Regular}. In particular, the order of the factors in the product $\K_1\K_2\cdots\K_r$ does not matter. Thus, given $\fC\subseteq\Comp(\L)$, we may form $\prod_{\K\in\fC}\K\unlhd F^*(\L)$. With the definition of the layer as in Definition~\ref{D:FstarELLocalities}, it turns out that $F^*(\L)$ is a central product of $E(\L)$ and $O_p(\L)$ and that 
\[E(\L)=\prod_{\K\in\Comp(\L)}\K.\]
Indeed, in \cite[Definition~11.8]{Henke:Regular}, the latter equation is used as the definition of $E(\L)$ in the case of regular localities; a somewhat similar definition is used in \cite[Theorem~8.5]{ChermakIII}. It is then shown that $E(\L)\unlhd\L$ and $O^p(F^*(\L))=E(\L)$ (cf. \cite[Theorem~8.5]{ChermakIII} and \cite[Lemma~11.18(a)]{Henke:Regular}). Therefore, it makes sense to define $E(\L)$ for arbitrary proper localities as in Definition~\ref{D:FstarELLocalities}. We will use the stated properties of components, $F^*(\L)$ and $E(\L)$ later on to reprove results about components of fusion systems.

%\begin{theorem}\label{RegularNormal}
%Let $(\L,\Delta,S)$ be a regular locality over $\F$, let $\N\unlhd\L$ and set $T:=\N\cap S$. Then $TC_S(T)\in\Delta$ and $\F_T(\N)$ is a normal subsystem of $\F$.
%\end{theorem}

%\begin{proof}
%It follows from Theorem~\ref{T:mainRegularPartialNormal} that
%\end{proof}

\section{Partial normal subgroups and their products with $p$-subgroups}

In this section we will state and prove some lemmas about partial normal subgroups which will be needed in the proofs of Theorem~\ref{main} and Theorem~\ref{T:mainProductWithPSubgroup}.

\begin{lemma}\label{GammaDeltaLemma}
Let $(\L,\Delta,S)$ be a locality, $\N\unlhd\L$ and $T:=S\cap\N$. Suppose
\begin{equation*}
P\cap T\in\Delta\mbox{ for every }P\in\Delta.
\end{equation*}
Then the following hold:
\begin{itemize}
 \item[(a)] $\L=\D(f)$ and $c_f\in\Aut(\L)$ for every $f\in N_\L(T)$.
 \item[(b)] Setting $\Gamma:=\{P\in\Delta\colon P\leq T\}$, the triple $(\N,\Gamma,T)$ is a locality.
 \item[(c)] If $(\L,\Delta,S)$ is of objective characteristic $p$, then $(\N,\Gamma,T)$ is of objective characteristic $p$.
\end{itemize}
\end{lemma}

\begin{proof}
See \cite[Lemma~3.29(b),(c)]{Henke:Regular}. 
\end{proof}

\begin{lemma}\label{L:NTQinNNQ}
Let $(\L,\Delta,S)$ be a locality, $\N\unlhd\L$ and $T:=\N\cap S$. Let $Q\in\Delta$. Then there exists $n\in\N$ with $N_S(Q)\leq S_n$ and $N_T(Q^n)\in\Syl_p(N_\N(Q^n))$. 
\end{lemma}

\begin{proof}
By \cite[Lemma~2.9]{Chermak:2015}, there exists $g\in\L$ such that $N_S(Q)\leq S_g$ and $N_S(Q^g)\in\Syl_p(N_\L(Q^g))$. By the Frattini Lemma and the Splitting Lemma \cite[Corollary~3.11, Lemma~3.12]{Chermak:2015}, there exist $n\in\N$ and $f\in N_\L(T)$ such that $(n,f)\in\D$, $g=nf$ and $S_g=S_{(n,f)}$. As $N_\N(Q^g)$ is a normal subgroup of $N_\L(Q^g)$, we have $N_T(Q^g)=N_S(Q^g)\cap N_\N(Q^g)\in\Syl_p(N_\N(Q^g))$. By Lemma~\ref{LocalitiesProp}(b), $c_f\colon N_\L(Q^n)\rightarrow N_\L(Q^g)$ is an isomorphism of groups which takes $N_\N(Q^n)$ onto $N_\N(Q^g)$. It follows that $N_T(Q^g)^{f^{-1}}$ is a Sylow $p$-subgroup of $N_\N(Q^n)$. Since $f\in N_\L(T)$, we have $N_T(Q^g)^{f^{-1}}\leq N_T(Q^n)$. As $N_T(Q^n)$ is a $p$-subgroup of $N_\N(Q^n)$, we conclude that $N_T(Q^n)\in\Syl_p(N_\N(Q^n))$. Note that $N_S(Q)\leq S_g=S_{(n,f)}\leq S_n$.
\end{proof}

\begin{lemma}\label{L:UniquenessLemmaMsupseteqN}
Let $(\L,\Delta,S)$ be a locality of objective characteristic $p$. Suppose $\M$ and $\N$ are partial normal subgroups of $\L$ such that $T:=\M\cap S=\N\cap S$ and $TC_S(T)O_p(\L)\in\Delta$. If $\F_T(\N)=\F_T(\M)$ and $\M\supseteq \N$, then $\M=\N$. 
\end{lemma}

\begin{proof}
By the Frattini argument for localities \cite[Corollary~4.8]{Chermak:2013}, we have $\L=\N N_\L(T)$. So using the Dedekind Lemma for partial groups \cite[Lemma~1.10]{Chermak:2015}, it follows $\M=\N N_\M(T)$. As $(\L,\Delta,S)$ is of objective characteristic $p$ and $P:=TC_S(T)O_p(\L)\in\Delta$, Lemma~\ref{L:LocalitiesPropg}(c) implies that $G:=N_\L(P)$ is a  model for $N_\F(P)$. Since $P\unlhd G$ and $C_S(P)\leq P$, it follows from Lemma~\ref{L:MSCharp}(c) that $C_G(P)\leq P$. By \cite[Lemma~3.5]{Chermak:2015}, $N_\M(T)$ and $N_\N(T)$ are contained $N_\L(TC_S(T))$ and thus in $G$ by Lemma~\ref{L:OpL}. Notice also that $T$ is normal in $G$ as $T\leq P\unlhd G$ and $T$ is strongly closed in $\F_S(\L)$ by \cite[Lemma~3.1(a)]{Chermak:2015}. Thus, $M:=N_\M(T)=\M\cap G$ and $N:=N_\N(T)=\N\cap G$ are normal subgroups of $G$ with $N\leq M$. As $\F_T(\M)=\F_T(\N)$, we have $\Aut_M(T)=\Aut_N(T)$ and so  $M=NC_M(T)$. Notice that $M$ is of characteristic $p$ by Lemma~\ref{L:MSCharp}(a). Moreover, \cite[Lemma~3.1(c)]{Chermak:2015} gives that $T$ is a maximal $p$-subgroup of $\M$ and thus a normal Sylow $p$-subgroup of $M$. Therefore, $C_M(T)\leq T\leq N$. So $M=NC_M(T)=N$ and $\M=\N N_\M(T)=\N M=\N$. 
\end{proof}

Notice that Lemma~\ref{ProductSubnormal} allows us to form the fusion system $\F_{(\N\cap S)R}(\N R)$ as in Definition~\ref{D:F(H)} if $\N\unlhd\L$ and $R\leq S$.  Therefore, we are able to formulate the following lemma.

\begin{lemma}\label{L:PartialNormalInclusionLemma}
Let $(\L,\Delta,S)$ be a locality of objective characteristic $p$. Suppose $\M$ and $\N$ are partial normal subgroups of $\L$ such that $\M\cap S\leq T:=\N\cap S$, $TC_S(T)O_p(\L)\in\Delta$ and $\F_T(\M T)\subseteq \F_T(\N)$. Then $\M\subseteq\N$. 
\end{lemma}

\begin{proof}
By Theorem~\ref{T:ProductsPartialNormal}, the subset $\M\N$ is a partial normal subgroup of $\L$ with $(\M\N)\cap S=(\M\cap S)(\N\cap S)=T$. Moreover, this theorem states that, given $f\in\M\N$, there exist $m\in\M$ and $n\in\N$ with $(m,n)\in\D$, $f=mn$ and $S_f=S_{(m,n)}$. For such $f$, $m$ and $n$, the morphism $c_f|_{S_f\cap T}\colon S_f\cap T\rightarrow T$ is the composition of $c_m|_{S_f\cap T}$ with $c_n|_{S_f^m\cap T}$. Thus, $c_f|_{S_f\cap T}$ is a morphism in $\F_T(\N)\supseteq \F_T(\M T)$ for all $f\in\M\N$. This shows $\F_T(\M\N)\subseteq \F_T(\N)$. By Remark~\ref{R:MandNinMN}, $\N\subseteq\M\N$ and so $\F_T(\N)\subseteq\F_T(\M\N)$. Hence, $\F_T(\M\N)=\F_T(\N)$ and Lemma~\ref{L:UniquenessLemmaMsupseteqN} (applied with $\M\N$ in place of $\M$) implies $\M\N=\N$. Using again Remark~\ref{R:MandNinMN}, this yields $\M\subseteq\N$ as required.
\end{proof}

\begin{corollary}\label{C:UniquenessLemmaStrong}
Let $(\L,\Delta,S)$ be a locality of objective characteristic $p$. Let $\M$ and $\N$ be partial normal subgroups of $\L$ with $T:=\M\cap S=\N\cap S$ and $\F_T(\M)=\F_T(\N)$. If $TC_S(T)O_p(\L)\in\Delta$, then $\M=\N$.
\end{corollary}

\begin{proof}
This follows from Lemma~\ref{L:PartialNormalInclusionLemma} applied twice, once with the roles of $\M$ and $\N$ reversed.
\end{proof}

\begin{lemma}\label{ProductPrepare}
Let $(\L,\Delta,S)$ be a locality and $\N\unlhd\L$. Then the following hold:
\begin{itemize}
 \item [(a)] The triple $(\N S,\Delta,S)$ is a locality.
 \item [(b)] If $P\in\Delta$ such that $N_\L(P)$ is of characteristic $p$, then $N_{\N S}(P)$ is of characteristic $p$.
\end{itemize}
\end{lemma}

\begin{proof}
For part (a) see \cite[Lemma~4.1]{Chermak:2015}; recall also that $\N S$ is a partial subgroup of $\L$ by Lemma~\ref{ProductSubnormal}. For the proof of (b) fix now $P\in\Delta$ such that $N_\L(P)$ is of characteristic $p$. Then by Lemma~\ref{L:MSCharp}(a), $O^p(N_\N(P))\unlhd N_\L(P)$ is of characteristic $p$. By \cite[Lemma~6.1(b)]{Henke:2020}, $O^p(N_{\N S}(P))=O^p(N_\N(P))$. It follows now from \cite[Lemma~1.3]{MS:2012b} that $N_{\N S}(P)$ is of characteristic $p$, i.e. (b) holds.
\end{proof}

\begin{lemma}\label{ProductWithPSubgroup}
Suppose $(\L,\Delta,S)$ is a locality of objective characteristic $p$ such that $\F=\F_S(\L)$ is saturated. Let $\N\unlhd\L$ and set $T:=S\cap\N$. If $\E:=\F_T(\N)$ is normal in $\F$, then $\N R\cap S=TR$ and $\F_{TR}(\N R)\subseteq \E R=(\E R)_\F$ for every subgroup $R$ of $S$. 
\end{lemma}

\begin{proof}
Let $R\leq S$. We have seen already in Lemma~\ref{ProductSubnormal} that $(\N R)\cap S=(\N\cap S)R=TR$. To ease notation, replacing $R$ by $RT$, we assume from now on $T\leq R$ and thus $(\N R)\cap S=R$. Fix $f\in \N R$ and set $P:=R\cap S_f$. We must show that $c_f|_P\colon P\rightarrow R$ is a morphism in $\E R$, since the fusion system $\F_R(\N R)$ is by definition generated by conjugation maps of this form. Write $f=nr$ with $n\in\N$ and $r\in R$. By \cite[Lemma~2.8]{Henke:2020}, we have $S_f=S_{(n,r)}=S_n$. Thus, $c_f=c_n|_P\circ c_r|_{P^n}$ where $P^n\leq S\cap (\N R)=R$. Note that $c_r|_{P^n}$ is clearly an element of $\E R$. So to prove (a), it is enough to show that $c_n|_P\colon P\rightarrow R$ is an element of $\E R$. 

\smallskip

By \cite[Lemma~6.2]{Henke:2020}, there exist $k\in\mathbb{N}$, $R_1,R_2,\dots,R_k\in \Delta$ and $(t,n_1,n_2,\dots,n_k)\in\D$ such that $n=tn_1n_2\cdots n_k$, $S_n=S_{(t,n_1,\dots,n_k)}$, $t\in T$ and, for all $i=1,\dots,k$,
\[n_i\in O^p(N_\N(R_i)),\;S_{n_i}=R_i,\;O_p(N_{\N S}(R_i))=R_i\mbox{ and }N_S(R_i)\in\Syl_p(N_{\N S}(R_i)).\]
Setting $P_0:=P^t$ and $P_i:=P_{i-1}^{n_i}$ for $i=1,\dots,k$, we have $c_n|_P=(c_t|_{P})\circ (c_{n_1}|_{P_0})\circ\dots\circ (c_{n_k}|_{P_{k-1}})$. Clearly $c_t|_P$ is a morphism in $\E R$ as $t\in T\leq R$. So it is sufficient to show that $c_{n_i}|_{P_{i-1}}\colon P_{i-1}\rightarrow P_i$ is a morphism in $\E R$ for all $i=1,\dots,k$. To prove this fix $i\in\{1,\dots,k\}$.

\smallskip

Since $P_{i-1}$ and $P_i$ are conjugate to $P\leq R$ under sequences of elements of the partial subgroup $\N R$, we have $\<P_{i-1},P_i\>\leq (\N R)\cap S=R$. So $\<P_{i-1},P_i\>\leq Q_i:=R_i\cap R$. As $(\L,\Delta,S)$ is a proper locality and $O_p(N_{\N S}(R_i))=R_i$, it follows from \cite[Lemma~6.3, Lemma~6.14]{Henke:Regular} that $R_i\cap T\in\E^c$ and thus $Q_i\cap T=R_i\cap T\in\E^c$. Note also that $[Q_i,N_\N(R_i)]\leq [R_i,N_\N(R_i)]\leq R_i\cap\N=R_i\cap T=Q_i\cap T$. In particular, $N_\N(R_i)\subseteq N_\N(Q_i)$. Moreover, as every element of $N_\N(R_i)$ induces an $\E$-automorphism of $R_i\cap T=Q_i\cap T$, it follows that the automorphisms of $Q_i$, which are obtained by conjugation by elements of $O^p(N_{\N}(R_i))$, are in $\Ac(Q_i)$ (cf. Definition~\ref{D:Product}). In particular, $c_{n_i}|_{Q_i}\in\Ac(Q_i)$ extends $c_{n_i}|_{P_{i-1}}$. Since $Q_i\cap T=R_i\cap T\in\E^c$, it follows that $c_{n_i}|_{P_{i-1}}$ is a morphism in $\E R$ as was left to show.
\end{proof}

\begin{lemma}\label{L:E1unlhdE2Help}
Let $(\L,\Delta,S)$ be a locality of objective characteristic $p$ such that $\F=\F_S(\L)$ is saturated. Let $\M$ and $\N$ be partial normal subgroups of $\L$ and $T:=S\cap\N$. Assume that $TC_S(T)O_p(\L)\in\Delta$. Suppose moreover that $\F_T(\N)$ is saturated and that $\F_{S\cap\M}(\M)$ is a normal subsystem of $\F_T(\N)$ and of $\F$. Then $\M\subseteq\N$.
\end{lemma}

\begin{proof}
As $\E_\M:=\F_{S\cap\M}(\M)\unlhd\E_\N:=\F_T(\N)$, we have in particular that $S\cap\M\leq T$. Hence, by Lemma~\ref{L:PartialNormalInclusionLemma}, it is sufficient to prove that $\F_T(\M T)\subseteq \E_\N$. By Lemma~\ref{ProductWithPSubgroup} applied with $\M$ and $T$ in the role of $\N$ and $R$, we have $\F_T(\M T)\subseteq \E_\M T=(\E_\M T)_\F$. By Remark~\ref{R:ERsubG}, we have moreover $\E_\M T\subseteq\E_\N$. So the assertion follows. 
\end{proof}

The following proposition can be regarded as a first step towards the proof of Theorem~\ref{T:mainProductWithPSubgroup}.

\begin{prop}\label{P:ERNR0}
Let $(\L,\Delta,S)$ be a proper locality over $\F$ and let $\N$ be a partial normal subgroup of $\L$. Set $T:=S\cap\N$ and suppose that $\E=\F_T(\N)$ is normal in $\F$. Assume furthermore 
\begin{align}\label{dagger}\tag{$\dagger$}
UC_S(\N)O_p(\L)\in\Delta\mbox{ for every }U\in\E^{cr}.
\end{align}
\noindent Then $(\E R)^{cr}\subseteq\Delta$ and $\E R=\F_{TR}(\N R)$ for every subgroup $R$ of $S$ with $C_S(\N)\subseteq RT$. Moreover, $(\N S,\Delta,S)$ is a proper locality over $\E S$. 
\end{prop}

\begin{proof}
Replacing $R$ by $RT$ we may assume that $T\leq R$ and $C_S(\N)\leq R$. Recall from Lemma~\ref{ProductSubnormal} that $\N R\cap S=TR=R$ and $\N R$ is a partial subgroup. We know moreover from Lemma~\ref{ProductPrepare}(a) that $(\N S,\Delta,S)$ is a locality. Set $X:=O_p(\L)$ and 
\[\tDelta:=\{P\leq S\colon PX\in\Delta\}.\]
By Lemma~\ref{L:OpL}, we have $\L=N_\L(X)$ and thus $\N S=N_{\N S}(X)$. So Lemma~\ref{L:VaryObjects} yields firstly that $(\L,\tDelta,S)$ is a proper locality and secondly that $(\N S,\Delta,S)$ is a proper locality if and only if $(\N S,\tDelta,S)$ is a proper locality. Moreover (\ref{dagger}) means that $UC_S(\N)\in\tDelta$ for every $U\in\E^{cr}$. Hence, replacing $\Delta$ by $\tDelta$, we may assume from now on that
\begin{align}\label{daggerdagger}\tag{$\dagger\dagger$}
UC_S(\N)\in\Delta\mbox{ for every }U\in\E^{cr}.
\end{align}
Throughout we will use that $C_S(\N)=C_S(\E)$ by \cite[Proposition~4]{Henke:2015}, and that $C_S(\E)$ is a strongly closed subgroup (cf. \cite[Chapter~6]{Aschbacher:2011} or \cite[Theorem~1]{Henke:2018}). We will proceed in three steps.

\smallskip

\emph{Step~1:} We show $(\E R)^{cr}\subseteq \Delta$ and $\Ac(P)\leq \Aut_\N(P)$ for every $P\in (\E R)^{cr}$. 

\smallskip

To prove this, fix $P\in (\E R)^{cr}$. As $\E$ is a normal subsystem of $\E R$ by Lemma~\ref{CSENormalER}, it follows from \cite[Lemma~1.20(d)]{AOV1} that $P\cap T\in \E^{cr}$. So \eqref{daggerdagger} implies 
\[Q:=(P\cap T)C_S(\E)=(P\cap T)C_S(\N)\in\Delta.\]
As  $P$ is centric radical in $\E R$ and $C_S(\E)$ is by Lemma~\ref{CSENormalER} normal in $\E R$, \cite[Proposition~I.4.5]{Aschbacher/Kessar/Oliver:2011} gives $C_S(\E)\leq P$. Hence, $Q\leq P$ and thus $P\in \Delta$. Because of the arbitrary choice of $P$, this shows $(\E R)^{cr}\subseteq \Delta$.

\smallskip

It remains to show that $\Ac(P)\leq \Aut_\N(P)$. By Lemma~\ref{L:NTQinNNQ}, there exists $m\in\N$ with $N_S(Q)\leq S_m$ and $N_T(Q^m)\in\Syl_p(N_\N(Q^m))$. Note that $P\leq N_S(Q)\leq S_m$. By Lemma~\ref{ProductWithPSubgroup}, we have moreover $\F_R(\N R)\subseteq \E R$. In particular, $c_m\colon S_m\cap R\rightarrow R$ is a morphism in $\E R$ and thus $P^m\in (\E R)^{cr}$. By Lemma~\ref{LocalitiesProp}(b), we have $(c_m|_P)^{-1}\Aut_\N(P)(c_m|_P)=\Aut_\N(P^m)$. Observe also that $(c_m|_P)^{-1}\Ac(P)(c_m|_P)=\Ac(P^m)$. Hence, it is sufficient to show that $\Ac(P^m)\leq \Aut_\N(P^m)$. Since $T$ is strongly closed, $P^m\cap T=(P\cap T)^m$ and thus $Q^m=(P^m\cap T)C_S(\N)$. Hence, replacing $P$ by $P^m$, we may assume that
\[ N_T(Q)\in\Syl_p(N_\N(Q)).\]
Let now $\alpha\in\Ac(P)$ be a $p^\prime$-element. By Lemma~\ref{L:LocalitiesPropg}(a), there exists $f\in N_\L(P)$ with $\alpha=c_f|_P$. 
As $T$ and $C_S(\E)$ are strongly closed, $Q$ is normalized by $\alpha$. So \[f\in G:=N_\L(Q).\] 
From the definition of $\Ac(P)$, we see moreover that $\alpha|_{P\cap T}\in\Aut_\E(P\cap T)$. Hence, as $\E=\F_T(\N)$, there exist $n_1,\dots,n_k\in \N$ such that $P\cap T\leq S_{(n_1,\dots,n_k)}$ and
\[\alpha|_{P\cap T}=(c_{n_1}|_{P_0})(c_{n_2}|_{P_1})\dots (c_{n_k}|_{P_{k-1}})\]
where $P_0:=P\cap T$ and $P_i:=P_{i-1}^{n_i}$ for $i=1,\dots,k$. Define now $Q_i:=P_iC_S(\N)$ for $i=0,1,\dots,k$. Note that $P_k=P\cap T$,  $Q_0=(P\cap T)C_S(\N)=Q=Q_k$ and $n_i\in N_\N(Q_{i-1},Q_i)$. Using $Q\in\Delta$, we conclude that $(n_1,\dots,n_k)\in\D$ via $Q_0,Q_1, \dots,Q_k\in\Delta$. In particular, by Lemma~\ref{LocalitiesProp}(c), we have 
\[\alpha|_{P\cap T}=c_n|_{P\cap T}\mbox{ where }n:=\Pi(n_1,\dots,n_k)\in N_\N(P\cap T)\subseteq N:=N_\N(Q).\]
Observe that $N$ is a normal subgroup of $G$. Let $N_S(Q)\leq S_0\in\Syl_p(G)$. Set 
\[\F_0=\F_{S_0}(G),\;T_0=N_T(Q)\mbox{ and }\E_0:=\F_{T_0}(N).\]
As $T_0\leq S_0\cap N$ and $T_0\in\Syl_p(N)$, it follows that $T_0=S_0\cap N$. In particular, since $P\leq N_S(Q)\leq S_0$, we have $P\cap T_0=P\cap T$. Moreover, we saw above that $P\cap T\in \E^{cr}\subseteq\E^c$. Since $\F_{T_0}(N)\subseteq \E$, it follows $P\cap T_0=P\cap T\in \F_{T_0}(N)^c$. Note now that $\alpha\in\Aut_{\F_0}(P)$ is a $p^\prime$ automorphism with $[P,\alpha]\leq P\cap T$ and $\alpha|_{P\cap T}=c_n|_{P\cap T}\in\Aut_N(P\cap T)$. So using the notation introduced in Definition~\ref{D:Product}, we have $\alpha\in\Ac_{\F_0,\E_0}(P)$.
We use now that $\Ac_{\F_0,\E_0}(P)=O^p(\Aut_N(P))$ by \cite[Proposition~3.1]{Henke:2013}. So there exists $x\in N_N(P)\subseteq N_\N(P)$ with $\alpha=c_x|_P$. As $\alpha$ was arbitrary, this shows $\Ac(P)\leq \Aut_\N(P)$ and Step~1 is complete. 

\smallskip

\emph{Step~2:} We show that $\E R=\F_R(\N R)$. By Lemma~\ref{ProductWithPSubgroup}, it is enough to show that $\E R\subseteq \F_R(\N R)$. Since $\E R$ is saturated by \cite[Theorem~1]{Henke:2013}, it is by Alperin's Fusion Theorem \cite[Theorem~I.3.6]{Aschbacher/Kessar/Oliver:2011} generated by the automorphism groups $\Aut_{\E R}(P)$ with $P\in (\E R)^{cr}\cap (\E R)^f$. So fixing $P\in (\E R)^{cr}\cap (\E R)^f$, we need to show that $\Aut_{\E R}(P)$ consists of morphisms in $\F_R(\N R)$. As $P$ is fully $\E R$-normalized, the Sylow axiom yields $\Aut_R(P)\in\Syl_p(\Aut_{\E R}(P))$ and thus $\Aut_{\E R}(P)=\Aut_R(P)O^p(\Aut_{\E R}(P))$. Moreover, by \cite[Lemma~4.7]{Henke:2013}, $O^p(\Aut_{\E R}(P))=\Ac(P)$. Since clearly $\Aut_R(P)\leq \Aut_{\N R}(P)$, it follows now from Step~1 that $\Aut_{\E R}(P)\leq \Aut_{\F_R(\N R)}(P)$ and thus $\E R=\F_R(\N R)$. 

\smallskip

\emph{Step~3:} We show that $(\N S,\Delta,S)$ is a proper locality over $\E S$. By Step~2, $\F_S(\N S)=\E S$. Since $(\L,\Delta,S)$ is by assumption a proper locality, Lemma~\ref{ProductPrepare}(b) gives that $N_{\N S}(P)$ is of characteristic $p$ for every $P\in\Delta$. Moreover, by Step~1, $(\E S)^{cr}\subseteq \Delta$. This completes Step~3 and thus the proof of the proposition.
\end{proof}

To ease notation write $\F^q_{O_p(\F)}$ for the set of elements $P\in\F^q$ with $O_p(\F)\leq P$. Similarly write $\delta(\F)_{O_p(\F)}$ for the set of $P\in \delta(\F)$ with $O_p(\F)\leq P$.

\begin{corollary}\label{C:ERNR0}
Let $(\L,\Delta,S)$ be a proper locality over $\F$ such that $\F^q_{O_p(\F)}\subseteq\Delta$ or $\delta(\F)_{O_p(\F)}\subseteq\Delta$. Let $\N$ be a partial normal subgroup of $\L$, set $T:=S\cap\N$ and suppose that $\E=\F_T(\N)$ is normal in $\F$. Then $(\E R)^{cr}\subseteq\Delta$ and $\E R=\F_{TR}(\N R)$ for every subgroup $R$ of $S$ with $C_S(\N)\subseteq TR$. Moreover, $(\N S,\Delta,S)$ is a proper locality over $\E S$.  
\end{corollary}

\begin{proof}
By Proposition~\ref{P:ERNR0}, it is sufficient to verify the condition \eqref{dagger}. We will use that $O_p(\L)=O_p(\F)$ by \cite[Proposition~5]{Henke:2015}. If $U\in\E^{cr}$, then $U\in\delta(\E)$ by Lemma~\ref{L:Regular10.4}. Hence, Theorem~\ref{T:mainRegularPartialNormal}(a) gives  $UC_S(\N)\in\delta(\F)$ and thus $UC_S(\N)O_p(\F)\in\delta(\F)_{O_p(\F)}$ as $\delta(\F)$ is overgroup-closed. So if $\delta(\F)_{O_p(\F)}\subseteq\Delta$, then \eqref{dagger} holds.

\smallskip

It follows from \cite[Proposition~4]{Henke:2015} that $C_S(\N)=C_S(\E)\in C_\F(\E)^{cr}$. Hence, Lemma~\ref{L:tE}(b),(c) yields $UC_S(\N)\in\F^q$ and thus $UC_S(\N)O_p(\L)\in\F^q_{O_p(\F)}$ for every $U\in\E^{cr}$. So \eqref{dagger} holds also in the case $\F^q_{O_p(\F)}\subseteq\Delta$. 
\end{proof}

\section{The correspondence between partial normal subgroups and normal subsystems}

In this section we will prove that there is a one-to-one correspondence between the normal subsystems of a saturated fusion system $\F$ and the partial normal subgroups of a proper locality over $\F$. More precisely, we show Theorem~\ref{main} except for the statement that the map $\Psi_\L$ is given by $\N\mapsto \F_{S\cap\N}(\N)$ if $F^*(\F)^{cr}\subseteq \Delta$. We postpone the proof the latter statement until we have revisited the necessary background on the generalized Fitting subsystem. The results obtained in this section will be the basis to prove further one-to-one correspondences in later sections.

\subsection{A construction of a partial normal subgroup}

In this subsection, we prove the following proposition, which is our first step towards showing Theorem~\ref{main}. The reader might want to recall Definition~\ref{D:Fnatural}.

\begin{theorem}\label{mainStepI}
Let $\F$ be a saturated fusion system over a $p$-group $S$, and let $\E$ be a normal subsystem of $\F$ over $T\leq S$. Suppose $(\L,\Delta,S)$ is an $\F$-natural locality of objective characteristic $p$ such that
\[\Gamma:=\{P\cap T\colon P\in\Delta\}\subseteq\Delta.\]
Then there exists a partial normal subgroup $\N$ of $\L$ such that $T=\N\cap S$ and $\E|_{\Gamma}=\F_T(\N)$. 
\end{theorem}

The remainder of this section is devoted to the proof of Theorem~\ref{mainStepI}. Fix a saturated fusion system $\F$ over $S$ with a normal subsystem $\E$ over $T$. Assume the proposition is false for $(\F,\E,T)$ and some choice of $(\L,\Delta,S)$. Then there exists an $\F$-natural locality $(\L^+,\Delta^+,S)$ of objective characteristic $p$ such that
\[\Gamma^+:=\{P\cap T\colon P\in\Delta^+\}\subseteq\Delta^+,\]
but there does not exist any partial normal subgroup $\N^+$ of $\L^+$ such that $T=\N^+\cap S$ and $\E|_{\Gamma^+}=\F_T(\N^+)$. Among all localities with these properties we pick $(\L^+,\Delta^+,S)$ such that $|\Gamma^+|$ is minimal. We will proceed now in a series of lemmas until we reach a contradiction.

\begin{property}\label{Gamma00}
The set $\Gamma^+$ is closed under $\F$-conjugacy and is overgroup-closed in $T$.  
\end{property}

\begin{proof}
Since $\Gamma^+\subseteq\Delta^+$, we have $\Gamma^+=\{P\in\Delta^+\colon P\leq T\}$. As $\Delta^+$ is $\F$-closed by Lemma~\ref{L:LocalitiesPropg}(a) and as $T$ is strongly closed, it follows that $\Gamma^+$ is closed under $\F$-conjugacy and under passing to overgroups in $T$.  
\end{proof}

Let $R\in\Gamma^+$ be of minimal order and choose $R$ such that $R$ is fully $\F$-normalized. Note that this is possible since $\Gamma^+$ is closed under $\F$-conjugacy by \ref{Gamma00}. Set \[\Gamma:=\Gamma^+\backslash R^\F.\] 

\begin{property}\label{Gamma0}
If $R^*\in R^\F$ and $R^*<Q\leq T$, then $Q\in\Gamma$.
\end{property}

\begin{proof}
By \ref{Gamma00}, $R\in\Gamma^+$ implies  $Q\in\Gamma^+$ if $R^*<Q\leq T$ for some $R^*\in R^\F$. Since $|Q|>|R^*|=|R|$, we have $Q\not\in R^\F$ and hence $Q\in\Gamma$.
\end{proof}

\begin{property}\label{NormalSubgroup}
The following hold: 
\begin{itemize}
\item[(a)] $N_\F(R)$ and $N_\E(R)$ are saturated subsystems of $\F$, and $N_\E(R)$ is normal in $N_\F(R)$.
\item[(b)] We have $N_S(R)\in\Syl_p(N_{\L^+}(R))$ and $N_{\L^+}(R)$ is  a model for $N_\F(R)$.
\item[(c)] There exists a unique normal subgroup $K$ of the group $N_{\L^+}(R)$ such that $N_T(R)=K\cap N_S(R)$ and $\F_{N_T(R)}(K)=N_\E(R)$.
\end{itemize}
\end{property}

\begin{proof}
Property (a) is a special case of Lemma~\ref{LocalNormalSubsystems}, and property (b) follows from Lemma~\ref{L:LocalitiesPropg}(b),(c) and the assumption that $(\L^+,\Delta^+,S)$ of objective characteristic $p$. By (b) and \cite[Theorem~II.7.5]{Aschbacher/Kessar/Oliver:2011} (which we generalized in Lemma~\ref{ModelLemma}), every normal subsystem of $N_\F(R)$ is the $p$-fusion system of a unique normal subgroup of $N_{\L^+}(R)$. Together with (a) this implies (c).
\end{proof}

\begin{property}\label{Gamma}
The subgroup $T$ is an element of $\Gamma$ and so $\Gamma\neq \emptyset$. Moreover, $\Gamma$ is closed under $\F$-conjugacy and is overgroup-closed in $T$.
\end{property}

\begin{proof}
Assume $T\not\in\Gamma$. As $S\in\Delta^+$, we have $T\in\Gamma^+$. So it follows $T=R$.  Since $R\in\Gamma^+$ is of minimal order, this implies  $\Gamma^+=\{T\}$. So $\E|_{\Gamma^+}=N_\E(T)$ and, by definition of $\Gamma^+$, every element of $\Delta^+$ contains $T$. In particular, for every $g\in\L^+$, we have $T\leq S_g$. As $T$ is strongly closed in $\F$ and $T\in\Delta^+$, it follows that $\L^+=N_{\L^+}(T)$ is a group. So the partial normal subgroups of the locality $(\L^+,\Delta^+,S)$ are precisely the normal subgroups of the group $\L^+$. Hence, it follows from \ref{NormalSubgroup}(c) that there is a unique partial normal subgroup $\N^+$ of $\L^+$ such that $\N^+\cap S=T$ and $\F_T(\N^+)=N_\E(T)=\E|_{\Gamma^+}$. This is a contradiction to the assumption that $(\L^+,\Delta^+,S)$ is a counterexample. Thus $T\in\Gamma$.

\smallskip

Since $\Gamma^+$ is by \ref{Gamma0} closed under $\F$-conjugacy and since $\Gamma=\Gamma^+\backslash R^\F$, the set $\Gamma$ is closed under $\F$-conjugacy. Given $P\in\Gamma$ and $P\leq Q\leq T$, we have $Q\in\Gamma^+$ as $\Gamma^+$ is overgroup-closed in $T$. If $|Q|=|P|$ then $Q=P\in\Gamma$. If $|Q|> |P|$ then again $Q\in\Gamma$ as $R\in\Gamma^+$ is of minimal order. Hence, $\Gamma$ is overgroup-closed in $T$.
\end{proof}

Define now $\Delta$ to be the set of subgroups of $S$, which contain an element of $\Gamma$. Then $\Delta$ is overgroup-closed in $S$ by construction. As $\Gamma$ is by \ref{Gamma} non-empty and closed under $\F$-conjugacy, it follows that $\Delta$ is non-empty and closed under $\F$-conjugacy. So $\Delta\neq\emptyset$ is $\F$-closed. Moreover,  $\Gamma$ is by \ref{Gamma} overgroup-closed in $T$, which implies that 
\[\Gamma=\{P\cap T\colon P\in\Delta\}\subseteq\Delta.\]
In particular, $\Delta$ is $\F_S(\L^+)$-closed as $(\L^+,\Delta^+,S)$ is $\F$-natural. Since $\Gamma\subseteq \Gamma^+\subseteq\Delta^+$ and $\Delta^+$ is overgroup-closed, we have moreover $\Delta\subseteq \Delta^+$. Thus, setting
\[\L:=\L^+|_{\Delta}\] 
the triple $(\L,\Delta,S)$ forms a locality. It follows from Lemma~\ref{L:LocalitiesPropg}(a) and the definition of the restriction that $(\L,\Delta,S)$ is $\F$-natural. Observe that $(\L,\Delta,S)$ is of objective characteristic $p$ as $(\L^+,\Delta^+,S)$ is of objective characteristic $p$. So it follows from the choice of $(\L^+,\Delta^+,S)$, in particular from the minimality of $|\Gamma^+|$, that there exists a partial normal subgroup $\N$ of $\L$ such that
\[\N\cap S=T\mbox{ and }\F_T(\N)=\E|_{\Gamma}.\]
We fix such $\N$ throughout. Write $\Pi^+\colon\D^+\rightarrow\L^+$ and $\Pi\colon\D\rightarrow\L$ for the partial products on $\L^+$ and $\L$. Recall from the definition of the restriction that this means that $\D=\D_\Delta$ and $\Pi=\Pi^+|_\D$, where  $\D_\Delta$ is formed inside of $\L^+$.

\begin{property}\label{NSPDelta}
Let $P\in\Delta^+\backslash\Delta$. Then $N_T(P)\in\Gamma$ and $N_S(P)\in\Delta$.
\end{property}

\begin{proof}
Since $P\in\Delta^+\backslash\Delta$, we have $P\cap T\in\Gamma^+\backslash\Gamma=R^\F$. Moreover, as $T\in\Gamma$ by \ref{Gamma}, we have $T\not\leq P$ and thus $P$ is properly contained in the $p$-group $TP$. So $P<N_{TP}(P)=N_T(P)P$ and $P\cap T<N_T(P)$. Now \ref{Gamma0} gives  $N_T(P)\in\Gamma$ and hence $N_S(P)\in\Delta$.  
\end{proof}

Set
\[M:=N_{\L^+}(R).\]
By \ref{NormalSubgroup}, $M$ is a model for $N_\F(R)$, the subsystem $N_\E(R)$ is normal in $N_\F(R)$, and there exists a unique normal subgroup $K$ of $M$ such that \[K\cap N_S(R)=N_T(R)\mbox{ and }\F_{N_T(R)}(K)=N_\E(R).\] We fix such $K$ throughout.

\begin{property}\label{DoubleStar}
Let $Q\leq N_T(R)$ with $R<Q$. Then $N_{N_\N(Q)}(R)=N_K(Q)$.
\end{property}

\begin{proof}
As $R<Q\leq T$, it follows from \ref{Gamma0} that $Q\in \Gamma\subseteq \Delta$. As $M$ is a model for $N_\F(R)$, we can pick  $g\in M$ with $Q^g\in N_\F(R)^f$. Then $Q^g\leq N_S(R)\leq S$ and in particular $g\in \L=\L^+|_{\Delta}$. Moreover, by Lemma~\ref{LocalitiesProp}(b), 
\[N_\N(Q^g)^{g^{-1}}=N_\N(Q)\]
and $c_{g^{-1}}\colon N_\N(Q^g)\rightarrow N_\N(Q)$ is an isomorphism of groups. Note that $c_{g^{-1}}$ leaves $R$ invariant as $g\in M$. Hence, $N_{N_\N(Q^g)}(R)^{g^{-1}}=N_{N_\N(Q)}(R)$. As $K$ is normal in $M$, we have furthermore $N_K(Q^g)^{g^{-1}}=N_K(Q)$. So if $N_{N_\N(Q^g)}(R)=N_K(Q^g)$, then the assertion holds. Hence, replacing $Q$ by $Q^g$, it is enough to show that $N_{N_\N(Q)}(R)=N_K(Q)$ in the case that $Q$ is fully normalized in $N_\F(R)$. Thus we assume from now on that $Q\in N_\F(R)^f$. Then 
\[\F_0:=N_{N_\F(R)}(Q)\]
is saturated, and $N_M(Q)$ is a model for $\F_0$ by \cite[Proposition~I.5.4]{Aschbacher/Kessar/Oliver:2011} and Lemma~\ref{L:MSCharp}(b). In particular, \[S_0:=N_S(R)\cap N_S(Q)\in\Syl_p(N_M(Q)).\] 
As $N_\E(R)\unlhd N_\F(R)$, Lemma~\ref{LocalNormalSubsystems} yields that $Q\in N_\E(R)^f$ and that 
\[\E_0:=N_{N_\E(R)}(Q)\mbox{ is a normal subsystem of $\F_0$ over }T_0:=N_T(R)\cap N_T(Q).
 \]
 As $K$ is a model for $N_\E(R)$, it follows from \cite[Proposition~I.5.4]{Aschbacher/Kessar/Oliver:2011} and Lemma~\ref{L:MSCharp}(b) that $N_K(Q)$ is a model for $\E_0$. Note also that $N_K(Q)$ is a normal subgroup of $N_M(Q)$. By \cite[Theorem~II.7.5]{Aschbacher/Kessar/Oliver:2011}, there is a unique normal subgroup $N_0$ of $N_M(Q)$ such that $T_0\in \Syl_p(N_0)$ and $\F_{T_0}(N_0)=\E_0$. So since the corresponding properties hold with $N_K(Q)$ in place of $N_0$, it is sufficient to show that \[N:=N_{N_\N(Q)}(R)\] is a normal subgroup of $N_M(Q)$ such that $T_0\in\Syl_p(N)$ and $\F_{T_0}(N)=\E_0$. To see that, note first of all that $N_\L(Q)=N_{\L^+}(Q)$ as $Q\in\Delta$. Hence, $N_M(Q)=N_{N_{\L^+}(R)}(Q)=N_{N_{\L^+}(Q)}(R)=N_{N_\L(Q)}(R)$. As $\N$ is a partial normal subgroup of $\L$, it follows that $N=N_{N_\L(Q)}(R)\cap \N$ is a normal subgroup of $N_M(Q)$. In particular, as $S_0\in\Syl_p(N_M(Q))$, we have that $N\cap S_0=\N\cap S_0=(\N\cap S)\cap S_0=T\cap S_0=T_0$ is a Sylow $p$-subgroup of $N$. So it remains only to prove that $\F_{T_0}(N)=\E_0$. Let $n\in N$. Since $\F_T(\N)=\E|_\Gamma\subseteq\E$ and $\<Q,R\>\leq S_n\cap T_0$, it follows that the map $c_n\colon S_n\cap T_0\rightarrow T_0$ is a morphism in $\E$ normalizing $Q$ and $R$. Thus this map is a morphism in $\E_0$ showing $\F_{T_0}(N)\subseteq \E_0$. Every morphism in $\E_0$ extends to a morphism $\phi$ in $\m{E}$ acting on $Q$ and $R$. Appealing to Lemma~\ref{LocalitiesProp}(c) and using that $\E|_\Gamma=\F_T(\N)$ and $Q\in\Gamma$, one sees that such a morphism $\phi$ is realized as a conjugation map by an element of $N=N_{N_\N(Q)}(R)$. So we have $\E_0=\F_{T_0}(N)$. As argued above this yields the assertion. 
\end{proof}

\begin{property}\label{KinKernel}
 Let $g\in K$ and $R<Q\leq N_T(R)$ such that $Q^g\leq N_T(R)$. Then $g\in\N$. 
\end{property}

\begin{proof}
By Alperin's fusion theorem for groups \cite[Main Theorem]{Alperin:1967a} applied in the group $K$, there exist sequences of subgroups $Q=P_0,P_1,\dots,P_l=Q^g\leq N_T(R)$ and $Q_1,\dots,Q_l\leq N_T(R)$ as well as elements  $g_i\in N_K(Q_i)$ for $i=1,\dots,l$ such that $\<P_{i-1},P_i\>\leq Q_i$, $P_{i-1}^{g_i}=P_i$ and $g=g_1g_2\dots g_l$ (where we conjugate and form the latter product in $K$). Since the product in $K$ is the restriction of $\Pi^+$ to $\W(K)$, it follows that $g=\Pi^+(g_1,\dots,g_l)$. As $R$ is normal in $M$ and thus in $K$, we have $R\leq P_i\leq Q_i$ for $i=1,\dots,r$. Note that $R$ is properly contained in $Q_i$ as $|R|<|Q|=|P_i|\leq |Q_i|$ for $i=1,\dots,n$ . Thus, \ref{DoubleStar} gives $g_i\in N_K(Q_i)=N_{N_\N(Q_i)}(R)\subseteq \N$ for $i=1,\dots,l$. Furthermore, by \ref{Gamma0}, we have $Q\in\Gamma\subseteq\Delta$ and thus $(g_1,\dots,g_l)\in\D=\D_\Delta$ via $Q$. As $\Pi=\Pi^+|_\D$ and $\N$ is a partial subgroup of $\L$, it follows $g=\Pi^+(g_1,\dots,g_l)=\Pi(g_1,\dots,g_l)\in\N$.
\end{proof}

Recall that $\N$ is a partial normal subgroup of $\L$ with $\N\cap S=T$ and $\F_T(\N)=\m{E}|_\Gamma$. We consider now the  quotient locality  $\L/\N$ as described in Subsection~\ref{SS:Homomorphisms}. Let
\[\alpha\colon \L\rightarrow \L/\N\]
be the natural projection, set $\ov{\L}:=\L/\N$ and adapt the bar notation. Recall that $\ov{\L}$ is a partial group and $\alpha$ is a homomorphism of partial groups with $\ker(\alpha)=\N$. Write $\ov{\Pi}$ for the product on $\ov{\L}:=\L/\N$. Setting $\ov{\Delta}:=\{\ov{P}\colon P\in\Delta\}$, the triple $(\ov{\L},\ov{\Delta},\ov{S})$ is a locality by \cite[Corollary~4.5]{Chermak:2015}. As $T\in\Delta$, we have $\{\ov{\One}\}=\ov{T}\in\ov{\Delta}$. Hence, $\ov{\L}$ is a group and $\ov{\Delta}$ is the set of subgroups of $\ov{S}$.

\smallskip

Set now $Q_T:=N_T(R)$. By the Frattini argument, we have $M=N_M(Q_T)K$. Hence, there is a natural isomorphism
\[\mu\colon M/K\longrightarrow N_M(Q_T)/N_K(Q_T)\]
defined by $gK\mapsto gN_K(Q_T)$ for all $g\in N_M(Q_T)$. By \ref{NSPDelta}, we have $Q_T=N_T(R)\in\Gamma\subseteq\Delta$ and therefore $N_M(Q_T)=N_{N_{\L^+}(Q_T)}(R)=N_{N_\L(Q_T)}(R)$ is a subgroup of $\L$. Hence, $\alpha$ induces a group homomorphism $N_M(Q_T)\rightarrow \ov{\L}$ with kernel $N_{N_\N(Q_T)}(R)$. By \ref{DoubleStar}, we have $N_{N_\N(Q_T)}(R)=N_K(Q_T)$. Hence, $\alpha$ induces an an injective group homomorphism
\[\ov{\alpha}\colon N_M(Q_T)/N_K(Q_T)\longrightarrow \ov{\L}\mbox{ with }gN_K(Q_T)\mapsto g\alpha\mbox{ for all }g\in N_M(Q_T).\] 
Writing $\pi\colon M\rightarrow M/K$ for the natural epimorphism, we define now $\alpha_R:=\pi\mu\ov{\alpha}$ to be the composition of the maps 
\[M\xrightarrow{\;\pi\;}M/K\xrightarrow{\;\mu\;}N_M(Q_T)/N_K(Q_T)\xrightarrow{\;\ov{\alpha}\;}\ov{\L}.\]
Then $\alpha_R\colon M\rightarrow \ov{\L}$ is a group homomorphism. As $\mu$ and $\ov{\alpha}$ are injective, we have $\ker(\alpha_R)=\ker(\pi)=K$. 

\begin{property}\label{AmalgamGetHomApply}
$\alpha|_{N_\L(R)}=\alpha_R|_{N_\L(R)}$.
\end{property}

\begin{proof}
Let $y\in N_\L(R)$. As $y\in\L$, we have $S_y\in\Delta$ and thus $R<S_y\cap T$. Hence, 
\[R<P:=N_{S_y\cap T}(R)\leq N_T(R)=Q_T\]
and, by \ref{Gamma0}, $P\in\Gamma\subseteq\Delta$. Now $P^y$ is defined in $\L$ and $P^y\leq Q_T$. As $N_\L(R)\subseteq M=N_M(Q_T)K$, we can write $y=xk$ with $x\in N_M(Q_T)$ and $k\in K$ (where the product $xk$ is formed in $M=N_{\L^+}(R)$). Then we have $P^x\leq Q_T$ and $(P^x)^k=P^y\leq Q_T$. As $R=R^x<P^x$, it follows from \ref{KinKernel} that $k\in\N$ and thus $k\alpha=\ov{\One}$. Note also that $(x,k)\in\D=\D_\Delta$ via $P$ and so $y=\Pi(x,k)$. Hence, as $\alpha$ is a homomorphism of partial groups, we have $y\alpha=\ov{\Pi}(x\alpha,k\alpha)=\ov{\Pi}(x\alpha,\ov{\One})=x\alpha$. Notice that $y\pi=x\pi$ and $y\pi\mu=x\pi\mu=xN_K(Q_T)$. So, by definition of $\alpha_R$, we have $y\alpha_R=(xN_K(Q_T))\ov{\alpha}=x\alpha$. Hence $y\alpha=y\alpha_R$ and the assertion holds. 
\end{proof}

We are now in a position to reach the final contradiction: Notice that $R\alpha=\ov{R}=\{\ov{\One}\}\in\ov{\Delta}$ and $\alpha_R$ is a group homomorphism from $M=N_{\L^+}(R)$ to $\ov{\L}=N_{\ov{\L}}(R\alpha)$. As $R$ is fully $\F$-normalized and $\F_S(\L^+)=\F|_{\Delta^+}\subseteq\F$, the subgroup $R$ is also fully $\F_S(\L^+)$-normalized. Hence, by \ref{Gamma0} and  \ref{AmalgamGetHomApply}, the hypothesis of \cite[Lemma~3.1]{Henke:2020} is fulfilled with $(\ov{\L},\ov{\Delta},\ov{S})$ and $\{R\}$ in place of $(\widetilde{\L},\widetilde{\Delta},\widetilde{S})$ and $\Gamma_0$. So by this Lemma, there exists a homomorphism of partial group $\gamma\colon\L^+\rightarrow\ov{\L}$ with $\gamma|_\L=\alpha$ and $\gamma|_M=\alpha_R$. Now
\[\N^+:=\ker(\gamma)\]
is a partial normal subgroup of $\L^+$ by \cite[Lemma~3.3]{Chermak:2013}. Moreover, we have  
\[\N^+\cap\L=\ker(\gamma|_\L)=\ker(\alpha)=\N\]
and
\begin{equation}\label{E:NNplusRK}
N_{\N^+}(R)=M\cap\N^+=\ker(\alpha_R)=K.
\end{equation}
In particular $\N^+\cap S=(\N^+\cap\L)\cap S=\N\cap S=T$. To obtain a contradiction it is thus enough to show that $\F_T(\N^+)=\E|_{\Gamma^+}$.

\smallskip

For any $Q\in\Gamma^+$, set $\Aut_{\N^+}(Q):=\{c_n|_Q\colon n\in N_{\N^+}(Q)\}$. As $K$ is by assumption a model for $N_\E(R)$, it follows from \eqref{E:NNplusRK} that 
\begin{equation*}
\Aut_{\N^+}(R)=\Aut_K(R)=\Aut_\E(R).
\end{equation*}
Let $R^*\in R^\F$. By \ref{NSPDelta}, $N_S(R^*)\in\Delta$. As $R$ is fully normalized, by \cite[Lemma~2.6(c)]{Aschbacher/Kessar/Oliver:2011}, there exists a morphism in $\Hom_\F(N_S(R^*),S)$ taking $R^*$ to $R$. Hence,  by Lemma~\ref{L:LocalitiesPropg}(a), we may pick $g\in\L$ with $N_S(R^*)\leq S_g$ and $(R^*)^g=R$. It follows now from Lemma~\ref{LocalitiesProp}(b) applied with $(\L^+,\Delta^+,\N^+,R^*,R)$ in place of $(\L,\Delta,\N,P,Q)$ that 
\begin{equation*}
(c_g|_{R^*})^{-1}\Aut_{\N^+}(R^*)(c_g|_{R^*})=\Aut_{\N^+}(R)=\Aut_\E(R).
\end{equation*}
As $\E\unlhd\F$, this implies
\begin{equation}\label{E:AutR}
\Aut_{\N^+}(R^*)=\Aut_\E(R^*)\mbox{ for all }R^*\in R^\F.
\end{equation}
Recall now that $\N\unlhd\L$ is chosen such that $\E|_\Gamma=\F_T(\N)$. Moreover, if $Q\in\Gamma\subseteq\Delta$, then $N_{\N^+}(Q)=N_\N(Q)$ and $\Aut_{\N^+}(Q)=\{c_n|_Q\colon n\in N_\N(Q)\}$. So using Lemma~\ref{LocalitiesProp}(c), one sees that
\begin{equation}\label{E:AutQ}
\Aut_{\N^+}(Q)=\Aut_\E(Q)\mbox{ for every }Q\in\Gamma.
\end{equation}
By Alperin's fusion theorem for fusion systems \cite[Theorem~I.3.6]{Aschbacher/Kessar/Oliver:2011}, every morphism in $\E$ is the product of restrictions of $\E$-automorphisms of subgroups of $T$. As $\Gamma^+$ is overgroup-closed in $T$ and closed under $\E$-conjugacy by \ref{Gamma00}, this implies that $\E|_{\Gamma^+}$ is generated by the automorphism groups $\Aut_\E(Q)$ where  $Q\in \Gamma^+=\Gamma\cup R^\F$. Hence, $\E|_{\Gamma^+}\subseteq \F_T(\N^+)$ by \eqref{E:AutR} and \eqref{E:AutQ}. 

\smallskip

By Lemma~\ref{GammaDeltaLemma}, $(\N^+,\Gamma^+,T)$ is a locality. So by Alperin's fusion theorem for localities \cite{Molinier:2016}, every element $m\in\N^+$ can be written as a product $m=\Pi(n_1,\dots,n_k)$ of elements $n_i\in N_{\N^+}(Q_i)$ with $Q_i\in\Gamma^+$ ($1\leq i\leq k$) such that $S_m=S_{(n_1,\dots,n_k)}$. Hence, $\F_T(\N^+)$ is generated by the groups $\Aut_{\N^+}(Q)$ with $Q\in\Gamma^+$. So by \eqref{E:AutR} and \eqref{E:AutQ}, we have also $\F_T(\N^+)\subseteq \E|_{\Gamma^+}$ and thus $\F_T(\N^+)=\E|_{\Gamma^+}$. Since this contradicts the choice of $(\L^+,\Delta^+,S)$, the proof of Theorem~\ref{mainStepI} is thereby complete.

\subsection{A special case of the one-to-one correspondence}\label{StepIISection}

In this section, we show that Theorem~\ref{main} holds if $\Delta$ is, in a certain sense, large enough. More precisely, we prove the following theorem and corollary using Theorem~\ref{mainStepI}.

\begin{theorem}\label{mainStepII}
Let $\F$ be a saturated fusion system over $S$, and let $\E$ be a normal subsystem of $\F$ over $T\leq S$. Suppose $(\L,\Delta,S)$ is a proper locality over $\F$ such that 
\begin{equation}\label{star}\tag{$*$}
 P_1P_2O_p(\F)\in\Delta\mbox{ for all }P_1\in\E^{cr},\;P_2\in C_\F(\E)^{cr}.
\end{equation}
Then there exists a unique partial normal subgroup $\N$ of $\L$ such that $T=S\cap\N$ and $\F_T(\N)=\E$.
\end{theorem}

\begin{cor}\label{C:StepII}
Let $(\L,\Delta,S)$ be a proper locality over a fusion system $\F$ such that  
\[P\in\Delta\mbox{ for all }P\in\F^q\mbox{ with }O_p(\F)\leq P.\]
Let $\E$ be a normal subsystem of $\F$ over a subgroup $T\leq S$. Then there exists a unique partial normal subgroup $\N$ of $\L$ such that $T=S\cap\N$ and $\F_T(\N)=\E$.
\end{cor}

\begin{proof}[Proof of Corollary~\ref{C:StepII} from Theorem~\ref{mainStepII}]
Assume that $P\in\Delta$ for every $P\in\F^q$ with $O_p(\F)\leq P$. It is a consequence of Lemma~\ref{L:tE}(b),(c) that $P_1P_2\in\F^q$ for all $P_1\in\E^{cr}$ and $P_2\in C_\F(\E)^{cr}$. As $\F^q$ is overgroup-closed, it follows thus from our assumption that \eqref{star} holds. Hence the assertion follows from Theorem~\ref{mainStepII}.
\end{proof}

So it remains to prove Theorem~\ref{mainStepII}. For the remainder of this subsection we assume the hypothesis of Theorem~\ref{mainStepII}. To ease notation we set 
\[\F_1:=\m{E},\;S_1:=T,\;\F_2:=C_\F(\m{E})\mbox{ and }S_2=C_S(\m{E}).\]

\smallskip

For $i=1,2$, the fusion system $\F_i$ is a normal subsystem of $\F$ over $S_i$. By Lemma~\ref{L:tE}, $\F_1$ and $\F_2$ centralize each other and 
\[\tE:=\F_1*\F_2\] is a normal subsystem of $\F$ over $\tT:=S_1S_2$. Suppose now there exists a proper locality  $(\L,\Delta,S)$ over $\F$. Notice that $TC_S(\E)O_p(\L)\in\Delta$ by \eqref{star} and $C_S(\E)\leq C_S(T)$. As $\Delta$ is overgroup-closed in $S$, it follows $TC_S(T)O_p(\L)\in\Delta$. Thus, if a partial normal subgroup $\N$ of $\L$ with $T=\N\cap S$ and $\E=\F_T(\N)$ exists, then it is unique by Corollary~\ref{C:UniquenessLemmaStrong}. Hence it remains only to prove the existence of $\N$.

\smallskip

By \cite[Proposition~5]{Henke:2015}, we have $\L=N_\L(O_p(\F))$. Hence, setting $\tDelta:=\{P\leq S\colon PO_p(\F)\in\Delta\}$, it follows from Lemma~\ref{L:VaryObjects} that $(\L,\tDelta,S)$ is a proper locality over $\F$. By \eqref{star}, we have $P_1P_2\in\tDelta$ for all $P_1\in\F_1^{cr}$ and all $P_2\in\F_2^{cr}$. Hence, replacing $\Delta$ by $\tDelta$, we can and will assume from now on that
\begin{equation}\label{star2}\tag{$**$}
 P_1P_2\in\Delta\mbox{ for all }P_1\in\F_1^{cr},\;P_2\in \F_2^{cr}.
\end{equation}
By Lemma~\ref{L:tE}(b), this is equivalent to assuming that $\tE^{cr}\subseteq\Delta$. Write $\Delta_0$ for the set of subgroups of $S$ containing an element of $\tE^{cr}$.

\begin{property}\label{L:Delta0Delta}
We have $\F^{cr}\subseteq \Delta_0\subseteq\Delta$ and $\Delta_0$ is $\F$-closed. 
\end{property}

\begin{proof}
Since $\tE^{cr}$ is closed under $\F$-conjugacy by Lemma~\ref{L:EcFinvariant}, the set $\Delta_0$ is closed under $\F$-conjugacy. By construction, $\Delta_0$ is also overgroup-closed in $S$. Notice that $\Delta_0\subseteq\Delta$ as $\Delta$ is overgroup-closed and  $\tE^{cr}\subseteq\Delta$ by \eqref{star2} and Lemma~\ref{L:tE}(b). For every $R\in\F^{cr}$, \cite[Lemma~1.20(d)]{AOV1} gives $R\cap \tT\in \tE^{cr}$ and thus $R\in\Delta_0$ by construction of $\Delta_0$. 
\end{proof}

By \ref{L:Delta0Delta}, the restriction
\[\L_0:=\L|_{\Delta_0}\]
is well-defined and so $(\L_0,\Delta_0,S)$ is a locality of objective characteristic $p$ with $\F^{cr}\subseteq \Delta_0$. Alperin's fusion theorem \cite[Theorem~I.3.6]{Aschbacher/Kessar/Oliver:2011} implies that $\F_S(\L_0)=\F_S(\L)=\F$. So $(\L_0,\Delta_0,S)$ is a proper locality over $\F$.

\smallskip

If there exists a partial normal subgroup $\N_0$ of $\L_0$ with $T=\N_0\cap S$ and $\E=\F_T(\N_0)$ then, as $\E=\F_T(\N_0)$ is normal in $\F$, it follows from Theorem~\ref{T:VaryObjects}(b),(c) that there exists $\N\unlhd\L$ with $\N\cap \L_0=\N_0$, $T=\N\cap S$ and $\E=\F_T(\N)$. So replacing $(\L,\Delta,S)$ by $(\L_0,\Delta_0,S)$, we may assume from now on that $\Delta=\Delta_0$. Then the following property holds.

\begin{property}\label{DeltaOvergroupstEcr}
The set $\Delta$ is the set of overgroups of the elements of $\tE^{cr}$ in $S$. In particular,
\[\Gamma:=\{P\cap \tT\colon P\in\Delta\}\subseteq\Delta\]
is the set of overgroups of the elements of $\tE^{cr}$ in $\tT$. 
\end{property}

By \ref{DeltaOvergroupstEcr} the hypothesis of Theorem~\ref{mainStepI} is fulfilled with $\tE$ in place of $\E$. Hence, it follow from that theorem that there exists a partial normal subgroup $\M$ of $\L$ such that $\M\cap S=\tT$ and $\F_{\tT}(\M)=\tE|_\Gamma$. As $\tE^{cr}\subseteq\Gamma$, Alperin's fusion theorem \cite[Theorem~I.3.6]{Aschbacher/Kessar/Oliver:2011} yields that $\tE=\tE|_\Gamma=\F_{\tT}(\M)$. Hence, if follows from Lemma~\ref{GammaDeltaLemma}(b),(c) that $(\M,\Gamma,\tT)$ is a proper locality over $\tE$. 

\smallskip

For each $i=1,2$, write $\Delta_i$ for the set of overgroups of $\F_i^{cr}$ in $S_i$. Notice that $\Delta_i$ is $\F_i$-closed for $i=1,2$. 
As $\F_i^{cr}\subseteq\F_i^s$ and $\F_i^s$ is $\F_i$-closed by \cite[Proposition~3.3]{Henke:2015}, we have $\Delta_i\subseteq\F_i^s$. 
Observe also that, by Lemma~\ref{L:tE}(b) and \ref{DeltaOvergroupstEcr}, $\Gamma$ is the set of overgroups in $\tT$ of the subgroups of the form $P_1P_2$ with $P_i\in\Delta_i$ for $i=1,2$. It follows from Lemma~\ref{L:InternalExternalCentralProduct} that $\tE$ is an internal central product of $\E$ and $C_\F(\E)$ not only in the definition given in this paper, but also in the sense of \cite[Definition~2.9]{Henke:2016}. So by \cite[Proposition~6.12(b)]{Henke:2016} applied with $(\tE,\M,\Gamma,\tT)$ in place of $(\F,\L,\Delta,S)$, the locality $(\M,\Gamma,\tT)$ is (in the sense of \cite[Definition~6.6]{Henke:2016}) an internal central product of two sublocalities $(\L_1,\Delta_1,S_1)$ and $(\L_2,\Delta_2,S_2)$ of $(\M,\Gamma,\tT)$ such that $(\L_i,\Delta_i,S_i)$ is a proper locality over $\F_i$ for $i=1,2$. 

\smallskip

Saying that $(\L_i,\Delta_i,S_i)$ is a sublocality of $(\M,\Gamma,\tT)$ means here (according to \cite[Definition~3.15]{Henke:2016}) that $\L_i$ is a partial subgroup of $\M$ (and thus of $\L$), that $\tT\cap\L_i=S_i$ (and thus $S\cap\L_i=S_i$), and that $(\L_i,\Delta_i,S_i)$ forms a locality. 

\smallskip

We will not need the precise definition of what it means that $(\M,\Gamma,\tT)$ is an internal central product of $(\L_1,\Delta_1,S_1)$ and $(\L_2,\Delta_2,S_2)$, but we will only use that $\L_i\unlhd \M$ for $i=1,2$ by \cite[Lemma~6.10]{Henke:2016}. More precisely, we use that 
\[\N:=\L_1\unlhd\M.\] 
Note that 
\[T=S_1=\L_1\cap S=\N\cap S\mbox{ and }\F_T(\N)=\m{E}\] 
as $(\N,\Delta_1,T)=(\L_1,\Delta_1,S_1)$ is a locality over $\F_1=\E$. So it remains only to show that $\N\unlhd \L$. Let $f\in N_\L(\tT)$. As $\N\unlhd\M$, it is by \cite[Corollary~3.13]{Chermak:2015} enough to show that $n^f\in\N$ for every $n\in\N\cap\D(f)$. 

\smallskip

Observe that $\L=\D(f)$ and $c_f\in\Aut(\L)$ by \ref{DeltaOvergroupstEcr} and Lemma~\ref{GammaDeltaLemma}(a). In particular, as $\N\unlhd\M\unlhd\L$, it follows $\N^f\unlhd\M^f=\M$. Moreover, $c_f$ induces an isomorphism from $\N$ to $\N^f$. In particular, as $T$ is strongly closed and a maximal $p$-subgroup of $\N$, it follows that $T=T^f\leq S\cap \N^f$ is a maximal $p$-subgroup of $\N^f$. Hence  $T^f=T=\N^f\cap S$. Set $\Delta_1^f:=\{P^f\colon P\in\Delta_1\}$. As $(\N,\Delta_1,T)$ is a locality, the triple $(\N^f,\Delta_1^f,T)$ is a locality; this is a special case of \cite[Theorem~4.3]{Chermak:2015}, but can also easily be shown by direct arguments. Hence, by \cite[Lemma~2.21(b)]{Henke:2020} or by a direct argument, $c_f$ induces an isomorphism from $\E=\F_T(\N)$ to $\F_T(\N^f)$. Since $\alpha:=c_f|_T\in\Aut_\F(T)$ and $\E\unlhd\F$, this implies $\F_T(\N)=\E=\E^\alpha=\F_T(\N^f)$. As $(\M,\Gamma,T)$ is a proper locality and $TC_{\tT}(T)=\tT\in\Gamma$, it follows now from Corollary~\ref{C:UniquenessLemmaStrong} applied with $(\M,\Gamma,T)$ in place of $(\L,\Delta,S)$ that $\N^f=\N$. This shows the assertion.

\subsection{The general one-to-one correspondence}\label{S:StepIV}

In this section we will show that there is a natural one-to-one correspondence between the normal subsystems of $\F$ and the partial normal subgroups of an arbitrary proper locality over $\F$. In particular, parts (a) and (b) of Theorem~\ref{mainStepIV} below can be considered as a weak version of Theorem~\ref{main}. We will postpone the proof of the complete statement of Theorem~\ref{main}, since it uses results about components and $E(\F)$, which are known by work of Aschbacher \cite{Aschbacher:2011}, but which we want to reprove later on using Theorem~\ref{mainStepIV}. We moreover state some details which are not mentioned in our main theorems.

\smallskip

\textbf{Throughout this section, assume that $\F$ is a saturated fusion system over $S$. Write $\fN(\F)$ for the set of normal subsystems of  $\F$ and, given a partial group $\L$, write $\fN(\L)$ for the set of partial normal subgroups of $\L$.}

\smallskip

Given two proper localities $(\L,\Delta,S)$ and $(\L^+,\Delta^+,S)$ with $\Delta\subseteq\Delta^+$ and $\L=\L^+|_{\Delta}$, the map
\[\Phi_{\L^+,\L}\colon \fN(\L^+)\rightarrow \fN(\L),\;\N^+\mapsto \N^+\cap \L.\]
is well-defined and a bijection by Theorem~\ref{T:VaryObjects}(b). Against our usual convention we use the left-hand notation for the maps $\Phi_{\L^+,\L}$ and $\Psi_\L$ below. We write $\F^q_{O_p(\F)}$ for the set of all $P\in\F^q$ with $O_p(\F)\leq P$. Similarly, $\delta(\F)_{O_p(\F)}$ denotes the set of all $P\in\delta(\F)$ with $O_p(\F)\leq P$. The main results of this subsection are summarized in the following theorem.

\begin{theorem}\label{mainStepIV}
For every proper locality $(\L,\Delta,S)$ over $\F$, there is an inclusion-preserving bijection
\[\Psi_\L\colon \fN(\L)\rightarrow \fN(\F)\]
such that the following properties hold:
\begin{itemize}
 \item [(a)] For any $\N\in\fN(\L)$, the normal subsystem $\Psi_\L(\N)$ is a fusion system over $\N\cap S$. Furthermore, $\Psi_\L(\N)$ is the smallest normal subsystem of $\F$ containing $\F_{S\cap\N}(\N)$. 
 \item [(b)] If $\F^q_{O_p(\F)}\subseteq\Delta$ or $\delta(\F)_{O_p(\F)}\subseteq\Delta$, then $\Psi_\L(\N)=\F_{S\cap\N}(\N)$. 
 \item [(c)] If $(\L^+,\Delta^+,S)$ is another proper locality over $\F$ with $\Delta\subseteq \Delta^+$ and $\L=\L^+|_\Delta$, then $\Psi_{\L^+}=\Psi_\L\circ \Phi_{\L^+,\L}$.
 \item [(d)] If $\E_1,\E_2\in\fN(\F)$ with $\E_1\unlhd\E_2$, then $\Psi_\L^{-1}(\E_1)\subseteq\Psi_\L^{-1}(\E_2)$.
 \item [(e)] If $\N_1,\N_2\in\fN(\L)$ with $\N_1\subseteq\N_2$, then $\Psi_\L(\N_1)\unlhd\Psi_\L(\N_2)$.
 \item [(f)] If $R\leq S$ with $R\unlhd\L$, then $\Psi_\L(R)=\F_R(R)$.
 \item [(g)] Let $\alpha\in\Aut(\L,S)$ and $\N\in\fN(\L)$. Then $\alpha|_{\N\cap S}$ induces an isomorphism from $\Psi_\L(\N)$ to $\Psi_\L(\N\alpha)$. 
\end{itemize}
\end{theorem}

We caution the reader that the assumption that $\E_1\unlhd\E_2$ in Theorem~\ref{mainStepIV}(d) is actually important, since $\Psi_\L^{-1}$ need not be inclusion-preserving as the following example shows.

\begin{ex}\label{E:1}
Let $G=G_1\times G_2$ with $G_1\cong G_2\cong A_4$. Setting $T_i:=O_p(G_i)$ for $i=1,2$ and $S:=T_1\times T_2$, we have $S\in\Syl_p(G)$. Set $\F=\F_S(G)$ and set $\Delta:=\F^s$. Note that $S=O_p(G)\in\Delta$. As $G$ is of characteristic $p$, $\Delta$ is the set of all subgroups of $S$ and $(G,\Delta,S)$ is a proper locality over $\F$ (e.g. by Lemma~\ref{L:MSCharp}(b) or by a direct argument). The partial normal subgroups of this locality correspond to the normal subgroups of the group $G$. However there are two normal subgroups $M$ and $N$ of $G$ such that $\F_{S\cap M}(M)\subseteq \F_{S\cap N}(N)$ and $M\not\leq N$. Namely, for $i=1,2$ fix $d_i\in G_i$ of order $3$. Take $M:=G_1$ and $N:=S\<d_1d_2\>$.  
\end{ex}

The remainder of this section is devoted to the proof of Theorem~\ref{mainStepIV}. We will build on Theorem~\ref{T:VaryObjects}, Theorem~\ref{T:mainRegularPartialNormal}(c) and the following reformulation of Corollary~\ref{C:StepII}.

\begin{lemma}\label{N*}
Let $(\L,\Delta,S)$ be a proper locality over $\F$ with $\F^q_{O_p(\F)}\subseteq \Delta$. Then there exists a map \[\Theta\colon \fN(\F)\rightarrow \fN(\L),\;\m{E}\mapsto \Theta(\m{E})\]
such that $\m{E}=\F_{\Theta(\m{E})\cap S}(\Theta(\m{E}))$. Moreover, $\Theta$ is injective and so $|\fN(\F)|\leq |\fN(\L)|$. 
\end{lemma}

\begin{proof}
By Corollary~\ref{C:StepII}, for every $\m{E}\in\fN(\F)$, there exists a unique partial normal subgroup $\Theta(\m{E})\in\fN(\L)$ with $\m{E}=\F_{\Theta(\m{E})\cap S}(\Theta(\m{E}))$. This shows that the map $\Theta$ exists. If $\m{E},\m{E}'\in\fN(\F)$ with $\Theta(\m{E})=\Theta(\m{E}')$ then $\m{E}=\F_{\Theta(\m{E})\cap S}(\Theta(\m{E}))= \F_{\Theta(\m{E}')\cap S}(\Theta(\m{E}'))=\m{E}'$. So $\Theta$ is injective, which implies that $|\fN(\F)|\leq |\fN(\L)|$.
\end{proof}

\begin{lemma}\label{RegularNormalMap}
 Let $(\L,\Delta,S)$ be a proper locality over $\F$ with $\delta(\F)_{O_p(\F)}\subseteq \Delta$. Then the map
\[\Psi_\L\colon \fN(\L)\rightarrow \fN(\F),\;\N\mapsto \F_{S\cap \N}(\N)\]
is well-defined and a bijection. In particular, $|\fN(\L)|=|\fN(\F)|$. Moreover, if $\N\unlhd\L$ and $T=\N\cap S$, then $TC_S(T)O_p(\L)\in\Delta$.
\end{lemma}

\begin{proof}
We will use throughout that $\L=N_\L(O_p(\F))$ and $O_p(\F)=O_p(\L)$ by \cite[Proposition~5]{Henke:2015} and Lemma~\ref{L:OpL}. Set $\L_\delta:=\L|_{\delta(\F)_{O_p(\F)}}$ so that $(\L_\delta,\delta(\F)_{O_p(\F)},S)$ is a proper locality. It follows from \cite[Lemma~10.6]{Henke:Regular} that $\delta(\F)=\{P\leq S\colon PO_p(\F)\in\delta(\F)_{O_p(\F)}\}$. Hence, by Lemma~\ref{L:VaryObjects}, the triple $(\L_\delta,\delta(\F),S)$ is a proper locality and thus a regular locality.

\smallskip

Let $\N\in\fN(\L)$ and set $T:=S\cap \N$. Then $\N_\delta=\N\cap \L_\delta\unlhd\L_\delta$ and $\N_\delta\cap S=T$.  Theorem~\ref{T:mainRegularPartialNormal}(c) yields $TC_S(T)\in\delta(\F)$ and $\F_T(\N_\delta)\in\fN(\F)$. In particular $\F_T(\N_\delta)$ is $\F$-invariant. By Theorem~\ref{T:VaryObjects}(c) this implies $\F_T(\N)=\F_T(\N_\delta)\in\fN(\F)$. Thus, $\Psi_\L$ is well-defined. As $TC_S(T)\in\delta(\F)$, we have  $TC_S(T)O_p(\L)\in\delta(\F)_{O_p(\F)}\subseteq\Delta$. Hence, it follows from Corollary~\ref{C:UniquenessLemmaStrong} that $\Psi_\L$ is injective. In particular, $|\fN(\L)|\leq |\fN(\F)|$.

\smallskip

By Theorem~\ref{T:VaryObjects}(d), there exists a proper locality $(\L^s,\F^s,S)$ over $\F$ with $\L^s|_\Delta=\L$. Then Theorem~\ref{T:VaryObjects}(b) implies $|\fN(\L^s)|=|\fN(\L)|$ and  Lemma~\ref{N*} shows $|\fN(\F)|\leq |\fN(\L^s)|$. So $|\fN(\F)|\leq |\fN(\L^s)|=|\fN(\L)|\leq |\fN(\F)|$. Thus, all the inequalities are equalities and $\Psi_\L$ is a bijection. 
\end{proof}

\begin{lemma}\label{BijectionFsFqdelta}
 Let $(\L,\Delta,S)$ be a proper locality over $\F$ such that $\F^q_{O_p(\F)}\subseteq\Delta$ or $\delta(\F)_{O_p(\F)}\subseteq\Delta$. Define
\[\Psi_\L\colon \fN(\L)\rightarrow \fN(\F),\;\N\mapsto \F_{S\cap \N}(\N)\]
Then the following hold:
\begin{itemize}
\item [(a)] The map $\Psi_\L$ is well-defined and an inclusion-preserving bijection. In particular, $\F_{S\cap\N}(\N)$ is a normal subsystem of $\F$ for every partial normal subgroup $\N$ of $\L$.
\item [(b)] If $\E_1,\E_2\in\fN(\F)$ with $\E_1\unlhd \E_2$ then $\Psi_\L^{-1}(\E_1)\subseteq \Psi_\L^{-1}(\E_2)$.
\end{itemize}
\end{lemma}

\begin{proof}
\textbf{(a)} Clearly $\Psi_\L$ is inclusion-preserving if it is well-defined. Hence, if $\delta(\F)_{O_p(\F)}\subseteq\Delta$, then (a) follows from Lemma~\ref{RegularNormalMap}. So for the proof of (a) we may assume  $\F^q_{O_p(\F)}\subseteq\Delta$ and only need to prove that $\Psi_\L$ is a well-defined bijection. By Theorem~\ref{T:VaryObjects}(d), there exists a proper locality $(\L^s,\F^s,S)$ over $\F$ with $\L^s|_\Delta$ and by Theorem~\ref{T:VaryObjects}(b), we have $|\fN(\L)|=|\fN(\L^s)|$. Hence, Lemma~\ref{RegularNormalMap} implies $|\fN(\L)|=|\fN(\L^s)|=|\fN(\F)|$. Thus, the injective map $\Theta\colon \fN(\F)\rightarrow \fN(\L)$ from Lemma~\ref{N*} is a bijection. As $\m{E}=\F_{\Theta(\m{E})\cap S}(\Theta(\m{E}))$ for every $\m{E}\in\fN(\F)$, the inverse of $\Theta$ must be the map $\Psi_\L$, which is thus in particular well-defined and a bijection. This proves (a). 

\smallskip

\textbf{(b)} Part (b) follows now from Lemma~\ref{L:E1unlhdE2Help} provided we can show that $TC_S(T)O_p(\L)\in\Delta$ for every $\N\unlhd\L$ and $T=\N\cap S$. In the case $\delta(\F)_{O_p(\F)}\subseteq \Delta$, this property was shown in Lemma~\ref{RegularNormalMap}. As $T$ is strongly closed, $TC_S(T)O_p(\L)\in\F^c\subseteq \F^q$. By \cite[Proposition~5]{Henke:2015} and Lemma~\ref{L:OpL}, we have $O_p(\F)=O_p(\L)$. Hence, if $\F^q_{O_p(\F)}\subseteq\Delta$, then also $TC_S(T)O_p(\L)\in\Delta$.
\end{proof}

We show now the following preliminary version of Theorem~\ref{mainStepIV}.

\begin{prop}\label{P:MainStepIVPrelim}
For every proper locality $(\L,\Delta,S)$ over $\F$, there exists a proper locality $(\L^s,\F^s,S)$ over $\F$ such that $\L^s|_\Delta=\L$. Given such $(\L,\Delta,S)$ and $(\L^s,\F^s,S)$, the map
\[\Psi_{\L^s}\colon \fN(\L^s)\longrightarrow \fN(\F),\;\N^s\mapsto \F_{S\cap\N^s}(\N^s)\]
is a well-defined bijection, and the map $\Psi_\L:=\Psi_{\L^s}\circ \Phi_{\L^s,\L}^{-1}$ is a well-defined inclusion-preserving  bijection  $\fN(\L)\rightarrow \fN(\F)$. Moreover, the following hold:
\begin{itemize}
 \item [(a)] For any $\N\in\fN(\L)$, the normal subsystem $\Psi_\L(\N)$ is a fusion system over $S\cap\N$. Furthermore, $\Psi_\L(\N)$ is the smallest normal subsystem of $\F$ containing $\F_{S\cap\N}(\N)$. In particular, the map $\Psi_\L$ is characterized by this property and thus independent of the choice of $\L^s$. 
 \item [(b)] If $\F^q_{O_p(\F)}\subseteq\Delta$ or $\delta(\F)_{O_p(\F)}\subseteq\Delta$, then $\Psi_\L(\N)=\F_T(\N)$ for every $\N\in\fN(\L)$. 
 \item [(c)] If $\E_1,\E_2\in\fN(\F)$ with $\E_1\unlhd\E_2$, then $\Psi_\L^{-1}(\E_1)\subseteq\Psi_\L^{-1}(\E_2)$.
\end{itemize}  
\end{prop}

\begin{proof}
Given a proper locality $(\L,\Delta,S)$, a proper locality $(\L^s,\F^s,S)$ over $\F$ with $\L^s|_\Delta=\L$ exists by Theorem~\ref{T:VaryObjects}(d). The map $\Psi_{\L^s}$ is a well-defined inclusion-preserving bijection by Lemma~\ref{BijectionFsFqdelta}(a). Moreover, according to Theorem~\ref{T:VaryObjects}(b), $\Phi_{\L^s,\L}$ is an inclusion-preserving bijection such that $\Phi_{\L^s,\L}^{-1}$ is also inclusion-preserving. In particular, $\Psi_\L:=\Psi_{\L^s}\circ \Phi_{\L^s,\L}^{-1}$ is a well-defined inclusion-preserving bijection. 

\smallskip

\textbf{(c)} By Lemma~\ref{BijectionFsFqdelta}, we have $\Psi_{\L^s}^{-1}(\E_1)\subseteq\Psi_{\L^s}^{-1}(\E_2)$ for any $\E_1,\E_2\in\fN(\F)$ with $\E_1\unlhd\E_2$. So (c) follows from the fact that $\Phi_{\L^s,\L}$ is inclusion-preserving. 

\smallskip

\textbf{(a,b)} Note that (b) follows from (a) and Lemma~\ref{BijectionFsFqdelta}(a). Hence, it remains only to prove (a). Let $\N\unlhd\L$, $T:=\N\cap S$ and $\N^s:=\Phi_{\L^s,\L}^{-1}(\N)\unlhd\L^s$. Then $\N^s\cap\L=\N$ and so $\N^s\cap S=T$. In particular,  $\E:=\Psi_\L(\N)=\F_T(\N^s)$ is a normal subsystem of $\F$ over $T$. Let now $\tE$ be a normal subsystem of $\F$ with $\F_T(\N)\subseteq \tE$. We need to show that $\E\subseteq\tE$. 

\smallskip

Note that $\tN^s:=\Psi_{\L^s}^{-1}(\tE)$ is a partial normal subgroup of $\L^s$ with $\F_{\tN^s\cap S}(\tN^s)=\tE$. It follows that $T\leq \tN^s\cap S$. In particular,
\[\M:=\tN^s\cap \N^s\]
is a partial normal subgroup of $\L^s$ with $\M\cap S=T$. As $\Psi_{\L^s}$ is an inclusion-preserving bijection, it follows that $\m{D}:=\F_T(\M)$ is normal in $\F$, $\m{D}\subseteq\F_T(\N^s)=\E$ and $\m{D}\subseteq\F_{S\cap\tN^s}(\tN^s)=\tE$. In particular, $\m{D}$ is $\F$-invariant. From the equivalent condition for $\F$-invariance given in \cite[Proposition~I.6.4(d)]{Aschbacher/Kessar/Oliver:2011}, it follows thus that $\m{D}$ is $\E$-invariant. So from the Frattini condition (stated in \cite[Definition~I.6.1]{Aschbacher/Kessar/Oliver:2011}), we can conclude that $\m{E}=\<\m{D},\Aut_{\m{E}}(T)\>$. 

\smallskip

As $\E=\F_T(\N^s)$, the group $\Aut_\E(T)$ is generated by the automorphisms of $T$ which are obtained by conjugation by elements $N_{\N^s}(T)$. Moreover, by \cite[Lemma~3.34]{Henke:Regular}, we have $N_{\N^s}(T)=N_\N(T)$. Hence, $\Aut_{\m{E}}(T)\subseteq\F_T(\N)\subseteq\tE$ and so $\m{E}=\<\m{D},\Aut_{\m{E}}(T)\>\subseteq\tE$. This proves the assertion.
\end{proof}

\begin{proof}[Proof of Theorem~\ref{mainStepIV}]
\textbf{(a,b,c,d)} By Proposition~\ref{P:MainStepIVPrelim}, there exists a unique inclusion-preserving bijection $\Psi_\L\colon\fN(\N)\rightarrow\fN(\F)$ such that properties (a),(b),(d) of Theorem~\ref{mainStepIV} hold; moreover, if $(\L^s,\F^s,S)$ is a proper locality over $\F$ with $\L^s|_\Delta=\L$, then $\Psi_\L=\Psi_{\L^s}\circ\Phi_{\L^s,\L}^{-1}$. If $(\L^+,\Delta^+,S)$ is a proper locality over $\F$ with $\Delta\subseteq\Delta^+$ and $\L=\L^+|_\Delta$, then we can choose the proper locality $(\L^s,\F^s,S)$ over $\F$ such that $\L^s|_{\Delta^+}=\L^+$. We have then $\Psi_{\L^+}=\Psi_{\L^s}\circ\Phi_{\L^s,\L^+}^{-1}$. As $\Phi_{\L^s,\L}=\Phi_{\L^+,\L}\circ\Phi_{\L^s,\L^+}$, we have $\Phi_{\L^s,\L^+}^{-1}=\Phi_{\L^s,\L}^{-1}\circ \Phi_{\L^+,\L}$ and hence 
\[\Psi_{\L^+}=\Psi_{\L^s}\circ\Phi_{\L^s,\L^+}^{-1}=\Psi_{\L^s}\circ\Phi_{\L^s,\L}^{-1}\circ \Phi_{\L^+,\L}=\Psi_\L\circ \Phi_{\L^+,\L}.
\]
This proves part (c) of Theorem~\ref{mainStepIV}. 

\smallskip

\textbf{(e)} For the proof of (e) we use that, by Lemma~\ref{L:Regular10.4}, the set $\delta(\F)$ is $\F$-closed with $\F^{cr}\subseteq\delta(\F)\subseteq\F^s$. So setting $\L_\delta=\L^s|_{\delta(\F)}$, the triple $(\L_\delta,\delta(\F),S)$ is a regular locality over $\F$. Recall that, by Theorem~\ref{T:VaryObjects}(b), $\Phi_{\L^s,\L}\colon\fN(\L^s)\rightarrow\fN(\L)$ is an inclusion-preserving  bijection such that $\Phi_{\L^s,\L}^{-1}$ is also inclusion-preserving. Using this, one sees that (e) is true if and only if it is true with $(\L^s,\F^s,S)$ in place of $(\L,\Delta,S)$. Similarly one can show that (e) is true for $(\L_\delta,\delta(\F),S)$ if and only if it is true for $(\L^s,\F^s,S)$. Hence, without loss of generality, we may assume for the proof of (e) that $(\L,\Delta,S)$ is a regular locality. Let $\N_1$ and $\N_2$ be two partial normal subgroups of $\L$ with $\N_1\subseteq\N_2$. By Theorem~\ref{T:mainRegularPartialNormal}(a), $\N_2$ is a regular locality over $\F_{S\cap\N_2}(\N_2)$. Since  $\N_1\unlhd\N_2$, it follows thus from Theorem~\ref{T:mainRegularPartialNormal}(c) applied with $\N_2$ in place of $\L$ that $\F_{S\cap\N_1}(\N_1)$ is normal in $\F_{S\cap\N_2}(\N_2)$. By part (b) of Theorem~\ref{mainStepIV}, we have $\Psi_\L(\N_i)=\F_{S\cap\N_i}(\N_i)$ for $i=1,2$. This shows part (e) of Theorem~\ref{mainStepIV}.

\smallskip

\textbf{(f)} Let now $R\leq S$ with $R\unlhd\L$. Then $R\unlhd\F=\F_S(\L)$. From this it is easy to check that $\F_R(R)\unlhd\F$. It follows thus from (a) that $\Psi_\L(R)=\F_R(R)$. This proves (f). 

\smallskip

\textbf{(g)} For the proof of (g) let now $\alpha\in\Aut(\L,S)$ and $\N\in\fN(\L)$. Notice that $\N\alpha\in\fN(\L)$. Set $\E:=\Psi_\L(\N)$ and $\beta:=\alpha|_S$. It is easy to check that $\beta\in\Aut(\F)$ and thus $\E^\beta\unlhd\F$ as $\E\unlhd\F$. Moreover, $\alpha|_\N\colon \N\rightarrow \N\alpha$ is an isomorphism of partial groups and $(S\cap\N)\alpha=S\alpha\cap \N\alpha=S\cap\N\alpha$. From this it is easy to see that $\alpha|_{\N\cap S}=\beta|_{\N\cap S}$ induces an isomorphism from $\F_{S\cap \N}(\N)$ to $\F_{S\cap \N\alpha}(\N\alpha)$. As $\E$ is by (a) the smallest normal subsystem of $\F$ containing $\F_{S\cap \N}(\N)$, it follows that $\E^\alpha$ is the smallest normal subsystem of $\F$ containing $\F_{S\cap \N}(\N)^\beta=\F_{S\cap \N\alpha}(\N\alpha)$. Hence, again by (a), we have $\E^\beta=\Psi_\L(\N\alpha)$, i.e. $\beta$ induces an isomorphism from $\E$ to $\Psi_\L(\N\alpha)$. This proves (g) and completes thus the proof of the theorem. 
\end{proof}

\section{Towards a more comprehensive dictionary}

In the previous section we showed that there is a one-to-one correspondence between the normal subsystems of a fusion system and the partial normal subgroups of an associated proper locality. We will use this now to establish a more comprehensive ``dictionary'' showing how concepts in fusion systems translate to concepts in associated proper localities. Along the way, we also reprove a theorem of Aschbacher \cite[Theorem~1]{Aschbacher:2011} about the existence of ``normal intersections'' of normal subsystems (see Theorem~\ref{T:Intersections}). 

\smallskip

\textbf{Throughout this section let $\F$ be a saturated fusion system over $S$. As before, write $\fN(\F)$ for the set of normal subsystems of $\F$ and $\fN(\L)$  for the set of partial normal subgroups of any given partial group $\L$. If $(\L,\Delta,S)$ is a proper locality over $\F$, then $\Psi_\L\colon\fN(\L)\rightarrow\fN(\F)$ will always denote the map from Theorem~\ref{mainStepIV}.}

\smallskip

The reader should observe that the map $\Psi_\L$ from Theorem~\ref{mainStepIV} must be the same as the map $\Psi_\L$ in Theorem~\ref{main}. Thus it makes sense to prove Theorem~\ref{T:Intersections} at this stage, even though the proof of Theorem~\ref{main} is not yet complete.

\subsection{Intersections of partial normal subgroups}\label{SS:IntersectionsNormal}

In any partial group, the intersection of partial normal subgroups is trivially a partial normal subgroup. Hence, the results from the previous section imply easily the existence of ``normal intersections'' of normal subsystems of fusion systems. As a first step we prove the following proposition.

\begin{prop}\label{P:Intersections}
Let $(\L,\Delta,S)$ be a proper locality over $\F$. Let $I$ be an index set such that, for any $i\in I$, we are given $\N_i\in\fN(\L)$. Set 
\[\M:=\bigcap_{i\in I}\N_i\] 
and for $i\in I$ set  $\E_i:=\Psi_\L(\N_i)$. Then the following hold:
\begin{itemize}
\item [(a)] $\E_i$ is a normal subsystem of $\F$ over $T_i:=S\cap \N_i$ for all $i\in I$. 
\item [(b)] $\M$ is a partial normal subgroup of $\L$ and $\Psi_\L(\M)$ is a normal subsystem of $\F$ over $\M\cap S=\bigcap_{i\in I}T_i$ which is contained in $\bigcap_{i\in I}\E_i$.
\item [(c)] $\Psi_\L(\M)$ is the largest normal subsystem of $\F$ which is normal in $\E_i$ for all $i\in I$. Indeed, every normal subsystem of $\F$ which is normal in $\E_i$ for all $i\in I$ is also normal in $\Psi_\L(\M)$. 
\end{itemize}  
\end{prop}

\begin{proof}
Clearly, $\M$ is a partial normal subgroup of $\L$. It follows from Theorem~\ref{mainStepIV}(a) that $\E_i$ is a normal subsystem of $\F$ over $T_i=\N_i\cap S$ for all $i\in I$, and that 
\[\E:=\Psi_\L(\M)\]
is a normal subsystem of $\F$ over $\M\cap S=\bigcap_{i\in I}T_i$. Moreover, as $\M\subseteq\N_i$, Theorem~\ref{mainStepIV}(e) gives that $\E\unlhd\E_i$ for all $i\in I$. In particular, $\E\subseteq\bigcap_{i\in I}\E_i$. So (a) and (b) hold, and for (c) we only need to show that, given a normal subsystem $\m{D}$ of $\F$ which is normal in $\E_i$ for all $i\in I$, we have $\m{D}\unlhd\E$. Pick $\m{D}$ with $\m{D}\unlhd\E_i$ for all $i\in I$ and observe that, by Theorem~\ref{mainStepIV}(d),
\[\K:=\Psi_\L^{-1}(\m{D})\subseteq\Psi_\L^{-1}(\E_i)=\N_i\mbox{ for all }i\in I.\]
Hence, $\K\subseteq \bigcap_{i\in I}\N_i=\M$ and, again by Theorem~\ref{mainStepIV}(e), $\m{D}=\Psi_\L(\K)\unlhd\Psi_\L(\M)=\E$. 
\end{proof}

We will now show the following generalization of Theorem~\ref{T:Intersections}. The reader might want to note that Theorem~\ref{T:IntersectionsI} could also be obtained the other way around from Theorem~\ref{T:Intersections} by induction as $I$ can be assumed to be finite.

\begin{theorem}\label{T:IntersectionsI}
Let $\F$ be a saturated fusion system over $S$. Then for every family $(\E_i)_{i\in I}$ of normal subsystems of $\F$ there exists a normal subsystem $\bigwedge_{i\in I}\E_i$ (denoted by $\E_1\wedge\E_2\wedge\cdots\wedge\E_k$ if $I=\{1,2,\dots,k\}$) such that the following hold, whenever $I$ is an index set and $\E_i$ is a normal subsystems of $\F$ over $T_i$ for all $i\in I$:
\begin{itemize}
 \item [(a)] The subsystem $\bigwedge_{i\in I}\E_i$ is a normal subsystem of $\F$ over $\bigcap_{i\in I}T_i$ contained in $\bigcap_{i\in I}\E_i$; moreover $\bigwedge_{i\in I}\E_i$ is the largest normal subsystem of $\F$ which is normal in $\E_i$ for all $i\in I$.
 \item [(b)] Every normal subsystem of $\F$, which is normal in $\E_i$ for all $i\in I$, is also normal in $\bigwedge_{i\in I}\E_i$.  
 \item [(c)] If $I$ is the disjoint union of subsets $I_1,\dots,I_k$, then 
\[\bigwedge_{i\in I}\E_i=\left(\bigwedge_{i\in I_1}\E_i\right)\wedge\left(\bigwedge_{i\in I_2}\E_i\right)\wedge\cdots \wedge \left(\bigwedge_{i\in I_k}\E_i\right).\]  
 \item [(d)] If $(\L,\Delta,S)$ is a proper locality over $\F$, then for any collection of  partial normal subgroups $\N_i$ of $\L$ ($i\in I$), we have \[\Psi_\L(\bigcap_{i\in I}\N_i)=\bigwedge_{i\in I}\Psi_\L(\N_i).\]
\end{itemize}
\end{theorem}

\begin{proof}
By Theorem~\ref{T:ProperLocalityExistence}, we may choose a proper locality $(\L,\Delta,S)$ over $\F$. Set
\[\bigwedge_{i\in I}\E_i:=\Psi_\L(\bigcap_{i\in I}\Psi_\L^{-1}(\E_i)).\]
With this definition, parts (a) and (b) hold by Proposition~\ref{P:Intersections}. In particular, the definition of $\bigwedge_{i\in I}\E_i$ is independent of the choice of $\L$. So part (d) holds by definition of $\bigwedge_{i\in I}\E_i$. Part (c) holds as taking the intersection of sets is an associative operation.  
\end{proof}

\subsection{Index prime to $p$ and $p$-power index}\label{SS:ResiduesFusionLoc}

The reader is referred to Section~\ref{SS:Fusionppowerindex} for the definition of $O^p(\F)$ and subsystems of $\F$ of $p$-power index. Write $\mathbb{D}^p=\mathbb{D}^p_\F$ for the set of normal subsystems of $\F$ of $p$-power index. Using the language introduced in the previous subsection we have the following lemma.

\begin{lemma}\label{L:OupperpFBigCap}
$O^p(\F)=\bigcap_{\m{D}\in\mathbb{D}^p}\m{D}=\bigwedge_{\m{D}\in\mathbb{D}^p}\m{D}$.
\end{lemma}

\begin{proof}
 Recall from Section~\ref{SS:Fusionppowerindex} (or from \cite[Theorem~I.7.4]{Aschbacher/Kessar/Oliver:2011}) that $O^p(\F)\in\mathbb{D}^p$ and $O^p(\F)\subseteq\m{D}$ for every $\m{D}\in\mathbb{D}^p$. Hence, $O^p(\F)\subseteq\bigcap_{\m{D}\in\mathbb{D}^p}\m{D}\subseteq O^p(\F)$ and thus $O^p(\F)=\bigcap_{\m{D}\in\mathbb{D}^p}\m{D}$. We use the characterization of $\bigwedge_{\m{D}\in\mathbb{D}^p}\m{D}$ given in Theorem~\ref{T:IntersectionsI}(a). Clearly $O^p(\F)=\bigcap_{\m{D}\in\mathbb{D}^p}\m{D}\supseteq \bigwedge_{\m{D}\in\mathbb{D}^p}\m{D}$. Notice that $O^p(\F)\unlhd\F$ and, by Lemma~\ref{L:pPowerIndexProducts}, $O^p(\F)=O^p(\m{D})\unlhd\m{D}$ for all $\m{D}\in\mathbb{D}^p$. Hence, it follows also $O^p(\F)\subseteq\bigwedge_{\m{D}\in\mathbb{D}^p}\m{D}$.
\end{proof}

A subsystem $\E$ of $\F$ is said to be of \emph{index prime to $p$} if it is a subsystem over $S$ and $O^{p^\prime}(\Aut_\F(P))\leq \Aut_\E(P)$ for every $P\leq S$. If $\E\unlhd\F$, then one checks easily that $\E$ is of index prime to $p$ if and only if $\E$ is a fusion system over $S$. Note that it follows easily from Theorem~\ref{T:Intersections}(a) (or from the similar statement in \cite[Theorem~1]{Aschbacher:2011}) that there is a unique smallest normal subsystem of $\F$ of index prime to $p$. Namely, if $\mathbb{D}^{p^\prime}=\mathbb{D}^{p^\prime}_\F$ denotes the set of all normal subsystems of $\F$ of index prime to $p$, then
\[O^{p^\prime}(\F):=\bigwedge_{\m{D}\in\mathbb{D}^{p^\prime}}\m{D}\]
is a normal subsystem of $\F$ of $p$-power index which is contained in $\bigcap_{\m{D}\in\mathbb{D}^{p^\prime}}\m{D}$ and thus the unique smallest element of $\mathbb{D}^{p^\prime}$ with respect to inclusion. It follows from \cite[Theorem~I.7.7]{Aschbacher/Kessar/Oliver:2011}, that $O^{p^\prime}(\F)$ as defined above is the smallest saturated subsystem of $\E$ of index prime to $p$ and thus coincides with the equally denoted subsystem from that theorem. We will use this property in the proof of Lemma~\ref{L:ComponentsResiduesFusion} below, but it is not needed in this section and in the proofs of our main results as long as we use the definition of $O^{p^\prime}(\F)$ above. 

\smallskip

For the results we state now, the reader might want to recall Definition~\ref{D:OpLOpprimeL}.

\begin{prop}\label{Oupperp}
Let $(\L,\Delta,S)$ be a proper locality over $\F$. If $\N\in\fN(\L)$ and $\E:=\Psi_\L(\N)$, then the following hold:
\begin{itemize}
 \item[(a)] $\N$ has index prime to $p$ if and only if $\E$ has index prime to $p$.
 \item[(b)] $\N$ has $p$-power index if and only if $\E$ has $p$-power index.
 \item[(c)] $\Psi_\L(O^p_\L(\N))=O^p(\E)$ and $\Psi_\L(O^{p^\prime}_\L(\N))=O^{p^\prime}(\E)$. In particular, $\Psi_\L(O^p(\L))=O^p(\F)$ and $\Psi_\L(O^{p^\prime}(\L))=O^{p^\prime}(\F)$.
\end{itemize}
\end{prop}

\begin{proof}
\textbf{(a)} By Theorem~\ref{mainStepIV}(a), $\E$ is a fusion system over $S\cap\N$. So $S\subseteq\N$ if and only if $\E$ is a fusion system over $S$, i.e. (a) holds. 

\smallskip

\textbf{(b,c)} For the proofs of (b) and (c) consider first the situation that $(\L^+,\Delta^+,S)$ is a proper locality with $\Delta\subseteq\Delta^+$ and $\L=\L^+|_\Delta$. Set $\N^+=\Phi_{\L^+,\L}^{-1}(\N)$. By \cite[Theorem~C(c)]{Henke:2020}, $\N$ has $p$-power index in $\L$ if and only if $\N^+$ has $p$-power index in $\L^+$. Moreover, Theorem~\ref{mainStepIV}(c) yields $\Psi_{\L^+}(\N^+)=\Psi_\L(\N)$. Hence 
\begin{equation}\label{E:OpLb}
\mbox{(b) holds if and only if (b) holds with $(\L^+,\N^+)$ in place of $(\L,\N)$.}
\end{equation}
Similarly, if $*$ is one of the symbols ``$p$'' or ``$p^\prime$'', then we have $O^*_\L(\N)=\Phi_{\L^+,\L}(O^*_{\L^+}(\N^+))$ by \cite[Lemma~7.5]{Henke:Regular}. Using again Theorem~\ref{mainStepIV}(c), this implies  
\[\Psi_{\L^+}(O^*_{\L^+}(\N^+))=\Psi_\L(\Phi_{\L^+,\L}(O^*_{\L^+}(\N^+)))=\Psi_\L(O^*_\L(\N)).\]
Thus, 
\begin{equation}\label{E:OpLc}
\mbox{ part (c) is true for $\L$ and $\N$ if and only if it is true for $\L^+$ and $\N^+$.} 
\end{equation}

\smallskip

By Theorem~\ref{T:VaryObjects}(d), there exists a subcentric locality $(\L^s,\F^s,S)$ over $\F$ with $\L^s|_\Delta=\L$.  So for the proof of (b), \eqref{E:OpLb} allows us to assume without loss of generality that  $\Delta=\F^s$. Then $\E=\F_T(\N)$ by Theorem~\ref{mainStepIV}(b). So by Corollary~\ref{C:ERNR0}, $\F_S(\N S)=\E S$ and $(\N S,\Delta,S)$ is a proper locality over $\E S$. If $\N$ has $p$-power index in $\L$, then $\L=\N S$ and thus $\F=\F_S(\N S)=\E S$. So $\E$ has in this case $p$-power index by Lemma~\ref{L:pPowerIndexProducts}. By the same lemma, if $\E$ has $p$-power index in $\F$, then $\E S=\F$. So $(\N S,\Delta,S)$ is in this case a proper locality over $\F$ and thus isomorphic to $(\L,\Delta,S)$ by \cite[Theorem~A(a)]{Henke:2015}. Hence, as $\N S\subseteq\L$, we have then $\L=\N S$ and $\N$ has $p$-power index in $\L$. This shows (b).

\smallskip

For the proof of (c) we continue to work with the subcentric locality $(\L^s,\F^s,S)$ as above. Note that the restriction $(\L^s|_{\delta(\F)},\delta(\F),S)$ is well-defined and a regular locality over $\F$ since $\delta(\F)$ is $\F$-closed by Lemma~\ref{L:Regular10.4}. Applying  \eqref{E:OpLc} twice, for the proof of (c) we may replace $(\L,\Delta,S)$ by $(\L^s|_{\delta(\F)},\delta(\F),S)$ and assume that $(\L,\Delta,S)$ is a regular locality over $\F$. Then $\E=\F_T(\N)$ by Theorem~\ref{mainStepIV}(b) and $(\N,\delta(\E),S\cap\N)$ is a regular locality over $\E$ by Theorem~\ref{T:mainRegularPartialNormal}(a). By \cite[Corollary~7.11]{ChermakIII} or \cite[Lemma~10.18]{Henke:Regular}, we have $O^*_\L(\N)=O^*(\N)$. Setting $\K:=O^*(\N)$, Theorem~\ref{mainStepIV}(b) gives moreover $\Psi_\L(\K)=\F_{S\cap \K}(\K)=\Psi_\N(\K)$. Thus, it is sufficient to show  $\Psi_\L(O^*(\L))=O^*(\F)$.

\smallskip

We use now the notation $\mathbb{D}^p$ and $\mathbb{D}^{p^\prime}$ from above as well as the notation from Definition~\ref{D:OpLOpprimeL}. If $*$ denotes the symbol ``$p$'', write $\mathbb{D}^*$ for $\mathbb{D}^p$ and set $\mathbb{K}^*:=\mathbb{K}_\L$. If $*$ denotes ``$p^\prime$'', write $\mathbb{D}^*$ for $\mathbb{D}^{p^\prime}$ and set $\mathbb{K}^*:=\mathbb{K}_\L'$. By Lemma~\ref{L:OupperpFBigCap} and the definition of $O^{p^\prime}(\F)$ above, we have 
\[O^*(\F)=\bigwedge_{\m{D}\in\mathbb{D}^*}\m{D}.\]
Similarly, $O^*(\L)=\bigcap_{\K\in\mathbb{K}^*}\K$ by definition. Properties (a) and (b) together with the fact that $\Psi_\L$ is a bijection give moreover that $\mathbb{D}^*=\{\Psi_\L(\K)\colon \K\in\mathbb{K}^*\}$. Hence, by Theorem~\ref{T:IntersectionsI}(c), we have
\[\Psi_\L(O^*(\L))=\Psi_\L(\bigcap_{\K\in\mathbb{K}^*}\K)=\bigwedge_{\K\in\mathbb{K}^*}\Psi_\L(\K)=\bigwedge_{\m{D}\in\mathbb{D}^*}\m{D}=O^*(\F).\]
This proves (c).
\end{proof}

\begin{corollary}\label{C:Oupperp}
Suppose $(\L,\Delta,S)$ is a proper locality over $\F$. Then $\L=O^p(\L)$ if and only if $\F=O^p(\F)$. Similarly $\L=O^{p^\prime}(\L)$ if and only if $\F=O^{p^\prime}(\F)$.
\end{corollary}

\begin{proof}
As $\F_S(\L)=\F$, Theorem~\ref{mainStepIV}(a) implies that $\Psi_\L(\L)=\F$. Proposition~\ref{Oupperp}(c) gives moreover that   $\Psi_\L(O^p(\L))=O^p(\F)$ and $\Psi_\L(O^{p^\prime}(\L))=O^{p^\prime}(\F)$. As $\Psi_\L$ is a bijection, the assertion follows.
\end{proof}

\subsection{Simple and quasisimple localities and fusion systems}

Recall the definition of simple and quasisimple proper localities from Definition~\ref{D:SimpleQuasisimple}. The fusion system $\F$ is called \emph{simple} if $\F$ has precisely two normal subsystems, and $\F$ is called \emph{quasisimple} if $\F=O^p(\F)$ and $\F/Z(\F)$ is simple.

\begin{prop}\label{SimpleQuasisimpleTranslate}
If $(\L,\Delta,S)$ is a proper locality over $\F$, then the following hold:
\begin{itemize}
 \item [(a)] $\L$ is simple if and only if $\F$ is simple.
 \item [(b)] $Z(\L)=Z(\F)$.
 \item [(c)] $\L$ is quasisimple if and only if $\F$ is quasisimple.
\end{itemize}
\end{prop}

\begin{proof}
\textbf{(a)} By Proposition~\ref{mainStepIV}, there exists a bijection $\Psi_\L\colon\fN(\L)\rightarrow\fN(\F)$, so (a) holds. 

\smallskip

\textbf{(b)} As $(\L,\Delta,S)$ is a proper locality, we have that $Z(\L)\subseteq C_\L(S)\leq S$. So \cite[Proposition~5]{Henke:2015} yields $Z(\L)=Z(\F)$, i.e. (b) holds. 

\smallskip

\textbf{(c)} For any $Z\leq Z(\L)$, it is true by \cite[Proposition~9.3(a),(d)]{Henke:2015} that $(\L/Z,\Delta/Z,S/Z)$ is a proper locality over $\F/Z$, where $\Delta/Z:=\{PZ/Z\colon P\in\Delta\}$. Set now $Z:=Z(\L)=Z(\F)$. By (a), $\L/Z=\L/Z(\L)$ is simple if and only if $\F/Z=\F/Z(\F)$ is simple. By Corollary~\ref{C:Oupperp}, $\F=O^p(\F)$ if and only if $\L=O^p(\L)$. Hence, (c) holds.
\end{proof}

\subsection{Commuting partial normal subgroups and subsystems centralizing each other}\label{SS:CentralizingCommuting} In this subsection we will relate the concepts introduced in Subsections~\ref{SS:CentralProdFusion} and \ref{SS:CFE} to the concepts introduced in Subsection~\ref{SS:Nperp}.

\begin{prop}\label{P:NperpCFE}
Suppose $(\L,\Delta,S)$ is a proper locality over $\F$ and $\N\unlhd \L$. Set $\E:=\Psi_\L(\N)$. Then 
\[\Psi_\L(\N^\perp)=C_\F(\E).\]
In particular, if $(\L,\Delta,S)$ is a regular locality, then $\F_{C_S(\N)}(C_\L(\N))=C_\F(\E)$.
\end{prop}

\begin{proof}
If $(\L,\Delta,S)$ is a regular locality, then $\N^\perp=C_\L(\N)$ by Theorem~\ref{T:mainRegularPartialNormal}(b). Moreover, by Theorem~\ref{mainStepIV}(b), the map $\Psi_\L$ is in this case given by $\Psi_\L(\M)=\F_{S\cap\M}(\M)$ for all $\M\unlhd\L$. Hence, it is sufficient to show $\Psi_\L(\N^\perp)=C_\F(\E)$ if $(\L,\Delta,S)$ is arbitrary. We show this in two steps.

\smallskip

\emph{Step~1:} We show $\Psi_\L(\N^\perp)=C_\F(\E)$ in the case $\Delta=\F^s$. For the proof assume $\Delta=\F^s$ and set $T:=\N\cap S$. Then by Lemma~\ref{L:NFTSubcentricLocality}, $(\N_\L(T),N_\F(T)^s,S)$ is a subcentric locality over $N_\F(T)$ and $C_\L(T)\unlhd\L$ with $C_\F(T)=\F_{C_S(T)}(C_\L(T))$. In particular, 
\[\M:=O^p_{N_\L(T)}(C_\L(T))\]
is defined. By \cite[Corollary~9.9]{Henke:Regular}, we have $\N^\perp=C_S(\N)\M$ and $\N^\perp\cap S=C_S(\N)$. By Theorem~\ref{mainStepIV}(b), $\E=\Psi_\L(\N)=\F_{S\cap\N}(\N)\unlhd\F$. Hence, it follows from \cite[Proposition~4]{Henke:2015} that $\N^\perp\cap S=C_S(\N)=C_S(\E)$. 

\smallskip

Using again Theorem~\ref{mainStepIV}(b) (applied both to $(\L,\Delta,S)$ and to $(N_\L(T),N_\F(T)^s,S)$), we have $\Psi_\L(\N^\perp)=\F_{S\cap\N^\perp}(\N^\perp)=\Psi_{N_\L(T)}(\N^\perp)$. Moreover, $\Psi_{N_\L(T)}(C_\L(T))=\F_{C_S(T)}(C_\L(T))=C_\F(T)$ and so, by Proposition~\ref{Oupperp}(c), $\Psi_{N_\L(T)}(\M)=O^p(\Psi_{N_\L(T)}(C_\L(T)))=O^p(C_\F(T))$. As $\M\subseteq\N^\perp\subseteq C_\L(T)$, it follows now from Theorem~\ref{mainStepIV}(e) that
\[O^p(C_\F(T))=\Psi_{N_\L(T)}(\M)\unlhd \Psi_{N_\L(T)}(\N^\perp)\unlhd \Psi_{N_\L(T)}(C_\L(T))=C_\F(T).\]
Hence, by Lemma~\ref{L:pPowerIndexProducts}, $\Psi_\L(\N^\perp)=\Psi_{N_\L(T)}(\N^\perp)$ is a normal subsystem of $C_\F(T)$ of $p$-power index. Thus, by \cite[Theorem~I.7.4]{Aschbacher/Kessar/Oliver:2011}, $\Psi_\L(\N^\perp)$ is the unique saturated subsystem of $C_\F(T)$ over $\N^\perp\cap S=C_S(\E)$, which is of $p$-power index in $C_\F(T)$ (cf. Subsection~\ref{SS:Fusionppowerindex}). So $\Psi_\L(\N^\perp)=C_\F(\E)$ by definition of $C_\F(\E)$ (cf. Subsection~\ref{SS:CFE}). 

\smallskip

\emph{Step~2:} We show $\Psi_\L(\N^\perp)=C_\F(\E)$ if $(\L,\Delta,S)$ is arbitrary. By Theorem~\ref{T:VaryObjects}(d), there exists a subcentric locality $(\L^s,\F^s,S)$ over $\F$ with $\L^s|_\Delta=\L$ and we can consider the map $\Phi_{\L^s,\L}\colon \fN(\L^s)\rightarrow \fN(\L),\M^s\mapsto \M^s\cap\L$ which is a bijection by Theorem~\ref{T:VaryObjects}(b). Setting $\N^s:=\Phi_{\L^s,\L}^{-1}(\N)$, it follows from Step~1 that $\Psi_{\L^s}((\N^s)^\perp)=C_\F(\E)$. Furthermore, by \cite[Lemma~9.11]{Henke:Regular}, we have $\Phi_{\L^s,\L}((\N^s)^\perp)=\N^\perp$. Hence, using Theorem~\ref{mainStepIV}(c) we can conclude that $\E=\Psi_\L(\N)=\Psi_\L(\Phi_{\L^s,\L}(\N^s))=\Psi_{\L^s}(\N^s)$ and  $\Psi_\L(\N^\perp)=\Psi_\L(\Phi_{\L^s,\L}((\N^s)^\perp))=\Psi_{\L^s}((\N^s)^\perp)=C_\F(\E)$. This completes the proof.
\end{proof}

\begin{corollary}\label{C:CentricPartialNormal}
Let $(\L,\Delta,S)$ be a regular locality over $\F$ and $\N\unlhd\L$. Set $T:=S\cap\N$ and $\E:=\Psi_\L(\N)$ so that $\E$ is a normal subsystem of $\F$ over $T$. Then the following are equivalent:
\begin{itemize}
 \item [(i)] $C_\F(\E)\subseteq\E$;
 \item [(ii)] $C_S(\E)\leq T$;
 \item [(iii)] $\N^\perp\subseteq\N$;
 \item [(iv)] $\N^\perp\cap S\leq T$;
 \item [(v)] $C_\F(\E)=\F_{Z(\E)}(Z(\E))$.
\end{itemize}
\end{corollary}

\begin{proof}
Clearly (i) implies (ii) and (v) implies (i). By Proposition~\ref{P:NperpCFE}, $\Psi_\L(\N^\perp)=C_\F(\E)$. In particular, $\N^\perp\cap S=C_S(\E)$. Hence, (ii) and (iv) are equivalent. 

\smallskip

By \cite[Lemma~9.21, Corollary~10.10]{Henke:Regular}, (iii) and (iv) are equivalent as well; moreover, these properties are equivalent to $\N^\perp=Z(\N^s)$, where $(\L^s,\F^s,S)$ is a subcentric locality over $\F$ with $\L^s|_\Delta=\L$ and $\N^s\unlhd\L^s$ such that $\N^s\cap \L=\N$. Assume the latter property holds. It remains to show that (v) follows. By Theorem~\ref{mainStepIV}(b),(c) we have that $\E=\Psi_\L(\N)=\Psi_\L(\Phi_{\L^s,\L}(\N^s))=\Psi_{\L^s}(\N^s)=\F_{S\cap \N^s}(\N^s)$. By \cite[Lemma~9.10]{Henke:Regular}, $Z(\N^s)\leq S$. It is shown in \cite[Proposition~4]{Henke:2015} that $C_S(\E)=C_S(\N^s)$. Hence, we can conclude that $Z(\N^s)\leq Z(\E)\leq C_S(\E)\cap T=C_S(\N^s)\cap T\leq Z(\N^s)$ and thus $\N^\perp=Z(\N^s)=Z(\E)$. It follows  $C_\F(\E)=\Psi_\L(\N^\perp)=\Psi_\L(Z(\E))=\F_{Z(\E)}(Z(\E))$ by Theorem~\ref{mainStepIV}(f). 
\end{proof}

Recall Definition~\ref{D:Commute}. We say that partial normal subgroups $\N_1,\dots,\N_k$ of $\L$ \emph{commute pairwise} if $\N_i$ commutes with $\N_j$ for all $i,j\in\{1,2,\dots,k\}$ with $i\neq j$.

\begin{prop}\label{P:N1N2Perpendicular}
Suppose $(\L,\Delta,S)$ is a proper locality over $\F$ and $k\in\mathbb{N}$ with $k\geq 1$. For $i=1,2,\dots,k$ let  $\N_i\unlhd\L$ and set $\E_i=\Psi_\L(\N_i)$. Then $\N_1,\N_2,\dots,\N_k$ commute pairwise if and only if $\E_1,\E_2,\dots,\E_k$ centralize each other in $\F$. Moreover, if so, then $\Psi_\L(\N_1\N_2\cdots \N_k)=\E_1*\E_2*\cdots *\E_k$.
\end{prop}

\begin{proof}
For $i=1,\dots,k$ set $S_i=\N_i\cap S$ and note that $\E_i$ is a subsystem of $\F$ over $S_i$. Assume first that $\N_1,\N_2,\dots,\N_k$ commute pairwise. Then $\N_j\subseteq \N_i^\perp$ for $i\neq j$ and hence $\M:=\prod_{j\neq i}\N_j\subseteq \N_i^\perp$. Thus, $\N_i\subseteq\M^\perp$ by \cite[Corollary~5.13]{Henke:Regular} and hence $\E_i=\Psi_\L(\N_i)\subseteq\Psi_\L(\M^\perp)\subseteq C_\F(\Psi_\L(\M))\subseteq C_\F(\M\cap S)$ by Proposition~\ref{P:NperpCFE}. By Theorem~\ref{T:ProductsPartialNormal}(a), $\M\cap S=\prod_{j\neq i}S_j$. Hence, it follows from Lemma~\ref{L:CentralizeEachOtherSaturated} that $\E_1,\dots,\E_k$ centralize each other. 

\smallskip

Assume now that $\E_1,\E_2,\dots,\E_k$ centralize each other. Fix $i,j\in\{1,2,\dots,k\}$ with $i\neq j$. Note that $\E_i$ and $\E_j$ centralize each other. So by Proposition~\ref{P:DEcentralize}, we have $\E_i\unlhd C_\F(\E_j)$. Thus, it follows from Theorem~\ref{mainStepIV}(d) and Proposition~\ref{P:NperpCFE} that $\N_i=\Psi_\L^{-1}(\E_i)\subseteq \Psi_\L^{-1}(C_\F(\E_j))=\N_j^\perp$. Hence, $\N_i$ commutes with $\N_j$ by definition of $\N_j^\perp$. 

\smallskip

We have shown now that $\E_1,\E_2,\dots,\E_k$ centralize each other in $\F$ if and only if $\N_1,\N_2,\dots,\N_k$ commute pairwise. Assume now that these two equivalent conditions hold. By Lemma~\ref{L:CharacterizeF1starF2Real}, $\E_1*\E_2*\cdots *\E_k\unlhd\F$ and $\E_i\unlhd\E_1*\E_2*\cdots *\E_k$ for $i=1,2,\dots,k$. Hence, by Theorem~\ref{mainStepIV}(d), we have $\N_i=\Psi_\L^{-1}(\E_i)\subseteq \N:=\Psi_\L^{-1}(\E_1*\E_2*\cdots *\E_k)$ for $i=1,2,\dots,k$ and therefore $\N_1\N_2\cdots\N_k\subseteq \N$. As $\Psi_\L$ is inclusion-preserving, it follows  $\E:=\Psi_\L(\N_1\N_2\cdots\N_k)\subseteq\Psi_\L(\N)=\E_1*\E_2*\cdots *\E_k$. Since $\N_i\subseteq\N_1\N_2\cdots\N_k$, Theorem~\ref{mainStepIV}(e) gives $\E_i=\Psi_\L(\N_i)\unlhd \E$ for $i=1,2,\dots,k$. Hence, by Lemma~\ref{L:CharacterizeF1starF2}(c), we have $\E_1*\E_2*\cdots *\E_k\subseteq \E$. This proves $\E_1*\E_2*\cdots *\E_k=\E=\Psi_\L(\N_1\N_2\cdots\N_k)$ as required.
\end{proof}

\subsection{Characteristic subsystems}

Following Aschbacher \cite[p.~40]{Aschbacher:2011} we define a subsystem $\E$ of $\F$ over $T\leq S$ to be \emph{characteristic} in $\F$ if $\E\unlhd\F$ and $\E$ is $\Aut(\F)$-invariant. The latter property means that, for every $\alpha\in\Aut(\F)$, $T\alpha=T$ and $\alpha|_T$ induces an automorphism of $\E$. 

\begin{prop}\label{P:CharacteristicSubsystems}
 Let $(\L,\Delta,S)$ be a proper locality over $\F$, let $\N\unlhd\L$ and set $\E:=\Psi_\L(\N)$. Then the following hold:
\begin{itemize}
 \item [(a)] If $\beta\in\Aut(\L,S)$ and $\alpha=\beta|_S$, then $\N\beta=\N$ if and only if $\E^\alpha=\E$.
 \item [(b)] Suppose $\Delta\in\{\F^c,\F^q,\F^s,\delta(\F)\}$ or assume more generally that $\Delta$ is $\Aut(\F)$-invariant. Then $\N$ is $\Aut(\L,S)$-invariant if and only if $\E$ is characteristic in $\F$. 
\end{itemize}
\end{prop}

\begin{proof}
\textbf{(a)} Suppose first that $\alpha$ and $\beta$ are as in (a). It is a restatement of Theorem~\ref{mainStepIV}(g) that $\E^\alpha=\Psi_\L(\N\beta)$. As $\Psi_\L$ is a bijection, this yields (a).

\smallskip

\textbf{(b)} The sets $\F^c$, $\F^q$, $\F^s$ and $\delta(\F)$ are $\Aut(\F)$-invariant; this is elementary to check for $\F^c$, $\F^q$ and $\F^s$ (cf. \cite[Lemma~3.6]{Henke:2015}) and for $\delta(\F)$ this is shown in \cite[Lemma~11.19]{Henke:Regular}. Thus, for the proof of (b) we may just assume that $\Delta$ is $\Aut(\F)$-invariant. One observes easily that $\beta|_S\in\Aut(\F)$ for every $\beta\in\Aut(\L,S)$; this follows also from the more general statement \cite[Lemma~2.21(b)]{Henke:2020}. On the other hand, as $\Delta$ is $\Aut(\F)$-invariant, it follows from \cite[Proposition~3.19]{Henke:2016} (which uses essentially the uniqueness statement in Theorem~\ref{T:ProperLocalityExistence}) that every element of $\Aut(\F)$ extends to an element of $\Aut(\L,S)$. Thus, (b) follows from (a).
\end{proof}

\section{Subnormal subsystems, partial subnormal subgroups and components} 

\textbf{Throughout this section let $\F$ be a saturated fusion system over $S$.}

\smallskip

In this section, we will be particularly interested in \emph{regular} localities over $\F$. The reason is that every partial subnormal subgroup of a regular locality is a regular locality and there are components, the layer and the generalized Fitting subgroup  defined; see \cite{ChermakIII, Henke:Regular} and Subsection~\ref{SS:Regular}. We will show Theorem~\ref{mainSubnormalGeneralizedFitting} and complete the proof of Theorem~\ref{T:mainDetails}(d) in Subsection~\ref{SS:TheoremEProof}. Moreover, Theorem~\ref{T:IntersectionsSubnormal} is proved in Subsection~\ref{SS:IntersectionsSubnormal}. As mentioned in the introduction, our work serves also the dual purpose of revisiting Aschbacher's theory. In particular, we take great care not to use his results, but to reprove them. Essentially this is done in Subsection~\ref{SS:EFNewProofs} and partly also in Subsection~\ref{SS:IntersectionsSubnormal}.  Once the necessary results on fusion systems are reproved we will also complete the proof of Theorem~\ref{main} in Subsection~\ref{SS:TheoremAProof}. At the end we use the results from Subsection~\ref{SS:TheoremEProof} to show two lemmas that are needed  in the next section.

\subsection{Important one-to-one correspondences} \label{SS:TheoremEProof}

\begin{prop}\label{P:psihat}
Let $(\L,\Delta,S)$ be a regular locality over $\F$ and consider the map
\[\hat{\Psi}_\L\colon \{\H\colon \H\subn\L\}\rightarrow\{\E\colon \E\subn\F\},\H\mapsto\F_{S\cap\H}(\H).\] 
Then the following hold:
\begin{itemize}
\item [(a)] $\hat{\Psi}_\L$ is well-defined and bijective. Moreover, $\hat{\Psi}_\L$ restricts to the map $\Psi_\L$ from Theorem~\ref{mainStepIV}.
\item [(b)] If $\E_1$ and $\E_2$ are subnormal subsystems of $\F$ with $\E_1\subn\E_2$, then $\hat{\Psi}_\L^{-1}(\E_1)\subn\hat{\Psi}_\L^{-1}(\E_2)$.
\item [(c)] Suppose $\H_1$ and $\H_2$ are partial subnormal subgroups of $\L$ with $\H_1\subseteq\H_2$. Then $\hat{\Psi}_\L(\H_1)\subn\hat{\Psi}_\L(\H_2)$. If $\H_1\unlhd\H_2$, then $\hat{\Psi}_\L(\H_1)\unlhd\hat{\Psi}_\L(\H_2)$.
\end{itemize} 
\end{prop}

\begin{proof}
We will use throughout that, by Theorem~\ref{mainStepIV}(a), the map $\Psi_\L\colon\fN(\L)\rightarrow\fN(\F),\N\mapsto \F_{S\cap\N}(\N)$ is well-defined and a bijection. Moreover, we will use that every partial subnormal subgroup $\H$ of $\L$ is a regular locality over $\F_{\H\cap S}(\H)$ by Corollary~\ref{C:RegularSubnormal}. 

\textbf{(a)} Using induction on $|\L|$ one verifies easily that $\hat{\Psi}_\L$ is well-defined and surjective. Clearly $\hat{\Psi}_\L$ restricts to $\Psi_\L$.

\smallskip

Suppose now that $\H_1$ and $\H_2$ are partial subnormal subgroups of $\L$ with $\F_{S\cap \H_1}(\H_1)=\F_{S\cap \H_2}(\H_2)$. Then in particular, $T:=\H_1\cap S=\H_2\cap S$. We want to show that $\H_1=\H_2$. Assume that this is false and $\L$ is a minimal counterexample. As $\H_1\neq\H_2$, there exists $i\in\{1,2\}$ such that $\H_i\neq \L$. Suppose without loss of generality that $\H_1\neq \L$. Then there exists $\tilde{\H}\unlhd \L$ with $\H_1\subn\tilde{\H}\neq \L$. Note that $T\subseteq \H_1\cap\H_2\subseteq \tilde{\H}\cap\H_2\unlhd \H_2$ and thus $O^{p^\prime}(\H_2)\subseteq \H_2\cap\tilde{\H}\subseteq\tilde{\H}$. Notice that $O^{p^\prime}(\H_2)$ is subnormal in $\L$ and thus, by Lemma~\ref{L:SubnormalBasic}(c), also in $\tilde{\H}$. Clearly $O^{p^\prime}(\H_1)$ is also subnormal in $\tilde{\H}$. As $\H_1$ and $\H_2$ are regular localities over $\E$, it follows from Proposition~\ref{Oupperp}(c) that $\F_T(O^{p^\prime}(\H_1))=O^{p^\prime}(\E)=\F_T(O^{p^\prime}(\H_2))$. As $|\tilde{\H}|<|\L|$ and $\L$ was assumed to be a minimal counterexample, it follows that 
\begin{equation*}
 O^{p^\prime}(\H_1)=O^{p^\prime}(\H_2).
\end{equation*}
By \cite[Lemma~11.14]{Henke:Regular}, $\Comp(\H_i)=\Comp(O^{p^\prime}(\H_i))$ for $i=1,2$. Hence  
\[\Comp(\H_1)=\Comp(\H_2)\mbox{ and }E(\H_1)=E(\H_2).\]
Set 
\[T_0:=E(\H_1)\cap T=E(\H_2)\cap T.\]
By the Frattini Lemma  \cite[Corollary~3.11]{Chermak:2015}, we have $\H_i=E(\H_i)N_{\H_i}(T_0)$ for $i=1,2$. Thus it remains to show that $N_{\H_1}(T_0)=N_{\H_2}(T_0)$. 

\smallskip

It follows from Lemma~\ref{L:SubnormalBasic}(a) that $\Comp(\H_1)\subseteq\Comp(\L)$. Write $\M$ for the product of the components of $\L$ which are not in $\Comp(\H_1)=\Comp(\H_2)$. Recall that $E(\L)$ can be described as the product of components of $\L$; moreover every component and every product of components is normal in $F^*(\L)$ (cf. Subsection~\ref{SS:Regular} and \cite[Proposition~11.7, Definition~11.8, Theorem~11.18(a)]{Henke:Regular}). Hence, using Theorem~\ref{T:ProductsPartialNormal}, one sees that $E(\L)=E(\H_1)\M$ and 
\[S_0:=E(\L)\cap S=T_0(\M\cap S).\]
By \cite[Lemma~4.5, Theorem~11.18(c)]{Henke:Regular}, the product $\H_i\M$ is a central product of $\H_i$ and $\M$ for $i=1,2$. In particular, by \cite[Lemma~4.8]{Henke:Regular}, we have $\H_i\subseteq C_\L(\M)$ and so  $N_{\H_i}(T_0)\subseteq N_{\H_i}(S_0)$. As $E(\H_i)\unlhd\H_i$ it follows $N_{\H_i}(T_0)=N_{\H_i}(S_0)$ for $i=1,2$.

\smallskip

By \cite[Corollary~11.10]{Henke:Regular}, $S_0\in\delta(\F)=\Delta$, i.e. 
\[G:=N_\L(S_0)\] 
is a subgroup of $\L$ which is a group of characteristic $p$. As $\H_i\subn\L$, it follows from Lemma~\ref{L:SubnormalBasic}(c) that $N_{\H_i}(T_0)=N_{\H_i}(S_0)$ is a subnormal subgroup of $G$ for $i=1,2$. Since $(\H_i,\delta(\E),T)$ is a regular locality over $\E$ for $i=1,2$ and $T_0\in\delta(\E)$ by \cite[Corollary~11.10]{Henke:Regular}, it follows from Lemma~\ref{L:LocalitiesPropg}(c) that $\F_T(N_{\H_1}(T_0))=N_\E(T_0)=\F_T(N_{\H_2}(T_0))$. So Lemma~\ref{ModelLemma} yields $N_{\H_1}(T_0)=N_{\H_2}(T_0)$. As argued above this implies $\H_1=\H_2$. This completes the proof of (a).

\smallskip

\textbf{(b)} Assume now that $\E_1$ and $\E_2$ are subnormal subsystems of $\F$ with $\E_1\subn\E_2$. Then
\[\H_2:=\hat{\Psi}_\L^{-1}(\E_2)\subn\L\]
and $\F_{S\cap\H_2}(\H_2)=\E_2$, i.e. $\H_2$ is a regular locality over $\E_2$ by Corollary~\ref{C:RegularSubnormal}. Hence,
\[\hat{\Psi}_{\H_2}\colon\{\H\colon \H\subn\H_2\}\longrightarrow\{\E\colon\E\subn\E_2\},\H\mapsto\F_{S\cap\H}(\H)\]
is a well-defined bijection. In particular, $\H_1:=\hat{\Psi}_{\H_2}^{-1}(\E_1)\subn\H_2\subn\L$ and thus $\H_1\subn\L$. Moreover, $\hat{\Psi}_\L(\H_1)=\F_{S\cap\H_1}(\H_1)=\hat{\Psi}_{\H_2}(\H_1)=\E_1$. Hence,
\[\hat{\Psi}_\L^{-1}(\E_1)=\H_1\subn\H_2=\hat{\Psi}_\L^{-1}(\E_2).\]
This proves (b).

\smallskip

\textbf{(c)} Let $\H_1$ and $\H_2$ be arbitrary partial subnormal subgroups of $\L$ with $\H_1\subseteq\H_2$. Then again, by Corollary~\ref{C:RegularSubnormal}, $\H_2$ is a regular locality over $\E_2:=\hat{\Psi}_\L(\H_2)$. Moreover, $\H_1\subn\H_2$ by Lemma~\ref{L:SubnormalBasic}(c). Thus, $\hat{\Psi}_{\H_2}$ is defined and $\hat{\Psi}_\L(\H_1)=\F_{S\cap\H_1}(\H_1)=\hat{\Psi}_{\H_2}(\H_1)\subn\E_2=\hat{\Psi}_\L(\H_2)$. If $\H_1\unlhd\H_2$, then by (a), $\hat{\Psi}_\L(\H_1)=\hat{\Psi}_{\H_2}(\H_1)=\Psi_{\H_2}(\H_1)\unlhd\E_2$. This proves (c). 
\end{proof}

The next definition follows Aschbacher \cite[p.3]{Aschbacher:2011}.

\begin{definition}\label{D:ComponentsEFFstarF}
\begin{itemize}
\item A \emph{component} of $\F$ is a subnormal subsystem which is quasisimple. By $\Comp(\F)$ we denote the set of components of $\F$. 
\item Let $\E_1,\E_2,\dots,\E_n$ be the normal subsystems of $\F$ containing every component of $\F$. Building on Theorem~\ref{T:IntersectionsI}, define $E(\F)=\bigwedge_{i=1}^n\E_i$ to be the largest normal subsystem of $\F$ which is normal in $\E_i$ for all $i=1,\dots,n$.
\item Set $F^*(\F):=E(\F)O_p(\F)$ (with the product defined as in Definition~\ref{D:Product}).    
\end{itemize}
\end{definition}

\begin{lemma}\label{GeneralizedFitting}
Let $(\L,\Delta,S)$ be a regular locality over $\F$ and let $\hat{\Psi}_\L$ be the map from Proposition~\ref{P:psihat}. Then the following hold.
\begin{itemize}
 \item[(a)] The map $\hat{\Psi}_\L$ from (a) induces a bijection from the set of components of $\L$ to the set of components of $\F$.
 \item[(b)] $\Psi_\L(E(\L))=\F_{S\cap E(\L)}(E(\L))=E(\F)$.
 \item[(c)] $E(\F)\unlhd C_\F(O_p(\F))$ and 
\[\Psi_\L(F^*(\L))=\F_{S\cap F^*(\L)}(F^*(\L))=F^*(\F)=E(\F)*\F_{O_p(\F)}(O_p(\F)).\]
\end{itemize}
\end{lemma}

\begin{proof}
We will use throughout that, by Proposition~\ref{P:psihat}, $\hat{\Psi}_\L$ is bijective  and restricts to the bijection $\Psi_\L\colon \fN(\L)\rightarrow\fN(\F)$ from Theorem~\ref{mainStepIV}.

\smallskip

\textbf{(a)} Part (a) follows from $\hat{\Psi}_\L$ being bijective and Proposition~\ref{SimpleQuasisimpleTranslate}(c) in combination with Corollary~\ref{C:RegularSubnormal}. 

\smallskip

\textbf{(b)} We will use the characterization of $E(\L)$ as the product of components of $\L$ throughout (cf. Subsection~\ref{SS:Regular}). 
As $E(\L)\unlhd \L$, we have $\E:=\Psi_\L(E(\L))=\F_{S\cap E(\L)}(E(\L))\unlhd \F$. Since every component of $\L$ is contained in $E(\L)$, it follows from (a) and the fact that $\hat{\Psi}_\L$ is clearly inclusion-preserving that $\E$ contains every component of $\F$. Hence, by definition of $E(\F)$, we have \[E(\F)\unlhd \E.\]
To prove the converse inclusion set $\M:=\Psi_\L^{-1}(E(\F))$ and let $\K$ be a component of $\L$. By (a), $\F_{S\cap\K}(\K)$ is a component of $\F$. Hence, it follows from the definition of $E(\F)$ and from Theorem~\ref{T:IntersectionsI}(a) that $E(\F)$ is a subsystem over a subgroup of $S$, which contains $\K\cap S$. As $E(\F)=\Psi_\L(\M)=\F_{S\cap\M}(\M)$, it follows that $\K\cap S\subseteq \M\cap S$ and so $\K\cap S\subseteq\K\cap\M\unlhd\K$. Since $\K$ is quasisimple, we have $\K=O^{p^\prime}(\K)$ by Lemma~\ref{L:QuasisimpleOpprime}. This implies $\K=\K\cap\M\subseteq \M$. As $\K$ was arbitrary, this shows $E(\L)\subseteq \M$ and thus $\E=\Psi_\L(E(\L))\subseteq \Psi_\L(\M)=E(\F)$ as $\Psi_\L$ is inclusion-preserving by Theorem~\ref{mainStepIV}. Therefore $\E=E(\F)$ and (b) holds.

\smallskip

\textbf{(c)} By \cite[Proposition~5]{Henke:2015} and Lemma~\ref{L:OpL}, $R:=O_p(\L)=O_p(\F)$, and by Theorem~\ref{mainStepIV}(b), $\Psi_\L(F^*(\L))=\F_{S\cap F^*(\L)}(F^*(\L))$. Moreover, \cite[Lemma~11.9]{Henke:Regular} states that  $F^*(\L)=E(\L)R$ and  \cite[Lemma~11.5]{Henke:Regular} gives $E(\L)\subseteq R^\perp=C_\L(R)$. The latter property implies by \cite[Lemma~3.5]{Henke:Regular} that $E(\L)$ commutes with $R$. By Theorem~\ref{mainStepIV}(f), we have $\Psi_\L(R)=\F_R(R)$. Hence, it follows from Proposition~\ref{P:N1N2Perpendicular} and from (b) that $E(\F)$ and $\F_R(R)$ centralize each other and  
\[\F_{S\cap F^*(\L)}(F^*(\L))=\Psi_\L(F^*(\L))=\Psi_\L(E(\L))*\Psi_\L(R)=E(\F)*\F_R(R).\]
In particular, $E(\F)=\Psi_\L(E(\L))\subseteq C_\F(R)$. So by Lemma~\ref{L:ERequalsEstarR}, $E(\F)*\F_R(R)=E(\F)R=F^*(\F)$. It follows moreover from  Proposition~\ref{P:DEcentralize} that $E(\F)\unlhd C_\F(\F_R(R))$. One observes easily that $\F_R(R)$ and $C_\F(R)$ centralize each other in $\F$, so by \cite[Theorem~2]{Henke:2018}, $C_\F(R)\subseteq C_\F(\F_R(R))$. By definition of $C_\F(\F_R(R))$, the converse inclusion holds as well. Hence $E(\F)\unlhd C_\F(\F_R(R))=C_\F(R)$. This proves (c).
\end{proof}

\begin{proof}[Proof of Theorem~\ref{mainSubnormalGeneralizedFitting}]
The first part is a reformulation of Proposition~\ref{P:psihat}. The statement about components is shown in Theorem~\ref{GeneralizedFitting}. 
\end{proof}

\begin{cor}\label{C:ELProper}
 Let $(\L,\Delta,S)$ be a proper locality over $\F$. Then 
\[\Psi_\L(E(\L))=E(\F)\mbox{ and }\Psi_\L(F^*(\L))=F^*(\F).\]  
\end{cor}

\begin{proof}
By Theorem~\ref{T:VaryObjects}(d) there exists a subcentric locality $(\L^s,\F^s,S)$ over $\F$ with $\L^s|_\Delta=\L$. Then $\L_\delta:=\L^s|_{\delta(\F)}$ is a regular locality. Thus, by Proposition~\ref{GeneralizedFitting}(b), 
\[\Psi_{\L_\delta}(E(\L_\delta))=E(\F)\mbox{ and }\Psi_{\L_\delta}(F^*(\L_\delta))=F^*(\F).\]
If $\Phi_{\L^s,\L}$ and $\Phi_{\L^s,\L_\delta}$ are defined as in Theorem~\ref{T:VaryObjects}(b), then it follows from \cite[Lemma~12.2]{Henke:Regular} that 
\[\Phi_{\L^s,\L_\delta}(E(\L^s))=E(\L_\delta)\mbox{ and }\Phi_{\L^s,\L}(E(\L^s))=E(\L).\]
Similarly, by \cite[Lemma~9.18(d)]{Henke:Regular},
\[\Phi_{\L^s,\L_\delta}(F^*(\L^s))=F^*(\L_\delta)\mbox{ and }\Phi_{\L^s,\L}(F^*(\L^s))=F^*(\L).\]
By Theorem~\ref{mainStepIV}(c), $\Psi_\L\circ\Phi_{\L^s,\L}=\Psi_{\L^s}=\Psi_{\L_\delta}\circ \Phi_{\L^s,\L_\delta}$. This implies
\[\Psi_\L(E(\L))=(\Psi_{\L_\delta}\circ \Phi_{\L^s,\L_\delta}\circ \Phi_{\L^s,\L}^{-1})(E(\L))=(\Psi_{\L_\delta}\circ \Phi_{\L^s,\L_\delta})(E(\L^s))=\Psi_{\L_\delta}(E(\L_\delta))=E(\F).\]
Similarly, it follows $\Psi_\L(F^*(\L))=\Psi_{\L_\delta}(F^*(\L_\delta))=F^*(\F)$. 
\end{proof}

\begin{proof}[Proof of Theorem~\ref{T:mainDetails}]
Parts (a) follows from Theorem~\ref{mainStepIV}(f) and \cite[Proposition~5]{Henke:2015}. Part (b) was proved in Proposition~\ref{Oupperp}(c), part (c) is Proposition~\ref{P:NperpCFE}, and part (d) is Corollary~\ref{C:ELProper}. 
\end{proof}

\subsection{Intersections of partial subnormal subgroups and of subnormal subsystems} \label{SS:IntersectionsSubnormal}

In this subsection and the next we demonstrate that many results about fusion systems follow easily from results about localities and the one-to-one correspondences shown so far. We start by revisiting the notion of a ``subnormal intersection'' of subnormal subsystems of $\F$. So in particular we prove Theorem~\ref{T:IntersectionsSubnormal}. Apart from the existence of a regular locality over $\F$, this needs only Proposition~\ref{P:psihat} and the elementary properties shown in Lemma~\ref{L:SubnormalBasic}. The arguments and results are very similar to the ones in Subsection~\ref{SS:IntersectionsNormal}. The existence of a ``subnormal intersection'' of subnormal subsystems of a saturated fusion system was first shown by Aschbacher \cite[7.2.2]{Aschbacher:2011}. However, we obtain a more precise characterization of this subsystem and we show how to translate between subnormal intersections in regular localities and in fusion systems.

\begin{prop}\label{P:IntersectionsSubnormal}
Let $(\L,\Delta,S)$ be a regular locality over $\F$ and consider the map $\hat{\Psi}_\L$ from Theorem~\ref{mainSubnormalGeneralizedFitting} or Proposition~\ref{P:psihat}. Let $I$ be an index set and suppose for every $i\in I$ we are given a partial subnormal subgroup $\H_i$ of $\L$. Set 
\[T_i:=\H_i\cap S\mbox{ and }\E_i:=\hat{\Psi}_\L(\H_i)=\F_{T_i}(\H_i) \mbox{ for all }i\in I\] 
and $\M:=\bigcap_{i\in I}\H_i$. Then the following hold:
\begin{itemize}
 \item [(a)] $\M$ is a partial subnormal subgroup of $\L$, and $\hat{\Psi}_\L(\M)$ is a subnormal subsystem of $\F$ over $\M\cap S=\bigcap_{i\in I}T_i$ contained in $\bigcap_{i\in I}\E_i$. 
 \item [(b)] $\hat{\Psi}_\L(\M)$ is the largest subsystem of $\F$ which is subnormal in $\E_i$ for all $i\in I$. Indeed, if $\m{D}\subn\F$ such that $\m{D}\subn\E_i$ for all $i\in I$, then $\m{D}\subn\hat{\Psi}_\L(\M)$.
\end{itemize}
\end{prop}

\begin{proof}
As $\L$ is finite, there are only finitely many partial subnormal subgroups of $\L$. Hence, we may assume that $I$ is finite. It follows then from Lemma~\ref{L:SubnormalBasic}(c) that $\M$ is subnormal in $\L$. So $\hat{\Psi}_\L(\M)$ is defined and  
\[\E:=\hat{\Psi}_\L(\M):=\F_{S\cap\M}(\M)\subn\F.\]
\indent \textbf{(a)} Observe that $\E=\F_{S\cap\M}(\M)$ is a fusion system over $S\cap\M=\bigcap_{i\in I}T_i$. As $\M\subseteq\H_i$ it follows from Proposition~\ref{P:psihat}(c) that $\E$ is subnormal in $\E_i$ for all $i\in I$. In particular, $\E\subseteq \bigcap_{i\in I}\E_i$. This proves (a). 

\smallskip

\textbf{(b)} Fix now a subsystem $\m{D}$ of $\F$ with $\m{D}\subn\E_i$ for all $i\in I$. Then in particular $\m{D}\subn\F$. For (b) it is sufficient to prove that $\m{D}\subn\E$. Observe that 
\[\K:=\hat{\Psi}_\L^{-1}(\m{D})\subn\L.\]
Moreover, by Proposition~\ref{P:psihat}(b), we have $\K\subn\hat{\Psi}_\L^{-1}(\E_i)=\H_i$ for all $i\in I$ and so $\K\subseteq\bigcap_{i\in I}\H_i=\M$. Hence, by Proposition~\ref{P:psihat}(c), $\m{D}=\hat{\Psi}_\L(\K)\subn\hat{\Psi}_\L(\M)=\E$. This proves the assertion.
\end{proof}

We are now in a position to prove Theorem~\ref{T:IntersectionsSubnormal}. Indeed we prove the following theorem. Notice that parts (a),(b),(e) generalize Theorem~\ref{T:IntersectionsSubnormal}. The other way around, as $I$ can be assumed to be finite, these parts could be obtained from Theorem~\ref{T:IntersectionsSubnormal} by induction on $|I|$. 

\begin{theorem}\label{T:IntersectionsSubnormal0}
Let $\F$ be a saturated fusion system over $S$. Then for every family $(\E_i)_{i\in I}$ of subnormal subsystems of $\F$ there exists a subsystem $\bigwedge_{i\in I}\E_i$ (denoted by $\E_1\wedge\E_2\wedge\cdots\wedge\E_k$ if $I=\{1,2,\dots,k\}$) such that the following hold, whenever $I$ is an index set and $\E_i$ is a subnormal subsystems of $\F$ over $T_i$ for all $i\in I$:
\begin{itemize}
 \item [(a)] $\bigwedge_{i\in I}\E_i$ is a subnormal subsystem of $\F$ over $\bigcap_{i\in I}T_i$ contained in $\bigcap\E_i$. Moreover, it is the largest  subsystem of $\F$ that is subnormal in $\E_i$ for all $i\in I$.
 \item [(b)] Every subsystem of $\F$ which is subnormal in $\E_i$ for all $i\in I$ is also subnormal in $\bigwedge_{i\in I}\E_i$.
 \item [(c)] If $I$ is the disjoint union of subsets $I_1,\dots,I_k$, then 
\[\bigwedge_{i\in I}\E_i=\left(\bigwedge_{i\in I_1}\E_i\right)\wedge\left(\bigwedge_{i\in I_2}\E_i\right)\wedge\cdots \wedge \left(\bigwedge_{i\in I_k}\E_i\right).\]  
 \item [(d)] If $j\in I$ such that $\E_j\unlhd\F$, then $\bigwedge_{i\in I}\E_i$ is a normal subsystem of $\bigwedge_{i\in I\backslash\{j\}}\E_i$.
 \item [(e)] Let $(\L,\Delta,S)$ be a regular locality over $\F$. If $\hat{\Psi}_\L$ is the map from Theorem~\ref{mainSubnormalGeneralizedFitting} or Proposition~\ref{P:psihat}, then for any collection $\H_i$ ($i\in I$) of partial subnormal subgroups of $\L$, we have $\hat{\Psi}_\L(\bigcap_{i\in I}\H_i)=\bigwedge_{i\in I}\hat{\Psi}_\L(\H_i)$.
\end{itemize}
\end{theorem}

\begin{proof}
Fix a regular locality $(\L,\Delta,S)$ over $\F$ (which exists by Lemma~\ref{L:Regular10.4}), and consider the map $\hat{\Psi}_\L$ from Theorem~\ref{mainSubnormalGeneralizedFitting} and Proposition~\ref{P:psihat}. If $\E_i$ ($i\in I$) is a collection of subnormal subsystems of $\F$, set
\[\bigwedge_{i\in I}\E_i:=\hat{\Psi}_\L(\bigcap_{i\in I}\hat{\Psi}_\L^{-1}(\E_i)).\]
\indent \textbf{(a,b,c,e)} It follows from Proposition~\ref{P:IntersectionsSubnormal} that (a) and (b) hold. In particular, the definition of $\bigwedge_{i\in I}\E_i$ does not depend on the choice of the regular locality $(\L,\Delta,S)$. This implies that (e) is true by the definition above. Since taking the intersection of sets is an  associative operation, part (c) follows as well.

\smallskip

\textbf{(d)} For the proof of (d), part (c) allows us to assume without loss of generality that $I=\{1,2\}$ and $j=1$, i.e. $\E_1\unlhd\F$. We need to show that $\E_1\wedge\E_2\unlhd\E_2$. Recall from Proposition~\ref{P:psihat}(a) that $\hat{\Psi}_\L$ restricts to the bijection  $\Psi_\L\colon\fN(\L)\rightarrow\fN(\F)$ from Theorem~\ref{mainStepIV}. So $\H_1:=\hat{\Psi}_\L^{-1}(\E_1)=\Psi_\L^{-1}(\E_1)$ is a partial normal subgroup of $\L$. Setting $\H_2:=\hat{\Psi}_\L^{-1}(\E_2)$, one observes therefore easily that $\H_1\cap\H_2\unlhd\H_2$. Hence, by the last part of Proposition~\ref{P:psihat}(c), it follows $\E_1\wedge\E_2=\hat{\psi}_\L(\H_1\cap\H_2)\unlhd\hat{\psi}_\L(\H_2)=\E_2$. This completes the proof.
\end{proof}

\subsection{More proofs of theorems about fusion systems}\label{SS:EFNewProofs}  In this subsection we reprove some results about fusion systems which are due to Aschbacher \cite{Aschbacher:2011} and mainly concern components, the layer and the generalized Fitting subsystem of a fusion system. Recall that $\F$ is assumed to be a saturated fusion system. 

\begin{theorem}[E-balance, Aschbacher {\cite[Theorem~7]{Aschbacher:2011}}]\label{T:Ebalance}
Let $X\leq S$ be fully $\F$-normalized. Then $E(N_\F(X))\subseteq E(\F)$.
\end{theorem}

\begin{proof}
By Theorem~\ref{T:ProperLocalityExistence}, there exists a subcentric locality $(\L,\Delta,S)$ over $\F$. We define $\bN_{\L}(X)=N_{\L}(X)|_{N_\F(X)^s}$ as in \cite[Subsection~3.9]{Henke:Regular} and remark that $(\bN_{\L}(X),N_\F(X)^s,N_S(X))$ is a subcentric locality over $N_\F(X)$ by \cite[Lemma~9.13]{Henke:2015}. By \cite[Theorem~5]{Henke:Regular}, we have $E(\bN_{\L}(X))\subseteq E(\L)$. Hence, applying Corollary~\ref{C:ELProper} (combined with Theorem~\ref{mainStepIV}(b)) 
both to $\F$ and to $N_\F(X)$, we can conclude that
\[E(N_\F(X))=\F_{N_S(X)\cap E(\bN_{\L}(X))}(E(\bN_{\L}(X)))\subseteq \F_{S\cap E(\L)}(E(\L))=E(\F).\]  
\end{proof}

The following lemma can be easily proved directly, but we give a proof using localities.

\begin{lemma}\label{L:CFRNFR}
For every $R\in\F^f$, the centralizer $C_\F(R)$ is a normal subsystem of $N_\F(R)$.
\end{lemma}

\begin{proof}
Recall that $N_\F(R)$ is saturated by \cite[Theorem~I.5.5]{Aschbacher/Kessar/Oliver:2011}. Hence, replacing $\F$ by $N_\F(R)$ we may assume $R\unlhd\F$. By Theorem~\ref{T:ProperLocalityExistence}, we can choose a subcentric locality $(\L,\Delta,S)$ over $\F$. Then $R=N_\L(R)$ by  \cite[Proposition~5]{Henke:2015}.  Moreover, by Lemma~\ref{L:NFTSubcentricLocality}, $C_\L(R)\unlhd\L$ and $\F_{C_S(R)}(C_\L(R))=C_\F(R)$. Hence, if $\Psi_\L$ is the map from Theorem~\ref{mainStepIV}, then by part (b) of that theorem, $C_\F(R)=\Psi_\L(C_\L(R))\unlhd\F$.  
\end{proof}

By Lemma~\ref{L:Regular10.4}, there exists a regular locality over $\F$. Thus, we can work under the following assumption to prove properties of $\F$.

\smallskip

\textbf{For the remainder of this subsection $(\L,\Delta,S)$ is a regular locality over $\F$.}

\begin{lemma}\label{L:QuasisimpleFusion}
Suppose $\F$ is quasisimple. Then $\F=O^{p^\prime}(\F)$ and $S$ is non-abelian. Moreover, if $\E$ is subnormal in $\F$, then $\E=\F$ or $\E\subseteq \F_{Z(\F)}(Z(\F))$.
\end{lemma}

\begin{proof}
By Proposition~\ref{SimpleQuasisimpleTranslate}(c), $\L$ is quasisimple. Thus, by Lemma~\ref{L:QuasisimpleOpprime}, $\L=O^{p^\prime}(\L)$, $S$ is non-abelian and every partial subnormal subgroup of $\L$ is either equal to $\L$ or contained in $Z(\L)$.  Corollary~\ref{C:Oupperp} gives now $\F=O^{p^\prime}(\F)$. Recall from Proposition~\ref{P:psihat} that $\hat{\Psi}_\L$ is a bijection from the set of partial subnormal subgroups of $\L$ to the set of subnormal subsystems of $\F$ which takes $\H\subn\L$ to $\F_{S\cap\H}(\H)$ and thus $\L$ to $\F_S(\L)=\F$. So if $\E$ is subnormal in $\F$ with $\E\neq\F$, then there exists $\H\subn\L$ with $\E=\F_{S\cap\H}(\H)$ and $\H\neq\L$. The property stated above together with Proposition~\ref{SimpleQuasisimpleTranslate}(b) gives then $\H\subseteq Z(\L)=Z(\F)$ and thus $\E=\F_{S\cap\H}(\H)\subseteq \F_{Z(\F)}(Z(\F))$. 
\end{proof}

The next theorem is now relatively easy to prove. Similar results were stated before by Aschbacher \cite[9.6,9.8,9.9,9.13]{Aschbacher:2011}. Recall the definition of an internal central product from Definition~\ref{D:InternalCentralProduct}. It should be noted in this context that, by Lemma~\ref{L:InternalExternalCentralProduct}, our notion of an internal central product of fusion systems is closely related to Aschbacher's notion of a central product \cite[p.14]{Aschbacher:2011}.

\begin{theorem}\label{T:AschbacherEF}
The following hold:
\begin{itemize}
\item [(a)] $E(\F)$ and $F^*(\F)$ are characteristic subsystems of $\F$.
\item [(b)] $E(\F)=\C_1*\C_2*\cdots *\C_n$ is an internal central product of the components $\C_1,\dots,\C_n$ of $\F$ (listed in arbitrary order).
\item [(c)] $\Comp(E(\F))=\Comp(F^*(\F))=\Comp(\F)$ and $E(\F)=F^*(E(\F))=E(E(\F))$.
\item [(d)] $E(\F)=O^p(E(\F))=O^p(F^*(\F))=O^{p^\prime}(E(\F))$.
\item [(e)] $E(\F)$ is a normal subsystem of $C_\F(O_p(\F))$. Moreover, $F^*(\F)=E(\F)*\F_{O_p(\F)}(O_p(\F))$ is an internal central product of $E(\F)$ and $\F_{O_p(\F)}(O_p(\F))$; thus $F^*(\F)$ is also an internal central product of $\F_{O_p(\F)}(O_p(\F))$ and the components of $\F$ (listed in arbitrary order).
\item [(f)] If $\C_1,\C_2,\dots,\C_k$ are pairwise distinct  components of $\F$, then $\C_1,\dots,\C_k$ centralize each other in $E(\F)$. Moreover,  $\C_1*\C_2*\cdots *\C_k$ is a normal subsystem of $E(\F)$ and of $F^*(\F)$, which is an internal central product of $\C_1,\dots,\C_k$ with 
\[O^p(\C_1*\C_2*\cdots *\C_k)=\C_1*\C_2*\cdots *\C_k=O^{p^\prime}(\C_1*\C_2*\cdots *\C_k).\] 
\item[(g)] $C_\F(E(\F))$ is constrained. 
\end{itemize}
\end{theorem}

\begin{proof}
We will use throughout that $F^*(\F)=\Psi_\L(F^*(\L))$ and $E(\F)=\Psi_\L(E(\L))$ by Lemma~\ref{GeneralizedFitting}(b),(c). Let  $\hat{\Psi}_\L$ always be the map from Proposition~\ref{P:psihat}. We will also use  that, by \cite[Proposition~11.7]{Henke:Regular},
\begin{equation}\label{E:Comp1}
\Comp(F^*(\L))=\Comp(\L)\mbox{ and }\K\unlhd F^*(\L)\mbox{ for every $\K\in\Comp(\L)$}.
\end{equation}
Moreover, we use that $E(\L)$ is the product of the components of $\L$ (cf. \cite[Definition~11.8, Theorem~11.18(a)]{Henke:Regular}). By Lemma~\ref{L:SubnormalBasic}(a),(c), this yields in particular that 
\begin{equation}\label{E:Comp2}
\Comp(E(\L))=\Comp(\L)\mbox{ and }E(E(\L))=E(\L).
\end{equation}

\smallskip

\textbf{(a)} By \cite[Lemma~9.19, Lemma~11.12]{Henke:Regular}, $F^*(\L)$ and $E(\L)$ are $\Aut(\L,S)$-invariant. Hence, by Proposition~\ref{P:CharacteristicSubsystems}, $F^*(\F)=\Psi_\L(F^*(\L))$ and $E(\F)=\Psi_\L(E(\L))$ are characteristic subsystems of $\F$.  This proves (a). Alternatively, part (a) can easily be checked directly by seeing that an automorphism of $\F$ takes components to components and normalizes $O_p(\F)$. 

\smallskip

\textbf{(c)} By definition of the map $\hat{\Psi}_\L$ from Proposition~\ref{P:psihat}, $\hat{\Psi}_\L$ restricts to $\hat{\Psi}_{E(\L)}$ and to $\hat{\Psi}_{F^*(\L)}$. Moreover, by Lemma~\ref{L:GeneralizedFitting}(a), $\hat{\Psi}_\L$ induces a bijection between $\Comp(\L)$ and $\Comp(\F)$; similarly $\hat{\Psi}_\L$ induces a bijection between $\Comp(F^*(\L))$ and $\Comp(F^*(\F))$ and between $\Comp(E(\L))$ and  $\Comp(E(\F))$. Hence, it follows from \eqref{E:Comp1} and \eqref{E:Comp2} that 
\[\Comp(\F)=\Comp(E(\F))=\Comp(F^*(\F)).\] 
By \eqref{E:Comp2} and  Lemma~\ref{GeneralizedFitting}(b),(c) (applied also with $E(\L)$ in place of $\L$) $E(\F)=\F_{S\cap E(\L)}(E(\L))=\F_{S\cap E(E(L))}(E(E(\L)))=E(\F_{S\cap E(\L)}(E(\L)))=E(E(\F))$. In particular $E(\F)=E(E(\F))\subseteq F^*(E(\F))\subseteq E(\F)$, i.e. equality holds and thus (c) is true.

\smallskip

\textbf{(b,d,f)} If $\K_1,\dots,\K_r$ are pairwise distinct components of $\L$, then it follows from \cite[Theorem~11.18(e)]{Henke:Regular} that $\K_1,\dots,\K_r$ commute pairwise. As mentioned above, $\hat{\Psi}_\L$ restricts to $\hat{\Psi}_{F^*(\L)}$ and $\hat{\Psi}_{E(\L)}$. Thus, by Proposition~\ref{P:psihat}(a), $\hat{\Psi}_\L$ restricts also to $\Psi_{F^*(\L)}$ and $\Psi_{E(\L)}$. Recall that, by \eqref{E:Comp1}, $\K_i\unlhd F^*(\L)$ and therefore also $\K_i\unlhd E(\L)$. So in particular,
\[\C_i:=\hat{\Psi}_\L(\K_i)=\Psi_{F^*(\L)}(\K_i)=\Psi_{E(\L)}(\K_i).\]
Set $\M:=\K_1\K_2\cdots \K_r$. Then $\M\unlhd F^*(\L)$ and $\M\unlhd E(\L)$ by Theorem~\ref{T:ProductsPartialNormal}(a). Applying  Proposition~\ref{P:N1N2Perpendicular} twice (once with $F^*(\L)$ and once with $E(\L)$ in place of $\L$), one sees that $\C_1,\dots,\C_r$ centralize each other in $E(\F)$ and that
\[\Psi_{F^*(\L)}(\M)=\C_1*\C_2*\cdots *\C_r=\Psi_{E(\L)}(\M)\]
is a normal subsystem of $F^*(\F)$ and $E(\F)$. Clearly $\C_1*\C_2*\cdots *\C_r$ is a central product of $\C_1,\dots,\C_r$. By \cite[Lemma~10.18, Theorem~11.11]{Henke:Regular}, $O^p_{F^*(\L)}(\M)=\M=O^{p^\prime}_{F^*(\L)}(\M)$. So by Proposition~\ref{Oupperp}(c),
\[O^p(\C_1*\C_2*\cdots\C_r)=O^p(\Psi_{F^*(\L)}(\M))=\Psi_{F^*(\L)}(O^p_{F^*(\L)}(\M))=\Psi_{F^*(\L)}(\M)=\C_1*\C_2*\cdots *\C_r\]
and similarly $O^{p^\prime}(\C_1*\C_2*\cdots *\C_r)=\C_1*\C_2*\cdots *\C_r$. It follows now from Proposition~\ref{GeneralizedFitting}(a) that (f) holds. As $E(\L)$ is the product of the components of $\L$, Proposition~\ref{GeneralizedFitting}(a) implies also that $E(\F)=\Psi_\L(E(\L))=\Psi_{F^*(\L)}(E(\L))$ is an internal central product of the components of $\F$, i.e. (b) holds. In particular, we have $E(\F)=O^p(E(\F))=O^{p^\prime}(E(\F))$. As $F^*(\F)=E(\F)O_p(\F)$ by definition, it follows from \cite[Theorem~1]{Henke:2013} that $O^p(F^*(\F))=O^p(E(\F))=E(\F)$ and so (d) holds. 

\smallskip

\textbf{(e)} By Lemma~\ref{GeneralizedFitting}(c), $E(\F)\unlhd C_\F(O_p(\F))$ and $F^*(\F)=E(\F)*\F_{O_p(\F)}(O_p(\F))$. Hence, the assertion follows from (b), Lemma~\ref{L:StarProdBasic}(c) and Lemma~\ref{L:StarAssociativeHelp}.

\smallskip

\textbf{(g)} It follows from \eqref{E:Comp2} and \cite[Lemma~11.16]{Henke:Regular} that $\Comp(E(\L)^\perp)=\emptyset$. As $E(\L)^\perp$ is a regular locality over $\F_{S\cap E(\L)^\perp}(E(\L)^\perp)$ by Theorem~\ref{T:mainRegularPartialNormal}(a), it follows from \cite[Lemma~11.6]{Henke:Regular} that $\F_{S\cap E(\L)^\perp}(E(\L)^\perp)$ is constrained. Moreover, by Theorem~\ref{mainStepIV}(b), $\Psi_\L(E(\L)^\perp)=\F_{S\cap E(\L)^\perp}(E(\L)^\perp)$. As $\Psi_\L(E(\L))=E(\F)$, it follows from Proposition~\ref{P:NperpCFE} that $C_\F(E(\F))=\Psi_\L(E(\L)^\perp)$ is constrained. So (g) holds.
\end{proof}

Part (a) of the following lemma was first proved by Aschbacher \cite[9.11]{Aschbacher:2011}. If $\E$ is a subsystem of $\F$ over $T$, then we write $O_p(\F)\subseteq\E$ to indicate that $O_p(\F)\leq T$.

\begin{lemma}\label{L:GeneralizedFitting}
\begin{itemize}
\item [(a)] $C_\F(F^*(\F))=\F_{Z(F^*(\F))}(Z(F^*(\F)))$.
\item [(b)] Set $\mathfrak{E}:=\{\E\unlhd\F\colon O_p(\F)\subseteq\E,\;C_\F(\E)\subseteq\E\}$. Then $F^*(\F)=\bigwedge_{\E\in\mathfrak{E}}\E$.  
\end{itemize}
\end{lemma}

\begin{proof}
Recall Definition~\ref{D:FstarELLocalities} and write $\mathfrak{N}$ for the set of partial normal subgroups $\N$ of $\L$ with $\N^\perp\subseteq\N$ and $O_p(\L)\subseteq\N$. By definition, $F^*(\L)=\bigcap_{\N\in\mathfrak{N}}\N$ and by \cite[Theorem~2]{Henke:Regular}, $F^*(\L)\in\mathfrak{N}$. By \cite[Proposition~5]{Henke:2015} and Lemma~\ref{L:OpL}, $O_p(\F)=O_p(\L)$. Hence, it follows from Corollary~\ref{C:CentricPartialNormal} that $\Psi_\L$ induces a bijection from $\mathfrak{N}$ to $\mathfrak{E}$ and that $C_\F(\E)=\F_{Z(\E)}(Z(\E))$ for every $\E\in\mathfrak{E}$. Together with Lemma~\ref{GeneralizedFitting}(c) this yields that $F^*(\F)=\Psi_\L(F^*(\L))\in\mathfrak{E}$ and (a) holds. Moreover, by Theorem~\ref{T:IntersectionsI}(d), 
\[\Psi_\L(F^*(\L))=\Psi_\L(\bigcap_{\N\in\mathfrak{N}}\N)=\bigwedge_{\N\in\mathfrak{N}}\Psi_\L(\N)=\bigwedge_{\E\in\mathfrak{E}}\E.\]
So (b) holds as well.
\end{proof}

\begin{lemma}\label{L:CharacteristicNormal}
Let $\E\unlhd\F$ and let $\m{D}$ be characteristic in $\E$. Then $\m{D}\unlhd\F$.
\end{lemma}

\begin{proof}
We use Theorem~\ref{mainStepIV} throughout, in particular part (b) of that theorem. Let $\E$ and $\m{D}$ be subsystems over $T$ and $R$ respectively. Pick $\N\unlhd\L$ with $\F_{S\cap \N}(\N)=\Psi_\L(\N)=\E$. Hence, by Theorem~\ref{T:mainRegularPartialNormal}(a), $(\N,\delta(\E),T)$ is a regular locality over $\E$. As $\m{D}\unlhd\E$, there exists $\M\unlhd\N$ with $\F_{S\cap\M}(\M)=\Psi_\N(\M)=\m{D}$. By Proposition~\ref{P:CharacteristicSubsystems}(b), $\M$ is $\Aut(\N,T)$-invariant. Therefore, it follows from \cite[Lemma~10.17]{Henke:Regular} applied with $\mathbb{K}=\{\M\}$ that $\M\unlhd\L$. Thus, $\m{D}=\F_{S\cap\M}(\M)=\Psi_\L(\M)\unlhd\F$.
\end{proof}

\begin{lemma}\label{L:EEtimeECFEEF}
 Let $\E\unlhd\F$. Then the following hold:
\begin{itemize}
\item [(a)] $E(\E)$, $F^*(\E)$ and $E(C_\F(\E))$ are normal in $\F$.
\item [(b)] The set of components of $\F$ is the disjoint union of $\Comp(\E)$ and $\Comp(C_\F(\E))$.
\item [(c)] $E(\E)$, $E(C_\F(\E))$ and $\F_{O_p(\F)}(O_p(\F))$ centralize each other, $E(\F)=E(\E)*E(C_\F(\E))$ and 
\[F^*(\F)=E(\F)*\F_{O_p(\F)}(O_p(\F))=E(\E)*E(C_\F(\E))*\F_{O_p(\F)}(O_p(\F)).\]
\end{itemize}
\end{lemma}

\begin{proof}
\textbf{(a)} Part (a) follows from Theorem~\ref{T:AschbacherEF}(a) and Lemma~\ref{L:CharacteristicNormal}; alternatively, part (a) can be concluded from \cite[Theorem~10.16(d), Lemma~11.13]{Henke:Regular}, Theorem~\ref{mainStepIV} and Lemma~\ref{GeneralizedFitting}(b),(c). 

\smallskip

\textbf{(b)} We use now Theorem~\ref{mainStepIV}, in particular part (b) of that theorem. Fix $\N\unlhd\L$ with $\F_{S\cap\N}(\N)=\Psi_\L(\N)=\E$. Then $\F_{S\cap\N^\perp}(\N^\perp)=\Psi_\L(\N^\perp)=C_\F(\E)$ by Proposition~\ref{P:NperpCFE}. So by Theorem~\ref{T:mainRegularPartialNormal}(a), $\N$ is a regular locality over $\E$ and $\N^\perp$ is a regular locality over $C_\F(\E)$. The map $\hat{\Psi}_\L\colon \{\H\colon\H\subn\L\}\rightarrow \{\m{D}\colon\m{D}\subn\F\},\H\mapsto \F_{S\cap\H}(\H)$ is by Proposition~\ref{P:psihat} a bijection, which then clearly restricts to $\hat{\Psi}_\N$ and $\hat{\Psi}_{\N^\perp}$. Now it follows Lemma~\ref{GeneralizedFitting}(a) that $\hat{\Psi}_\L$ induces a bijection $\Comp(\L)\rightarrow\Comp(\F)$, a bijection $\Comp(\N)\rightarrow\Comp(\E)$ and a bijection $\Comp(\N^\perp)\rightarrow \Comp(C_\F(\E))$.  By \cite[Lemma
11.16]{Henke:Regular}, $\Comp(\L)$ is the disjoint union of $\Comp(\N)$ and $\Comp(\N^\perp)$. Hence (b) follows.

\smallskip

\textbf{(c)} This follows from part (b), Lemma~\ref{L:StarProdBasic}(c), Lemma~\ref{L:StarAssociativeHelp}, Lemma~\ref{L:CharacterizeF1starF2}(a) and Theorem~\ref{T:AschbacherEF}(b),(e).
\end{proof}

To formulate the next lemma, it will be convenient to use the following notation.

\begin{notation}\label{N:PcapE}
If $\E$ is a subsystem of $\F$ over $T$ and $P\leq S$, then set $P\cap\E:=P\cap T$. In particular, $S\cap\E:=T$.  
\end{notation}

\begin{lemma}\label{L:SubnormalComponents}
Suppose $\E$ is a subnormal subsystem of $\F$. Then the following hold:
\begin{itemize}
 \item[(a)] $\Comp(\E)=\{\C\in\Comp(\F)\colon\C\subn\E\}=\{\C\in\Comp(\F)\colon \C\subseteq\E\}=\{\C\in\Comp(\F)\colon S\cap\C\leq S\cap\E\}$.
 \item[(b)] If $\C_1,\dots,\C_k\subseteq\Comp(\F)\backslash\Comp(\E)$ are pairwise distinct, then $\E,\C_1,\C_2,\dots,\C_k$ centralize each other in $\F$. Moreover, $\E*\C_1*\C_2*\cdots *\C_k$ is an internal central product of $\E,\C_1,\C_2,\dots,\C_k$ with $\Comp(\E*\C_1*\C_2*\cdots *\C_k)=\Comp(\E)\cup\{\C_1,\dots,\C_k\}$.
\end{itemize}
\end{lemma}

\begin{proof}
\textbf{(a)} If $\C\in\Comp(\E)$, then $\C\subn\E\subn\F$ and thus $\C\in\Comp(\F)$. So 
\[\Comp(\E)\subseteq\{\C\in\Comp(\F)\colon\C\subn\E\}\subseteq\{\C\in\Comp(\F)\colon \C\subseteq\E\}\subseteq\{\C\in\Comp(\F)\colon S\cap\C\leq S\cap\E\}.\]
Let now $\C\in\Comp(\F)$ with $S\cap\C\leq S\cap\E$. Then $\C\wedge\E$ is by Proposition~\ref{P:IntersectionsSubnormal}(a) a subsystem over $S\cap\C$ which is subnormal in $\C$ and $\E$. By Lemma~\ref{L:QuasisimpleFusion}, $S\cap\C$ is non-abelian and every subnormal subsystem of $\C$ is equal to $\C$ or contained in $\F_{Z(\C)}(Z(\C))$. Thus, $\C\wedge\E=\C$. In particular $\C\subn\E$ and thus $\C\in\Comp(\E)$. This proves (a).

\smallskip

\textbf{(b)} Let $\C_1,\dots,\C_k$ be as in (b). The map $\hat{\Psi}_\L\colon \{\H\colon\H\subn\L\}\rightarrow \{\m{D}\colon\m{D}\subn\F\},\H\mapsto \F_{S\cap\H}(\H)$ is by Proposition~\ref{P:psihat}(a) a bijection, which restricts by Lemma~\ref{GeneralizedFitting}(a) to a bijection between $\Comp(\L)$ and $\Comp(\F)$. So $\H:=\hat{\Psi}_\L^{-1}(\E)\subn\L$ and $\K_i:=\hat{\Psi}_\L^{-1}(\C_i)\in\Comp(\L)$ for $i=1,\dots,k$. By \cite[Theorem~11.18(c)]{Henke:Regular}, $\M:=\H\prod_{i=1}^k\K_i$ is a partial subnormal subgroup of $\L$ which is an internal central product of $\H$, $\K_1,\dots,\K_k$ with $\Comp(\M)=\Comp(\H)\cup\{\K_1,\dots,\K_k\}$. So it follows from Lemma~3.5, Lemma~4.8 and Lemma~4.9 in \cite{Henke:Regular} that $\H,\K_1,\dots,\K_k$ are partial normal subgroups of $\M$ which commute pairwise. Notice that $\hat{\Psi}_\L$ restricts to $\hat{\Psi}_\M$ and thus, by Proposition~\ref{P:psihat}(a), to $\Psi_\M$. Now $\Psi_\M(\M)=\hat{\Psi}_\L(\M)\subn\F$, $\Psi_\M(\H)=\hat{\Psi}_\M(\H)=\E$ and $\Psi_\M(\K_i)=\hat{\Psi}_\M(\K_i)=\C_i$. By Corollary~\ref{C:RegularSubnormal}, $\M$ is a regular locality over $\Psi_\M(\M)$. So by  Proposition~\ref{P:N1N2Perpendicular}, $\E,\C_1,\dots,\C_k$ centralize each other in $\Psi_\M(\M)$ and thus in $\F$; moreover, $\Psi_\M(\M)=\E*\C_1*\C_2*\cdots\C_k$  is an internal central product of $\E,\C_1,\dots,\C_k$. By Lemma~\ref{GeneralizedFitting}(a), $\hat{\Psi}_\M$ induces a bijection between $\Comp(\M)$ and $\Comp(\Psi_\M(\M))$; similarly, $\hat{\Psi}_\H$ induces a bijection between $\Comp(\H)$ and $\Comp(\E)$. Thus, as $\hat{\Psi}_\M$ restricts to $\hat{\Psi}_\H$ and $\Comp(\M)=\Comp(\H)\cup\{\K_1,\dots,\K_k\}$, part (b) follows. 
\end{proof}

\begin{corollary}\label{C:LayerNFRCFR}
For every $R\in \F^f$, $\Comp(N_\F(R))=\Comp(C_\F(R))$ and $E(N_\F(R))=E(C_\F(R))$. 
\end{corollary}

\begin{proof}
 By Lemma~\ref{L:CFRNFR}, $C_\F(R)\unlhd N_\F(R)$ and by Lemma~\ref{GeneralizedFitting}(c), $E(N_\F(R))\subseteq C_{N_\F(R)}(O_p(N_\F(R)))\subseteq C_\F(R)$. Hence, it follows from Lemma~\ref{L:SubnormalComponents}(a) that $\Comp(N_\F(R))=\Comp(C_\F(R))$. Now the assertion follows from Theorem~\ref{T:AschbacherEF}(b).
\end{proof}

\begin{lemma}\label{L:ComponentsResiduesFusion}
If $\E\unlhd\F$ is a saturated subsystem of $\F$ of $p$-power index or of index prime to $p$, then $\Comp(\F)=\Comp(\E)$ and $E(\F)=E(\E)$. In particular,
\[\Comp(\F)=\Comp(O^p(\F))=\Comp(O^{p^\prime}(\F))\mbox{ and }E(\F)=E(O^p(\F))=E(O^{p^\prime}(\F)).\]
\end{lemma}

\begin{proof}
Assume first that $\E\unlhd\F$ is of $p$-power index or of index prime to $p$. Let $\H:=\Psi_\L^{-1}(\E)$. By Proposition~\ref{Oupperp}(a),(b), $\H$ has $p$-power index or index prime to $p$ in $\L$. Hence, by \cite[Lemma~11.14]{Henke:Regular}, $\Comp(\H)=\Comp(\L)$. By Theorem~\ref{T:mainRegularPartialNormal}(a), $\H$ is a regular locality over $\E$. As the map $\hat{\Psi}_\L$ from Proposition~\ref{P:psihat} restricts to $\hat{\Psi}_\H$, it follows from Lemma~\ref{GeneralizedFitting}(a) that $\Comp(\F)=\{\hat{\Psi}_\L(\K)\colon \K\in\Comp(\L)\}=\{\hat{\Psi}_\H(\K)\colon\K\in\Comp(\H)\}=\Comp(\E)$. Hence, $E(\F)=E(\E)$ by Theorem~\ref{T:AschbacherEF}(b).
This proves in particular that $\Comp(\F)=\Comp(O^p(\F))=\Comp(O^{p^\prime}(\F))\mbox{ and }E(\F)=E(O^p(\F))=E(O^{p^\prime}(\F))$.

\smallskip

Suppose now that $\E$ is an arbitrary saturated subsystem of $\F$ of $p$-power index or of index prime to $p$. If $\E$ has $p$-power index, then by Lemma~\ref{L:pPowerIndexProducts}, $O^p(\E)=O^p(\F)$, so $\Comp(\F)=\Comp(O^p(\F))=\Comp(O^p(\E))=\Comp(\E)$ and $E(\F)=E(\E)$. By \cite[Theorem~I.7.7]{Aschbacher/Kessar/Oliver:2011}, $O^{p^\prime}(\F)$ (as we defined it in Subsection~\ref{SS:ResiduesFusionLoc}) 
is the smallest saturated subsystem of $\F$ of index prime to $p$. Hence, if $\E$ has index prime to $p$, then $O^{p^\prime}(\E)=O^{p^\prime}(\F)$. Hence, $\Comp(\E)=\Comp(O^{p^\prime}(\E))=\Comp(O^{p^\prime}(\F))=\Comp(\F)$ and $E(\E)=E(\F)$. 
\end{proof}

\begin{corollary}
 Let $\E$ be a normal subsystem of $\F$ over $T\leq S$. Then 
\[\Comp(C_\F(\E))=\Comp(C_\F(T))=\Comp(N_\F(T))\mbox{ and }E(C_\F(\E))=E(C_\F(T))=E(N_\F(T)).\]
\end{corollary}

\begin{proof}
 By definition, $C_\F(\E)$ is a saturated subsystem of $C_\F(T)$ over $C_S(\E)$ of $p$-power index (see Subsection~\ref{SS:CFE}). Hence, $\Comp(C_\F(\E))=\Comp(C_\F(T))$ and $E(C_\F(\E))=E(C_\F(T))$ by Lemma~\ref{L:ComponentsResiduesFusion}. Corollary~\ref{C:LayerNFRCFR} yields now the assertion. 
\end{proof}

\begin{lemma}[{Aschbacher \cite[14.2]{Aschbacher:2011}}]
The following conditions are equivalent:
\begin{itemize}
 \item[(i)] $\F$ is constrained;
 \item[(ii)] $F^*(\F)=\F_{O_p(\F)}(O_p(\F))$;
 \item[(iii)] $\Comp(\F)=\emptyset$. 
\end{itemize}
\end{lemma}

\begin{proof}
By \cite[Lemma~11.6]{Henke:Regular}, (i) is equivalent to $F^*(\L)=O_p(\L)$ and to $\Comp(\L)=\emptyset$. By \cite[Proposition~5]{Henke:2015} and Lemma~\ref{L:OpL}, $O_p(\L)=O_p(\F)$. So the assertion follows from Lemma~\ref{GeneralizedFitting}.
\end{proof}

\subsection{The proof of Theorem~\ref{main}}\label{SS:TheoremAProof}

We have now revisited the background on fusion systems which is necessary to complete the proof of  Theorem~\ref{main}. The essential part that was missing before is the following lemma.

\begin{lemma}\label{L:mainHelp}
Let $\E\unlhd \F$. Then $Q_1Q_2O_p(\F)\in\F^*(\F)^{cr}$ for all $Q_1\in E(\E)^{cr}$ and $Q_2\in E(C_\F(\E))^{cr}$. 
\end{lemma}

\begin{proof}
By Lemma~\ref{L:EEtimeECFEEF}(c), we have $E(\F)=E(\E)*E(C_\F(\E))$ and $F^*(\F)=E(\F)*\F_{O_p(\F)}(O_p(\F))$. Therefore, applying \cite[Lemma~2.2, Lemma~2.14(f)]{Henke:Regular} twice, we get first $Q_1Q_2\in E(\F)^{cr}$ and then $Q_1Q_2O_p(\F)\in F^*(\F)^{cr}$. 
\end{proof}

\begin{proof}[Proof of Theorem~\ref{main}] 
Let $\Psi_\L$ be as in Theorem~\ref{mainStepIV} and notice that part (a) of this theorem asserts that Theorem~\ref{main}(a) holds. Recall Notation~\ref{N:PcapE}. By Lemma~\ref{L:GeneralizedFitting}(a), $C_S(F^*(\F))\subseteq S\cap F^*(\F)$. Therefore, it follows from Lemma~\ref{L:EcsubsetFq} that $F^*(\F)^{cr}\subseteq F^*(\F)^c\subseteq\F^q$. Similarly, $F^*(\F)^{cr}\subseteq F^*(\F)^s\subseteq\delta(\F)$. Hence, for the proof of (b) it remains to show the following: If $F^*(\F)^{cr}\subseteq\Delta$, then $\Psi_\L(\N)=\F_{S\cap\N}(\N)$ for all $\N\unlhd\L$. 

\smallskip

So suppose now that $F^*(\F)^{cr}\subseteq\Delta$. Let $\E$ be a normal subsystem of $\F$. If $P_1\in\E^{cr}$ and $P_2\in C_\F(\E)^{cr}$, then 
it follows from \cite[Lemma~1.20(d)]{AOV1} that $Q_1:=P_1\cap E(\E)\in E(\E)^{cr}$ and $Q_2:=P_2\cap E(C_\F(\E))\in E(C_\F(\E))^{cr}$. Hence, Lemma~\ref{L:mainHelp} gives that $Q_1Q_2O_p(\F)\in F^*(\F)^{cr}\subseteq \Delta$. As $\Delta$ is overgroup-closed, this implies $P_1P_2O_p(\F)\in\Delta$. Therefore, it follows from Theorem~\ref{mainStepII} that there exists a unique partial normal subgroup $\N$ of $\L$ with $\F_{S\cap\N}(\N)=\E$. As $\E\unlhd\F$ was arbitrary, this means that there is a map
\[\Theta\colon \fN(\F)\rightarrow\fN(\L),\E\mapsto\Theta(\E)\]
such that $\E=\F_{S\cap\Theta(\E)}(\Theta(\E))$ for all $\E\unlhd\F$. Such a map is clearly injective. As $\Psi_\L$ is bijective, we have $|\fN(\L)|=|\fN(\F)|$. Hence, $\Theta$ is bijective as well. Observe now that $\Theta^{-1}(\N)=\F_{S\cap \N}(\N)$ for all $\N\unlhd\L$. In particular, for all $\N\in\fN(\L)$, we have $\F_{S\cap\N}(\N)\unlhd\F$. Thus, $\F_{S\cap\N}(\N)$ is equal to $\Psi_\L(\N)$ by the characterization of $\Psi_\L(\N)$ given in Theorem~\ref{mainStepIV}(a).

\smallskip

If $(\fN(\F),\unlhd)$ is a poset, then Theorem~\ref{main}(c) is equivalent to Theorem~\ref{mainStepIV}(d),(e). Clearly, the normality relation on $\fN(\F)$ is reflexive and antisymmetric. To prove transitivity let $\E_1,\E_2,\E_3\in\fN(\F)$ with $\E_1\unlhd\E_2$ and $\E_2\unlhd\E_3$. Then by Theorem~\ref{mainStepIV}(d), $\Psi_\L^{-1}(\E_1)\subseteq \Psi_\L^{-1}(\E_2)\subseteq\Psi_\L^{-1}(\E_3)$. Hence, by Theorem~\ref{mainStepIV}(e), $\E_1=\Psi_\L(\Psi_\L^{-1}(\E_1))\unlhd\Psi_\L(\Psi_\L^{-1}(\E_3))=\E_3$. This shows that $(\fN(\F),\unlhd)$ is a poset and completes thus the proof.
\end{proof}

\subsection{An alternative characterization of $\delta(\F)$}

As a consequence of Lemma~\ref{GeneralizedFitting}(c) we can characterize the set $\delta(\F)$ now without mentioning localities at all. Recall Notation~\ref{N:PcapE}.

\begin{lemma}\label{L:deltaF}
$\delta(\F)=\{P\leq S\colon P\cap F^*(\F)\in\F^s\}=\{P\leq S\colon P\cap F^*(\F)\in F^*(\F)^s\}$.
\end{lemma}

\begin{proof}
By Lemma~\ref{L:Regular10.4}, there exists a regular locality $(\L,\Delta,S)$ over $\F$. Hence, the assertion follows from Lemma~\ref{GeneralizedFitting}(c) combined with Lemma~10.2 and  Lemma~10.11(f) in \cite{Henke:Regular} (see also \cite[Remark~10.12]{Henke:Regular}).
\end{proof}

The following lemma will be useful in the next subsection.

\begin{lemma}\label{deltaLemma}
$\delta(E(\F))=E(\F)^s\subseteq F^*(\F)^s\subseteq \delta(\F)$. 
\end{lemma}

\begin{proof}
Lemma~\ref{L:deltaF} gives $F^*(\F)^s\subseteq\delta(\F)$. As $F^*(E(\F))=E(\F)$ by Theorem~\ref{T:AschbacherEF}(c), it follows also from Lemma~\ref{L:deltaF} that  $\delta(E(\F))=E(\F)^s$. By Theorem~\ref{T:AschbacherEF}(d),   
 $E(\F)=O^p(F^*(\F))$ has $p$-power index in $F^*(\F)$. Therefore, \cite[Proposition~3(d)]{Henke:2015} yields $E(\F)^s\subseteq F^*(\F)^s$. 
\end{proof}

\section{Products with $p$-subgroups}

\textbf{Throughout let $\F$ be a saturated fusion system over $S$.}

\begin{lemma}\label{ProductPSubgroupRegularLocality1}
 Let $\E$ be a normal subsystem of $\F$ over $T$. Then the following hold:
\begin{itemize}
 \item [(a)] $\Comp(\E S)=\Comp(\E)$ and $E(\E S)=E(\E)$. 
 \item [(b)] $\delta(\F)\subseteq \delta(\E S)$.
 \item [(c)] For every $P\in\delta(\E S)$, we have $PO_p(\E S)\in\delta(\F)$.
 \item [(d)] Let $(\L,\Delta,S)$ be a regular locality over $\F$ and suppose $\N$ is a partial normal subgroup of $\L$ with $\E=\F_{S\cap\N}(\N)$. Then $(\N R,\delta(\E R),TR)$ is a regular locality over $\E R$ for every $R\leq S$.
\end{itemize}
\end{lemma}

\begin{proof}
\textbf{(a)} By Lemma~\ref{L:pPowerIndexProducts} and Lemma~\ref{CSENormalER}, $\E$ is a normal subsystem of $\E S$ of $p$-power index. Hence, by
Lemma~\ref{L:ComponentsResiduesFusion}, we have $E(\E S)=E(\E)$. So (a) holds.

\smallskip

For the proof of the remaining parts let $(\L,\Delta,S)$ be a regular locality over $\F$ (which exists always by Lemma~\ref{L:Regular10.4}). By Theorem~\ref{mainStepIV} (in particular by part (b) of that theorem), there exists furthermore $\N\unlhd\L$ with $\F_{S\cap \N}(\N)=\Psi_\L(\N)=\E$. By Corollary~\ref{C:ERNR0}, $(\N S,\Delta,S)$ is a proper locality over $\E S$. Set
\[S_1:=S\cap E(C_\F(\E))\mbox{ and }S_0:=S_1O_p(\F).\]
It follows from Lemma~\ref{L:EEtimeECFEEF}(a) that $S_1\unlhd S$. Moreover, by \cite[Proposition~4]{Henke:2015}, $S_1\leq C_S(\E)=C_S(\N)$ and so $\N\subseteq C_\L(S_1)\subseteq N_\L(S_1)$ by \cite[Lemma~3.5]{Henke:Regular}. By \cite[Lemma~5.5]{Henke:2015}, $N_\L(S_1)$ is a partial subgroup of $\L$. Thus, $\N S\subseteq N_\L(S_1)$. By \cite[Proposition~5]{Henke:2015} and Lemma~\ref{L:OpL}, $O_p(\F)=O_p(\L)\unlhd\N S$. So it follows 
\begin{equation}\label{E:S0Normal}
S_0\unlhd\N S\mbox{ and thus }S_0\unlhd \E S=\F_S(\N S).  
\end{equation}
Notice also that, by Lemma~\ref{L:EEtimeECFEEF}(c) and Lemma~\ref{L:StarProdBasic}(c), 
\begin{equation}\label{E:EEstarD}
F^*(\F)=E(\E)*\m{D}\mbox{ where }\m{D}:=E(C_\F(\E))*\F_{O_p(\F)}(O_p(\F)).
\end{equation}
Observe moreover that $\m{D}$ is a fusion system over $S_0$.

\smallskip

\textbf{(b)} Let $P\in \delta(\F)$ and set $\tilde{P}:=PS_0$. As $\delta(\F)$ is overgroup closed, we have $\tilde{P}\in\delta(\F)$ and thus $\tilde{P}\cap F^*(\F)\in F^*(\F)^s$ by Lemma~\ref{L:deltaF}. By \eqref{E:EEstarD} $S\cap F^*(\F)=(S\cap E(\E))S_0$ and thus $\tilde{P}\cap F^*(\F)=(\tilde{P}\cap E(\E))S_0$. Thus, it follows from \eqref{E:EEstarD} and \cite[Lemma~2.14(g)]{Henke:Regular} that $\tilde{P}\cap E(\E)\in E(\E)^s$. Hence, by (a) and Lemma~\ref{deltaLemma}, we have $\tilde{P}\cap E(\E S)\in E(\E S)^s\subseteq \delta(\E S)$. As $\delta(\E S)$ is overgroup closed, it follows $\tilde{P}\in\delta(\E S)$. As $S_0$ is by \eqref{E:S0Normal} normal in $\N S$ and $(\N S,\Delta,S)$ is a proper locality over $\E S$, \cite[Lemma~7.3]{ChermakIII} or \cite[Lemma~10.6]{Henke:Regular} yields now $P\in\delta(\E S)$. This proves (b).

\smallskip

\textbf{(c)} Let $P\in\delta(\E S)$ and set $\hat{P}:=PO_p(\E S)$. As $\delta(\E S)$ is overgroup closed, we have $\hat{P}\in\delta(\E S)$ and thus, by Lemma~\ref{L:deltaF}, $\hat{P}\cap F^*(\E S)\in F^*(\E S)^s$. By (a) and Theorem~\ref{T:AschbacherEF}(e), $F^*(\E S)=E(\E)*\F_{O_p(\E S)}(O_p(\E S))$, so in particular, $\hat{P}\cap F^*(\E S)=(\hat{P}\cap E(\E))O_p(\E S)$. Now \cite[Lemma~2.14(g)]{Henke:Regular} yields $\hat{P}\cap E(\E)\in E(\E)^s$. By \eqref{E:S0Normal}, $S_0\leq O_p(\E S)\leq \hat{P}$. As $\hat{P}\cap E(\E)\in E(\E)^s$, it follows thus from \eqref{E:EEstarD} and \cite[Lemma~2.14(g)]{Henke:Regular} that $\hat{P}\cap F^*(\F)=(\hat{P}\cap E(\E))S_0\in F^*(\F)^s$. Thus, by Theorem~\ref{L:deltaF}, $\hat{P}\in\delta(\F)$. This shows (c). 

\smallskip

\textbf{(d)} Recall from above that $(\L,\Delta,S)$ is a regular locality over $\F$, in particular $\Delta=\delta(\F)$. Moreover, $(\N S,\Delta,S)$ is a proper locality over $\E S$. In particular, by \cite[Proposition~5]{Henke:2015}, $O_p(\E S)\unlhd \N S$. It follows from (b) and (c)  and from \cite[Lemma~7.3]{ChermakIII} or \cite[Lemma~10.6]{Henke:Regular} that 
\[\delta(\E S)=\{P\leq S\colon PO_p(\E S)\in\delta(\F)\}.\]
Thus, Lemma~\ref{L:VaryObjects} yields that $(\N S,\delta(\E S),S)$ is a proper locality and hence
\begin{equation}\label{E:NSRegular}
(\N S,\delta(\E S),S)\mbox{ is a regular locality over $\E S$.}
\end{equation}
Let now $R\leq S$. As $S$ is a $p$-group, $R$ is subnormal in $S$. If $R=R_0\unlhd R_1\unlhd\dots\unlhd R_n=S$ is a subnormal series, then by Lemma~\ref{ProductSubnormal}, $\N R=\N R_0\unlhd \N R_1\unlhd \dots \unlhd \N R_n=\N S$. It follows thus from \eqref{E:NSRegular} and Corollary~\ref{C:RegularSubnormal} that $\F_0:=\F_{T R}(\N R)$ is saturated and $(\N R,\delta(\F_0),TR)$ is a regular locality over $\F_0$. By Lemma~\ref{ProductWithPSubgroup}, $\F_0\subseteq \E R$ and thus $O^p(\F_0)\subseteq O^p(\E R)=O^p(\E)$. As $\E\subseteq \F_0$, it follows $O^p(\F_0)=O^p(\E)$. By \cite[Theorem~1]{Henke:2013}, $\E R$ is the unique saturated fusion system $\mathcal{D}$ on $TR$ with $O^p(\mathcal{D})=O^p(\E)$. Hence, $\F_0=\E R$. This shows the assertion.
\end{proof}

%\begin{prop}\label{ProductPSubgroupContainRegular}
%Let $(\L,\Delta,S)$ be a proper locality over $\F$ such that $\delta(\F)\subseteq\Delta$. Let $\N\unlhd\L$, $T:=S\cap\N$ and $\E:=\F_T(\N)$. Then $\E\unlhd\F$ and $\F_{T R}(\N R)=\E R$ for every $R\leq S$. Moreover, $(\N S,\Delta,S)$ and $(\N S,\Delta\cup\delta(\E S),S)$ are proper localities over $\E S$.
%\end{prop}

%\begin{proof}

%\end{proof}

\begin{proof}[Proof of Theorem~\ref{T:mainProductWithPSubgroup}]
Property (a) is Proposition~\ref{ProductPSubgroupRegularLocality1}(d) and property (c) is Corollary~\ref{C:ERNR0}. Thus it remains to show part (b). For that let $(\L,\Delta,S)$ be a proper locality over $\F$ such that $\delta(\F)\subseteq\Delta$. Let moreover $\N\unlhd\L$, $R\leq S$, $T:=S\cap\N$ and $\E:=\F_T(\N)$.

\smallskip

Notice that $\L_0:=\L|_{\delta(\F)}$ is a regular locality over $\F$. Set  $\N_0:=\N\cap \L_0$ and observe $T=\N_0\cap S$. By Theorem~\ref{mainStepIV}(b),(c), $\E=\Psi_\L(\N)=\Psi_{\L_0}(\N_0)=\F_T(\N_0)$ is normal in $\F$. In particular, $\E R$ is well-defined.  By Proposition~\ref{ProductPSubgroupRegularLocality1}(d), $\E R=\F_{T R}(\N_0 R)$. So by Lemma~\ref{ProductWithPSubgroup}, $\F_{TR}(\N R)\subseteq \E R=\F_{TR}(\N_0 R)\subseteq \F_{TR}(\N R)$ and thus $\F_{TR}(\N R)=\E R$.

\smallskip

By Corollary~\ref{C:ERNR0}, $(\N S,\Delta,S)$ is a proper locality over $\E S$. In particular, by \cite[Proposition~5]{Henke:2015}, $X:=O_p(\E S)$ is normal in $\N S$. Set $\Delta^+:=\Delta\cup\delta(\E S)$. By Lemma~\ref{ProductPSubgroupRegularLocality1}(c), we have $PX\in\delta(\F)\subseteq\Delta$ for every $P\in\delta(\E S)$. So using that $\Delta$ is overgroup closed in $S$, we can conclude that
\[\Delta\subseteq\Delta^+\subseteq \{P\leq S\colon PX\in\Delta\}.\]
Notice also that $\Delta^+$ is $\E S$-closed, as $\Delta$ and $\delta(\E S)$ are $\E S$-closed (cf. Lemma~\ref{L:Regular10.4}). Hence, it follows from Lemma~\ref{L:VaryObjects} that $(\N S,\Delta^+,S)$ is a proper locality over $\E S$ and so the assertion holds.
\end{proof}

\section{Products of normal subsystems}

\begin{theorem}\label{ProductTheorem}
Let $(\L,\Delta,S)$ be a regular locality over $\F$. Suppose $\N_1,\N_2,\dots,\N_k$ are partial normal subgroups of $\L$. For each $i=1,2,\dots,k$ set $T_i:=\N_i\cap S$ and $\E_i:=\F_{T_i}(\N_i)$. Furthermore, set 
\[\N=\N_1\N_2\cdots\N_k\mbox{ and }T:=T_1 T_2\dots T_k.\]
Then $\E_1,\E_2,\dots,\E_k$ are normal in $\F$ and the following hold:
\begin{itemize}
\item [(a)] $\N$ is a partial normal subgroup of $\L$ with $\N\cap S=T$, and $\F_T(\N)$ is the unique smallest normal subsystem of $\F$ over $T$ containing each $\E_i$ as a normal subsystem ($i=1,2,\dots,k$). 
\item [(b)] If $\tE\subn\F$ with $\E_i\subn\tE$ for all $i=1,2,\dots,k$, then $\E_i\unlhd\tE$ for all $i=1,2,\dots,k$ and $\F_T(\N)\unlhd\tE$.
\item [(c)] $\F_T(\N)=\<\F_T(\N_1 T),\F_T(\N_2 T),\dots,\F_T(\N_kT)\>=\<\E_1T,\E_2T,\dots,\E_kT\>$. 
\end{itemize}
\end{theorem}

\begin{proof}
\textbf{(a,b)} By Theorem~\ref{T:ProductsPartialNormal}(a), $\N$ is a partial normal subgroup of $\L$ with $\N\cap S=T$. By Theorem~\ref{mainStepIV}, in particular by part (b) of that theorem, $\E:=\F_T(\N)=\Psi_\L(\N)\unlhd\F$ and similarly $\E_i=\Psi_\L(\N_i)\unlhd\F$ for $i=1,2,\dots,k$. Clearly, $\m{E}$ is a subsystem over $T$. As $\N_i\subseteq\N$, it follows from Theorem~\ref{mainStepIV}(e) that $\E_i$ is normal in  $\m{E}$ for $i=1,\dots,k$. Thus (a) holds if (b) is true.

\smallskip

Let now $\tE$ be a subnormal subsystem of $\F$ in which $\E_1,\E_2,\dots,\E_k$ are subnormal. Consider the map $\hat{\Psi}_\L$ from Proposition~\ref{P:psihat}, which restricts to $\Psi_\L$. Proposition~\ref{P:psihat}(b) gives then $\N_i=\Psi_\L^{-1}(\E_i)=\hat{\Psi}_\L^{-1}(\E_i)\subseteq\hat{\Psi}_\L^{-1}(\tE)$ for each $i=1,2,\dots,k$. This implies $\N_i\unlhd \hat{\Psi}_\L^{-1}(\tE)$ and $\N=\N_1\N_2\cdots\N_k\unlhd \Psi_\L^{-1}(\tE)$. So Proposition~\ref{P:psihat}(c) yields $\E_i=\hat{\Psi}_\L(\N_i)\unlhd\tE$ and $\E=\Psi_\L(\N)=\hat{\Psi}_\L(\N)\unlhd \tE$. This shows (b) and thus (a) holds as well.

\smallskip

\textbf{(c)} By Proposition~\ref{ProductPSubgroupRegularLocality2}, we have $\F_T(\N_iT)=\E_iT$ for $i=1,2,\dots,k$. Moreover, as $\N_iT\subseteq \N$, it follows $\F_T(\N_iT)\subseteq\E$ for $i=1,2,\dots,k$. Hence, 
\[\<\E_1T,\E_2T,\dots,\E_kT\>=\<\F_T(\N_1T),\F_T(\N_2T),\dots,\F_T(\N_kT)\>\subseteq\E.\]
The fusion system $\E$ is by definition generated by the group homomorphisms which are of the form $c_f|_{S_f\cap T}\colon S_f\cap T\rightarrow T$ with $f\in\N$. Fix $f\in\N$. By Theorem~\ref{T:ProductsPartialNormal}(d), there exist $n_i\in\N_i$ for $i=1,\dots,k$ such that $f=\Pi(n_1,n_2,\dots,n_k)$ and $S_f=S_{(n_1,n_2,\dots,n_k)}$. As $T$ is strongly $\F$-closed by \cite[Lemma~3.1(a)]{Chermak:2015}, it follows that $c_f|_{S_f\cap T}$ is the composition of restrictions of the morphisms  $c_{n_i}|_{S_{n_i}\cap T}$ ($i=1,\dots,k$). As $c_{n_i}|_{S_{n_i}\cap T}$ is a morphism in $\F_T(\N_i T)$ for each $i=1,\dots,k$, it follows that $\E\subseteq\<\F_T(\N_1T),\F_T(\N_2T),\dots,\F_T(\N_k T)\>$. Hence
\[\E=\<\F_T(\N_1T),\F_T(\N_2T),\dots,\F_T(\N_kT)\>=\<\E_1T,\E_2T,\dots,\E_kT\>.\]
This shows the assertion.
\end{proof}

The following theorem is an extended version of Theorem~\ref{ProductCorollary}.

\begin{theorem}\label{T:ProductCorollaryDetails}
Let $\F$ be a saturated fusion system over $S$. Let $k\geq 1$ and let $\m{E}_i$ be a normal subsystem of $\F$ over $T_i$ for $i=1,2,\dots,k$. Set $T:=T_1T_2\cdots T_k$. Then there exists a subsystem $\E_1\E_2\cdots\E_k$ of $\F$ such that the following hold:
\begin{itemize}
\item [(a)] $\E_1\E_2\cdots\E_k$ is the unique smallest normal subsystem of $\F$ over $T$ containing each $\E_i$ as a normal subsystem ($i=1,2,\dots,k$).
\item [(b)] If $\tE\subn\F$ such that $\E_i\subn\tE$ for all $i=1,2,\dots,k$, then $\E_1\E_2\cdots\E_k\unlhd\tE$.
\item [(c)] $\E_1\E_2\cdots\E_k$ is generated by $\E_1T,\E_2T,\dots,\E_kT$. In particular, setting 
\[\Ac_i(P):=\<\phi\in \Aut_\F(P):\phi\mbox{ $p^\prime$-element, }[P,\phi]\leq P\cap T_i\mbox{ and }\phi|_{P\cap T_i}\in \Aut_{\E_i}(P\cap T_i)\>\]
for each $i=1,2,\dots,k$ and each $P\leq T$, it follows that 
\[\E_1\E_2\cdots\E_k=\<\Ac_i(P)\colon i\in\{1,2,\dots,k\},\;P\leq T\mbox{ with }P\cap T_i\in\E_i^c\>_T\] 
\item [(d)] The product of normal subsystems is associative, i.e. if $1\leq l<k$, then \[\E_1\E_2\cdots\E_k=(\E_1\cdots\E_l)(\E_{l+1}\cdots\E_k).\]
\item [(e)] If $\E_1,\E_2,\dots,\E_k$ centralize each other in $\F$, then $\E_1\E_2\cdots\E_k=\E_1*\E_2*\cdots *\E_k$ is an internal central product of $\E_1,\E_2,\dots,\E_k$ (as defined in Definition~\ref{D:InternalCentralProduct}).
\item [(f)] Suppose $(\L,\Delta,S)$ is a proper locality over $\F$. If $\Psi_\L$ is the map from Theorem~\ref{main} and  $\N_i:=\Psi_\L^{-1}(\E_i)$ for $i=1,2,\dots,k$, then  
\[\Psi_\L(\N_1\N_2\cdots\N_k)=\E_1\E_2\cdots\E_k.\] 
\end{itemize}
\end{theorem}

\begin{proof}
By Lemma~\ref{L:Regular10.4}, there exists a regular locality $(\L^\delta,\delta(\F),S)$ over $\F$. By Theorem~\ref{mainStepIV} (in particular by part (b) of that theorem), there exist partial normal subgroups $\N_i^\delta\unlhd\L^\delta$ with $\N_i^\delta\cap S=T_i$ and $\F_{T_i}(\N_i^\delta)=\Psi_{\L^\delta}(\N_i^\delta)=\m{E}_i$ for $i=1,2,\dots,k$. Set now
\[\E_1\E_2\cdots\E_k:=\F_T(\N_\delta).\] 
where
\[\N^\delta:=\N_1^\delta\N_2^\delta\cdots\N_k^\delta,\mbox{ and }T:=T_1T_2\cdots T_k.\]
\textbf{(a,b,c,d,e)} With this definition of $\E_1\E_2\cdots\E_k$ it follows from Theorem~\ref{ProductTheorem} that parts (a) and (b) of the theorem hold and moreover
\[\E_1\E_2\cdots\E_k=\<\E_1T,\E_2T,\dots,\E_kT\>.\]
Using Definition~\ref{D:Product}, the latter equality implies part (c). It follows from Theorem~\ref{T:ProductsPartialNormal}(b) that part (d) holds. Moreover, Proposition~\ref{P:N1N2Perpendicular} yields (e).

\smallskip

\textbf{(f)} Let $(\L,\Delta,S)$ be a proper locality over $\F$. By Theorem~\ref{T:VaryObjects}(d), there exists a subcentric locality $(\L^s,\F^s,S)$ over $\F$ such that $\L=\L^s|_\Delta$. As $\F^{cr}\subseteq\delta(\F)$ and $\delta(\F)$ is $\F$-closed by Lemma~\ref{L:Regular10.4}, the restriction $\L^s|_{\delta(\F)}$ is defined. Clearly, $(\L^s|_{\delta(\F)},\delta(\F),S)$ is a regular locality over $\F$. Notice that either of the properties (a) or (c) implies that the definition of $\E_1\E_2\cdots\E_k$ does not depend on the choice of $(\L^\delta,\delta(\F),S)$. Thus, we may assume $\L^\delta=\L^s|_{\delta(\F)}$.

\smallskip

We will consider now the bijections $\Phi_{\L^s,\L}$ and $\Phi_{\L^s,\L^\delta}$ given by Theorem~\ref{T:VaryObjects}(b) and the bijections $\Psi_\L$ and $\Psi_{\L^s}$ given by Theorem~\ref{mainStepIV}. By part (c) of the latter theorem, we have
\[\Psi_{\L^s}=\Psi_\L\circ \Phi_{\L^s,\L}=\Psi_{\L^\delta}\circ \Phi_{\L^s,\L^\delta}.\]
Set $\N_i:=\Psi_\L^{-1}(\E_i)$ and $\N_i^s:=\Phi_{\L^s,\L}^{-1}(\N_i)$ for $i=1,2,\dots,k$. Put furthermore
\[\N:=\N_1\N_2\cdots\N_k\mbox{ and }\N^s:=\N_1^s\N_2^s\cdots\N_k^s.\]
By definition of $\N_i^s$, we have $\N_i^s\cap\L=\N_i$ for $i=1,2,\dots,k$. Hence, it follows from Corollary~\ref{C:ProductRestrictions} that $\N^s\cap\L=\N$ and thus 
\[\Phi_{\L^s,\L}(\N^s)=\N.\]
Observe also that
\[\N_i^s=\Phi_{\L^s,\L}^{-1}(\N_i)=\Phi_{\L^s,\L}^{-1}(\Psi_\L^{-1}(\E_i))=\Psi_{\L^s}^{-1}(\E_i)=\Phi_{\L^s,\L^\delta}^{-1}(\Psi_{\L^\delta}^{-1}(\E_i))=
 \Phi_{\L^s,\L^\delta}^{-1}(\N_i^\delta).
\]
Hence, $\N_i^s\cap\L^\delta=\N_i^\delta$. Thus, again by Corollary~\ref{C:ProductRestrictions}, $\N^s\cap \L^\delta=\N^\delta$, i.e. $\Phi_{\L^s,\L^\delta}(\N^s)=\N^\delta$. We obtain now
\[\Psi_\L(\N)=\Psi_{\L^s}(\Phi_{\L^s,\L}^{-1}(\N))=\Psi_{\L^s}(\N^s)=\Psi_{\L^s}(\Phi_{\L^s,\L^\delta}^{-1}(\N^\delta))=\Psi_{\L^\delta}(\N^\delta)=\E_1\E_2\cdots\E_k,\]
where the last equality uses Theorem~\ref{mainStepIV}(b). This shows (f) and completes thus the proof.
\end{proof}

We caution the reader that, while $\E_1\E_2\cdots\E_k$ is the smallest normal subsystem of $\F$ in which $\E_1,\E_2,\dots,\E_k$ are normal, there may exist a smaller normal subsystem containing $\E_1,\E_2,\dots,\E_k$. This is illustrated by the following example.

\begin{ex}
Fix the notation as in Example~\ref{E:1}. Set moreover $\E_i=\F_{T_i}(G_i)$ for $i=1,2$. Notice that $G_i\unlhd G$ and thus $\E_i\unlhd\F$ for $i=1,2$. Moreover, by Theorem~\ref{T:ProductCorollaryDetails}(f), $\E_1\E_2=\F_S(G_1G_2)=\F_S(G)=\F$. Observe that $\E_1$ and $\E_2$ are contained in $\F_S(N)$ and, as $N\unlhd G$, we have $\F_S(N)\unlhd \F$. However, $\F_S(N)$ is properly contained in $\F=\E_1\E_2$. 
\end{ex}

\bibliographystyle{amsalpha}
\bibliography{repcoh}

\end{document}